\documentclass[aap,MSNbibl,seceqn,dvips]{arximspdf}

\doi{10.1214/13-AAP939} %kopijuoti is PTS
\volume{24}
\issue{3}
\pubyear{2014}
\firstpage{935}
\lastpage{1001}

\makeatletter

\newtheorem{theorem}{Theorem}[section]
\newtheorem{lemma}[theorem]{Lemma}
\newtheorem{corollary}[theorem]{Corollary}
\newproclaim{definition}[theorem]{Definition}
\newproclaim{remark}[theorem]{Remark}
\newcommand{\widebar}{\overline}
\makeatother

\begin{document}
\begin{frontmatter}

\title{Universality of covariance matrices}
\runtitle{Universality of covariance matrices}

\begin{aug}
\author[A]{\fnms{Natesh S.} \snm{Pillai}\corref{}\ead[label=e1]{pillai@stat.harvard.edu}\thanksref{t1}}
\and
\author[B]{\fnms{Jun} \snm{Yin}\ead[label=e2]{jyin@math.wisc.edu}\thanksref{t2}}
\runauthor{N. S. Pillai and J. Yin}
\affiliation{Harvard University and
University of Wisconsin--Madison}
\address[A]{Department of Statistics\\
Harvard University\\
1 Oxford Street\\
Cambridge, Massachusetts 02138\\
USA\\
\printead{e1}} %adresu isvedimo komanda gale!
\address[B]{Department of Mathematics\\
University of Wisconsin--Madison\\
480 Lincoln Dr.\\
Madison, Wisconsin 53706\\
USA\\
\printead{e2}}
\end{aug}
\thankstext{t1}{Supported by the NSF Grant DMS-11-07070.}
\thankstext{t2}{Supported by the NSF Grant DMS-10-01655.}

% HISTORY:
\received{\smonth{9} \syear{2012}} \revised{\smonth{4} \syear{2013}}

% ABSTRACT
%
\begin{abstract}
In this paper we prove the universality of covariance matrices of the
form $H_{N \times N} = {X}^\dagger X$ where $X$ is an ${M \times N}$
rectangular matrix with independent real valued entries $x_{ij}$
satisfying $\mathbb{E}x_{ij} = 0$ and $\mathbb{E}x^2_{ij} = {1 \over
M}$, $N$, $M\to \infty$. Furthermore it is assumed that these entries
have sub-exponential tails or sufficiently high number of moments. We
will study the asymptotics in the regime $N/M = d_N \in(0,\infty),
\lim_{N\to\infty}d_N \neq0, \infty$. Our main result is the edge
universality of the sample covariance matrix at \emph{both} edges of
the spectrum. In the case $\lim_{N \to\infty} d_N=1$, we only focus on
the largest eigenvalue. Our proof is based on a~novel version of the
Green function comparison theorem for data matrices with dependent
entries. En route to proving edge universality, we establish that the
Stieltjes transform of the empirical eigenvalue distribution of $H$ is
given by the Marcenko--Pastur law uniformly up to the edges of the
spectrum with an error of order $ (N \eta)^{-1}$ where $\eta$ is the
imaginary part of the spectral parameter in the Stieltjes transform.
Combining these results with existing techniques we also show bulk
universality of covariance matrices. All our results hold for both real
and complex valued entries.
\end{abstract}

% KEYWORDS
% Pirmas kwd is didziosios raides
%

\end{frontmatter}

\setcounter{footnote}{2}
%s1 #&#
\section{Introduction}\label{sec1}
In this paper we prove the universality of covariance matrices. Let $X
= (x_{ij}) $ be an $M \times N$ %($N\leq M$)
data matrix with independent centered real valued entries with variance
$ M^{-1} $,%
%
%e1.1 #&#
\begin{equation}
\label{eqnXmat} x_{ij} = M^{-1/2} q_{ij},\qquad
\mathbb{E} q_{ij} = 0,\qquad\mathbb{E}q_{ij}^2 =1.
\end{equation}
Furthermore, the entries $q_{ij}$ have a~sub-exponential decay, that
is, there exists a~constant $\vartheta>0$ such that for $u>1$,
%
%
%e1.2 #&#
\begin{equation}
\label{eqnXmatexpbd} \mathbb{P}\bigl(|q_{ij}| > u\bigr) \leq\vartheta^{-1}
\exp \bigl(-u^\vartheta \bigr).
\end{equation}
The covariance matrix corresponding to data matrix $X$ is given by $H =
{X}^\dagger X$. We will be working in the regime
\[
d=d_N=N/M, \qquad\lim_{ N\to\infty} d \neq{0, \infty}.
\]
Thus without loss of generality, we will assume henceforth that for
some small constant $\theta$, for all $N \in\mathbb{N}$,
\[
{ \theta< d_N<\theta^{-1}. } %\quad\mathrm{and} \quad\theta< |d_N-1|
\]
All our constants may depend on $\theta$ and $\vartheta$, but we will
not denote this dependence. In this paper we focus on the case where
the matrix $X$ has real valued entries which is a~natural assumption
for applications in statistics, economics, etc. However all of the
results in this paper also hold for complex valued entries with the
moment condition (\ref{eqnXmat}) replaced with its complex valued
analogue,%
%
%e1.3 #&#
\begin{equation}
\label{eqnXmatcomp} x_{ij} = M^{-1/2} q_{ij}, \qquad
\mathbb{E} q_{ij} = 0, \qquad\mathbb{E}q_{ij}^2
= 0, \qquad\mathbb{E} |q_{ij}|^2 = 1.
\end{equation}
Furthermore, in some technical results in the present work, the
independence of matrix entries are weakened (see Theorem
\ref{thmlargedev}), which are the key inputs of \cite{BPZ} and
\cite{PillYin1102}.

Covariance matrices are fundamental objects in modern multivariate
statistics where the advance of technology has led to high-dimensional
data. They have manifold applications in various applied fields; see
\cite{Dieng06,John01,John07,John08} for an extensive account on
statistical applications, \cite{Hard08,Onat09} for applications in
economics and \cite{Pattetal06} in population genetics, to name a~few.
%To quote Johnstone \cite{John08},``It is a~striking feature of the
%classical theory of multivariate statistical analysis that most of the
%standard technique--principal components, canonical correlations,
%multivariate analysis of variance (MANOVA), discriminant analysis and
%so forth are founded on the eigen-analysis of covariance matrices."
In the regime we study in this paper where $N,M$ are proportional to
each other, the exact asymptotic distribution of the eigenvalues is not
known, except for some cases under specific assumptions on the
distributions of the entries of the covariance matrix, for example,
when the entries are Gaussian. In this context, akin to the central
limit theorem, the phenomenon of universality helps us to obtain the
asymptotic distribution of the eigenvalues without having restrictive
assumptions on the distribution on the entries. Borrowing a~physical
analogy, as observed by Wigner, the eigenvalue gap distribution for
a~large complicated system is universal in the sense that it depends
only on the symmetry class of the physical system, but not on other
detailed structures.

A fundamental example is the well-studied Wishart matrix (the
covariance matrix obtained from a~data matrix $X$ consisting of i.i.d.
centered Gaussian random variables) for which one has closed form
expressions for many objects of interest including the joint
distribution of the eigenvalues.
%Furthermore the empirical spectrum of the Wishart matrix converges to
%the Marcenko--Pastur law. {\mathbf I don't see the second sentence is
%useful here, what is your real aim?}
%Besides the eigenvalue gap distribution, similar predictions
% hold also for short distance correlation functions of the
%eigenvalues. Since the gap distribution
%can be expressed in terms
%of correlation functions, mathematical analysis is usually performed
%on correlation functions.
%{F}rom now on, we refer to {\it universality} for the fact that the
%short distance behavior of
%the eigenvalue correlation functions of various covariance matrix
%ensembles.
%Among these are the same as those of the one corresponding to the
%Gaussian entries, \textit{i.e.}, the Wishart matrix.
In this paper we prove the universality of covariance matrices (both at
the bulk and at the edges) under the assumption that entries of the
corresponding data matrix are independent, have mean~$0$, variance $1$
and have a~sub-exponential tail decay. This implies that,
asymptotically, the distribution of the local statistics of eigenvalues
of the covariance matrices of the above kind are identical to those of
the Wishart matrix.\vadjust{\goodbreak}

Over the past two decades, great progress has been made in proving the
universality properties of i.i.d. matrix elements (\textit{standard
Wigner ensembles}). The most general results to date for the
universality of Wigner ensembles are obtained in Theorems 7.3 and 7.4
of \cite{EKYY12}, in which bulk (edge) universality is proved for
Wigner matrices under the assumption that entries have a~uniformly
bounded $4+\varepsilon$ ($12+\varepsilon$) moment for some
$\varepsilon> 0$, and then recently improved further by \cite{EYcom}\vadjust{\goodbreak}
and \cite{LY1}. The key ideas for the universality of Wigner ensembles
were developed through several important steps in
\cite{ESY4,EPRSY,EYYBulkuni,EYYgenwig,EYYrigid}. The ideas we use in
this paper are also adapted from the above cited papers. There are also
related results in \mbox{\cite{TV081,TV082}}. However, the results
regarding universality of \emph{local statistics} for covariance
matrices have been obtained only recently, which we survey below.

%s1.1 #&#
\subsection{Review of previous work}
First we review previous results for extreme eigenvalues. In
\cite{BYin93,BSYin,YinB}, the authors showed the almost sure
convergence of extreme eigenvalues. In \cite{GotzTikh04}, the authors
derived the rate of convergence of the spectrum to the
Marchenko--Pastur law. In \cite{Sosh99}, Soshnikov showed that for
$d_N=1-O(N^{-1/3})$, if $q_{ij}$ in (\ref{eqnXmat}) have a~symmetric
distribution and Gaussian decay, then the largest eigenvalues
(appropriately rescaled) converge to the Tracy--Widom distribution.
This condition on $d_N$ was replaced with $\lim_{N
\rightarrow\infty}{d_N }\in(0,\infty)$ by P\'ech\'e \cite{Pech07}.
Using similar assumptions as in \cite{Sosh99} and \cite{Pech07},
Feldheim and Sodin \cite{Sod1} showed that the smallest eigenvalues
(appropriately rescaled) converge to the Tracy--Widom distribution for
$\lim_{N \rightarrow\infty} {d_N}\neq1$. More recently, for $\lim_{N
\rightarrow\infty} {d_N}\neq1$, Wang \cite{Wang11} proved the
Tracy--Widom law for the limiting distribution of the extreme
eigenvalues under the assumption that $q_{ij}$ in (\ref{eqnXmat}) have
vanishing third moment and sufficiently high number of moments. For
``square'' matrices, that is, when $N=M$ and thus $d_N = 1$, Tao and Vu
\cite{TV10} proved the universality of the smallest eigenvalues
assuming the matrix entries have sufficiently high number of moments.
The limiting distribution of the smallest eigenvalue for square
matrices with standard Gaussian entries were computed by Edelman
\cite{edelman}. In our main result below, \textit{we show universality
of eigenvalues for ``rectangular'' data matrices at both edges of the
spectrum, assuming only} (\ref{eqnXmat}) \textit{and}
(\ref{eqnXmatexpbd}).

Now we review results for the local statistics of the eigenvalues in
the bulk of the spectrum. It was widely believed until recently that
the distribution of the distance between adjacent eigenvalues is
independent of the distribution of $q_{ij}$ in (\ref{eqnXmat}).
In~\cite{BenAPech05} Arous and P\'ech\'e showed this bulk universality
when $d_N=1+O(N^{-5/48})$. Tao and Vu \cite{TaoVu09} proved that the
asymptotic distribution for local statistics at the bulk corresponding
to two covariance matrices are identical, if the entries in these two
matrices have identical first four moments. On the other hand, in
\cite{Pech09} and~\cite{ESYY}, P\'ech\'e, Erd\'os, Schlein, Yau and the
second author of this paper showed this bulk universality under some
regularity conditions and decay assumptions on the distribution of
the\vadjust{\goodbreak}
matrix entries. \emph{We also show bulk universality but under weaker
assumptions than those in \cite{Pech09} and \cite{ESYY}}; see Remark
\ref{remremcompbul} for more details.

%s1.2 #&#
\subsection{Our key results}
Let $X^{\mathbf{v}} = [x^{\mathbf{v}} _{ij}]$ with independent entries
satisfying
(\ref{eqnXmat}) and (\ref{eqnXmatexpbd}), and let%
\[
\lambda^{\mathbf{v}}_1\geq\lambda^{\mathbf{v}}_2
\cdots\lambda^{\mathbf{v}}_{\min\{M,N\}}\geq0
\]
denote the nontrivial singular values of the data matrix $X^{\mathbf
{v}}$. Let $\mathbb{P}^{\mathbf{v}}$ denote the probability measure
according to which the entries of $X^{\mathbf{v}}$ are distributed.
Let~$X^{\mathbf{w}}$, $\{\lambda^{\mathbf{w}}_k\}_{k \leq\min\{M,N\}}$
and $\mathbb{P}^ {\mathbf{w}}$ be defined analogously. The following is
our main result:

%
%
%th1.1 #&#
\begin{theorem}[(Universality of extreme eigenvalues)] \label{twthm}
For $\lim_{N \rightarrow\infty} d_N \in(0, \infty)$, there is an
$\varepsilon>0$ and $\delta>0$ such that for any real number $s$ (which
may depend on $N$),
%
%
%e1.4 #&#
\begin{eqnarray}\label{tw}
&& \mathbb{P}^\mathbf{v} \bigl( N^{2/3} \bigl(
\lambda^{\mathbf{v}}_1 -\lambda_+ \bigr) \le s- N^{-\varepsilon}
\bigr)- N^{-\delta}\nonumber
\\
&&\qquad  \le \mathbb{P}^\mathbf{w} \bigl(
N^{2/3} \bigl( \lambda^{\mathbf{w}}_1 -\lambda_+ \bigr)
\le s \bigr)
\\
&&\qquad \le \mathbb{P}^\mathbf{v} \bigl( N^{2/3} \bigl( \lambda
^{\mathbf{v}}_1 -\lambda_+ \bigr) \le s+ N^{-\varepsilon} \bigr)+
N^{-\delta}\nonumber
\end{eqnarray}
for $N\ge N_0$ sufficiently large, where $N_0$ is independent of $s$.
An analogous result holds for the smallest eigenvalues
$\lambda^{\mathbf{v},\mathbf{w}}_{\min\{M,N\}}$, when $\lim_{N\to
\infty}d_N \in(0,\infty) \setminus\{
1\}$. %, \lim_{N\to\infty}d_N \neq1$.
\end{theorem}

In \cite{Sosh99,Pech07} and \cite{Sod1}, Soshnikov, P\'ech\'e, Feldheim
and Sodin proved that for the covariance matrices whose entries have
a~symmetric probability density function (which includes the Wishart
matrix), the largest and smallest $k$ eigenvalues after appropriate
centering and rescaling converge in distribution to the Tracy--Widom
law.\footnote{Here we use the term Tracy--Widom law as in
\cite{Sosh99}.} We have the following immediate corollary of Theorem
\ref{twthm}:
%
%
%co1.2 #&#
\begin{corollary} Let $X$ with independent entries satisfying (\ref
{eqnXmat}) and (\ref{eqnXmatexpbd}), and let $\lim_{N \rightarrow
\infty} d_N \in(0, \infty)$. For any fixed $k>0$, we have
\begin{eqnarray*}
&& \biggl(\frac{M\lambda_1 - (\sqrt{N} + \sqrt{M})^2} {(\sqrt{N} +
\sqrt{M})((1/\sqrt N) + (1/\sqrt M))^{1/3}}, \ldots,
\\
&&\hspace*{29pt} \frac{M\lambda_k- (\sqrt{N} + \sqrt{M})^2} {(\sqrt{N} +
\sqrt{M})((1/\sqrt N) + (1/\sqrt M))^{1/3}} \biggr) \longrightarrow
\mathrm{TW}_1,
\end{eqnarray*}
where $\mathrm{TW}_1$ denotes the Tracy--Widom distribution. An
analogous statement holds for the smallest eigenvalues, when $\lim_{N\to\infty}d_N \in(0,\infty) \setminus\{1\}$. % \lim_{N\to
\end{corollary}

%
%
%re1.3 #&#
\begin{remark}\label{remark1}
Clearly, our result covers the case where the matrix entries have
Gaussian divisible distribution (see \cite{Wang11}, Section~2) and the
case where the support of the distribution of the matrix entries
consists of only two points. Using these two cases and the results of
\cite{Wang11}, the sub-exponential-decay assumption in Corollary
\ref{eqnXmatexpbd} can be replaced with the existence of sufficiently
high number of moments. For details, see the discussion below the
Theorem 2.2 of \cite{Wang11}. However we believe that all of our
results can be proved under a~uniform bound on $p$th moments of the
matrix elements (say $p=4$ or $5$), using the methods in \cite{EKYY12}
and~\cite{LY1}; we will pursue this elsewhere.
\end{remark}

%
%
%re1.4 #&#
\begin{remark}
Theorem \ref{twthm} can be extended to obtain universality of finite
correlation functions of extreme eigenvalues. For example, we have the
following extension of (\ref{tw}): for any fixed $k$,
%
%
%e1.5 #&#
\begin{eqnarray}\label{twa}
\qquad && \mathbb{P}^\mathbf{v} \bigl( N^{2/3} \bigl(
\lambda^{\mathbf{v}}_1 -\lambda_+ \bigr) \le s_1-
N^{-\varepsilon}, \ldots, N^{2/3} \bigl( \lambda^{\mathbf{v}}_{k}
-\lambda_+ \bigr) \le s_{k}- N^{-\varepsilon} \bigr)-
N^{-\delta}\nonumber
\\
&&\qquad\le\mathbb{P}^\mathbf{w} \bigl( N^{2/3} \bigl( \lambda
^{\mathbf{w}}_1 -\lambda_+ \bigr) \le s_1, \ldots,
N^{2/3} \bigl( \lambda^{\mathbf{w}}_{k} -\lambda_+ \bigr)
\le s_{k} \bigr)
\nonumber\\[-8pt]\\[-8pt]
&&\qquad\le\mathbb{P}^\mathbf{v} \bigl( N^{2/3} \bigl( \lambda
^{\mathbf{v}}_1 -\lambda_+ \bigr) \le s_1+
N^{-\varepsilon}, \ldots,\nonumber
\\
&&\hspace*{70pt}  N^{2/3} \bigl( \lambda^{\mathbf{v}}_{k}
-\lambda_+ \bigr) \le s_{k}+ N^{-\varepsilon} \bigr)+
N^{-\delta}\nonumber
\end{eqnarray}
for all sufficiently large $N$. The proof of (\ref{twa}) is similar to
that of (\ref{tw}), and we will not provide details, except stating the
general form of the Green function comparison theorem (Theorem
\ref{GFCT2}) needed in this case. We remark that edge universality is
usually formulated in terms of joint distributions of edge eigenvalues
in the form (\ref{twa}) with fixed parameters $s_1, s_2,\ldots,$ etc.
Our result holds uniformly in these parameters, that is, they may
depend on $N$. However, the interesting regime is $|s_j|\le O((\log
N)^{\log\log N})$; otherwise, the rigidity estimate obtained in
(\ref{resrig}) will give stronger control than (\ref{twa}).
\end{remark}

The first step toward proving Theorem \ref{twthm} is to obtain a~strong
\textit{local Marcenko--Pastur law}, a~precise estimate of the local
eigenvalue density in the optimal scale $N^{-1+o(1)}$. We state and
prove this in Theorem \ref{451}. This theorem is our key technical tool
for proving rigidity of eigenvalues (see Theorem \ref{452}) and
universality. En route to this, we also obtain precise bounds on the
matrix elements of the corresponding Green function. All of our results
regarding the strong Marcenko--Pastur law do not require independence
of the entries of the data matrix, but need only weak dependence as
will be explained in Section~\ref{secMPlaw}. An important technical
ingredient required for the estimates for our strong Marcenko--Pastur
law and the rigidity of eigenvalues is an abstract decoupling lemma
(Lemma \ref{abstractZlemma}) for weakly dependent random variables,
proved in Section~\ref{seczlemma}.

Using the strong Marcenko--Pastur law and the existing results (such as
\cite{EYYgenwig} and Theorem 2.1 in \cite{ESYY}), we also show bulk
universality holds for covariance matrices in almost optimal scale:\vadjust{\goodbreak}
%%$b=N^{-1+\varepsilon}$.

%
%
%th1.5 #&#
\begin{theorem}[(Universality of eigenvalues in bulk)] \label{thmmain}
Let $X^{\mathbf{v}}, X^{\mathbf{w}}$ be as defined before. Assume that
$\lim_{N \rightarrow\infty} d_N \in(0,\infty)\setminus\{1\}$. Let
$E\in[\lambda_-+r,\break  \lambda_+-r] $ with some $ r> 0$. Then for any
$\varepsilon>0$, $N^{-1+\varepsilon} < b < r/2$, any fixed integer
$n\ge1$ and for any compactly supported continuous test function
$O\dvtx\mathbb{R}^n\to\mathbb{R}$, we have
%
%
%e1.6 #&#
\begin{eqnarray}\label{abstrthm2intro}
&& \lim_{N \rightarrow\infty} \int_{E-b}^{E+b}
\frac{\mathrm{d}E'}{2b} \int_{\mathbb{R}^n} O(\alpha_1,\ldots,
\alpha_n) \bigl( p_{\mathbf{v}N}^{(n)} - p_{\mathbf{w}, N}
^{(n)} \bigr)
\nonumber\\[-4pt]\\[-12pt]
&&\hspace*{91pt}{}\times \biggl(E'+\frac{\alpha_1}{N\varrho_c(E)}, \ldots,
E'+\frac{\alpha_n}{
N\varrho_c(E)} \biggr) \prod_i
\frac{\mathrm{d}\alpha_i}{\varrho_c(E)} =0,\hspace*{-22pt}\nonumber
\end{eqnarray}
where $ p_{\mathbf{v}, N}^{(n)}$ and $ p_{\mathbf{w}, N}^{(n)}$ %
are the n-points correlation functions of the eigenvalues of
$(X^\mathbf{v})^\dagger X^\mathbf{v}$ and $(X^\mathbf {w})^\dagger
X^\mathbf{w}$, respectively.
\end{theorem}

%
%
%re1.6 #&#
\begin{remark}
As in Remark \ref{remark1}, using the four moment theorem in
\cite{TaoVu09}, the sub-exponential-decay assumption for the matrix
entries can be replaced with the existence of a~sufficiently high
number of moments.
\end{remark}

%
%
%re1.7 #&#
\begin{remark}\label{remremcompbul}
Compared to the results obtained in \cite{Pech09,ESYY}, our
Theorem~\ref{thmmain} is an improvement on two fronts: \textup{(i)}~in
\cite{Pech09,ESYY}, for (\ref{abstrthm2intro}), the authors required
that
\[
\sum_{i=1}^{M_k}\bigl|\partial_x^i
\log u_0(x)\bigr|\leq C_k\bigl(1+|x|\bigr)^{C_k}
\]
for some $M_k$ and $C_k$, where $u_0$ is the probability density
function of the matrix entries; see formulas (1.3)--(1.5) in
\cite{Pech09} and formula (3.6) in \cite{ESYY}. \textup{(ii)} We show
that the bulk university holds in almost optimal scale:
$b=N^{-1+\varepsilon}$. In the main theorem of \cite{ESYY}, bulk
universality was shown for $b \sim O(1)$.\footnote{For two quantities
$a,b$ we write $a \sim b$ to denote $cb\leq a~\leq C b$ for some $c, C
> 0$.} We also note that in~\cite{Pech09}, the integral in
(\ref{abstrthm2intro}) is not required. On the other hand, the proof in
\cite{Pech09} does not work for covariance matrices with real valued
entries.
\end{remark}

%
%
%re1.8 #&#
\begin{remark}
Our result heavily relies on the Theorem 2.1 of \cite{ESYY}, but we are
able to show universality up to this optimal scale, mainly because of
our stronger results on the strong local Marcenko--Pastur law and the
rigidity result for eigenvalues obtained in Theorems~\ref{451}
and~\ref{452}, respectively.
\end{remark}

%
%
%re1.9 #&#
\begin{remark}
Tao and Vu \cite{TaoVu09} derived bulk universality without the
integral in (\ref{abstrthm2intro}), but they required that the matrix
entries of the two covariance matrices have identical first four
moments.
\end{remark}

%s1.3 #&#
\subsection{Main ideas}
%The above result on edge universality is one of the key results of our
%work.
The approach we take in this paper to prove universality is the one
developed in a~recent series of papers
\cite{EKYY11,EKYY12,ESY4,EPRSY,ESYY,EYYBulkuni,EYYgenwig,EYYrigid};
however, there are some important differences which we highlight below.
Our proof of the above result proceeds via the Green function
comparison theorem as in the case of Wigner matrices; however,
\emph{unlike Wigner matrices}, \textit{the elements within the same column of
a~covariance matrix are not independent}. In order to address this key
difficulty, we introduce new ideas and \emph{establish a~novel version
of the Green function comparison theorem}. In particular, in Theorem
\ref{thmsuffedgeuni} (see Section~\ref{ld-MPlaw}) we give sufficient
criteria for proving edge universality for matrix ensembles of the form
${Y}^\dagger Y$ for a~generic data matrix $Y$ with dependent entries
(e.g., correlation matrices). \emph{This enables us to show the edge
universality for covariance matrices when $\lim_{N\to\infty} d_N \in
(0, \infty)$, under the assumption that the first two moments of the
matrix entries are equal to that of the standard Gaussian}.
%various types of data matrices $Y$ with dependent entries.
Our method is also useful for establishing universality for a~huge
class of matrix ensembles with dependent entries. For example, in
a~recent paper \cite{BPZ}, Bao, Pan and Zhou used our method to show
universality for a~class of correlation matrices. For more general edge
universality results for correlation matrices, see a~later paper
\cite{PillYin1102}, which is also based on our Green function
comparison theorem. As mentioned above, for our strong Marcenko--Pastur
law, we use an abstract decoupling lemma (Lemma \ref{abstractZlemma})
for weakly dependent random variables. This lemma is novel and is
applicable in other settings such as non-Hermitian ensembles
\cite{BYY2012}.

For proving bulk universality of eigenvalues, we follow the general
approach for the universality of Gaussian divisible ensembles
\cite{EYYgenwig,EYYrigid,EKYY11,EKYY12,KY1,KY2} by embedding the
covariance matrix into a~stochastic flow of matrices and so that the
eigenvalues evolve according to a~distinguished coupled system of
stochastic differential equations, called the Dyson Brownian motion
\cite{Dy}. An important idea in the papers mentioned above is to
estimate the time to local equilibrium for the Dyson Brownian motion
with the introduction of a~new stochastic flow, the \textit{local
relaxation flow}, which locally behaves like a~Dyson Brownian motion
but has a~faster decay to global equilibrium. This approach, first
introduced in \mbox{\cite{ESY4,ESYY}}, eliminates entirely the usage of
explicit formulas. We will also follow this route and use the strong
local Marcenko--Pastur law to show that the time for the Dyson Brownian
motion (corresponding to the covariance matrix) to reach local
equilibrium is about $O(N^{-1})$. Once we prove this result, all that
remains to be done is to show that the local statistics at
$t=O(N^{-1})$ coincide with those of the initial matrix, that is,
$t=0$. To achieve this, we again use the Green function comparison
method. Roughly speaking, the Green function comparison method exploits
the fact that the equilibrium time is very ``small'' [$O(N^{-1})$], and
therefore the first few moments of the matrix entries at time $t =
N^{-1}$ will be nearly identical to those at $t = 0$.

%s1.4 #&#
\subsection{Comments on other limiting regimes of $d_N$}
The assumption\break $\lim_{N \rightarrow\infty} d_N\in(0,\infty)$ is
mostly for simplicity, and we believe that with some more effort, most
of our results can be extended to the case $\lim_{N \to\infty} d_N =
\{0, \infty\}$. This will be pursued in our future works.

However, we believe that universality at the soft edge for
$\lim_{N\rightarrow\infty} d_N=1$ will be much harder. There is
a~singularity of the eigenvalue density at $x=0$. More precisely, the
typical distance between adjacent eigenvalues near $x=0$ is
$O(N^{-2})$. % i.e., hard edge.
For studying the smallest eigenvalue one needs to overcome several
obstacles: (1)~The usual moment method which estimates $\mathbb
{E}({X}^\dagger X)^k$ with large $k\in\mathbb{N}$ does not work in
obtaining bounds for the smallest eigenvalue. (2) For the ``square
case'' ($N = M$), in \cite{TV10} the authors proceeded via analyzing
$X^{-1}$ directly; this strategy seems out of reach for the nonsquare
case. (3)~In fact, as in \cite{BYY2012,CMS}, one can prove that the
$m(z)$ does satisfy the local Marchenko--Pastur law in the case
$\lim_{N\rightarrow\infty} d_N=1$ up to the scale $\eta\gg(N
|m_c|)^{-1}$. Note $\eta=(N\Im m_c)^{-1}$ is the scale of individual
eigenvalue. At the soft edge (i.e., for largest eigenvalues), it can be
shown that $\Im m_c\leq|m_c|$, and thus we have a~strong estimate on
$m(z)$ in the scale which is small enough for estimating the
distribution of single eigenvalue. But at the hard edge $\Im
m_c\sim|m_c|$, so our method used for estimating $m_c$ at the soft edge
cannot be directly applied to the hard edge. It is proved in
\cite{BYY2012,CMS} that the density of eigenvalues satisfy the
Marchenko--Pastur law. (Only the case $d_N=1$ is proved in
\cite{BYY2012,CMS}, but the result can be easily extended to the case
$\lim_{N\rightarrow\infty} d_N=1$.) For the distribution of the
smallest eigenvalues, the only universality result we know is in
\cite{TV10}, as mentioned above.

Finally we note that the authors in \cite{LY1} recently showed
a~necessary and sufficient condition on the edge universality of Wigner
matrices. Based on this, we conjecture that for the edge universality
of covariance matrices whose entries are i.i.d., the necessary and
sufficient condition on the distribution of the matrix entries is given
by $\lim_{s\to\infty} s^4\mathbb{P}(|q_{12}|\ge s)=0$.

%s1.5 #&#
\subsection{Organization of the paper} %The rest of the paper is
%organized as follows.
In Section~\ref{secprelim} we set notation and give some basic
definitions. In Section~\ref{secMPlaw} we give statements of the strong
version of the Marcenko--Pastur law, rigidity and delocalization of
eigenvectors. In Sections~\ref{proUE} and \ref{proUB}, we prove,
respectively, the edge and bulk universality results. In
Sections~\mbox{\ref{ld-MPlaw}--\ref{secp45}} we give proofs of the
strong Marcenko--Pastur law and rigidity of eigenvalues. In
Section~\ref{seczlemma}, we state and prove an abstract decoupling
lemma for weakly dependent random variables which is used to prove the
strong Marcenko--Pastur law.

%s2 #&#
\section{Preliminaries} \label{secprelim}
Define
%
%
%e2.1 #&#
\begin{eqnarray}\label{eqnGg}
H &:=& {X}^\dagger X,\qquad G(z):= (H-z)^{-1}=
\bigl({X}^\dagger X-z \bigr)^{-1},
\nonumber\\[-10pt]\\[-10pt]
m(z)&:=&\frac{1}{N}\operatorname{Tr}G(z),\qquad
\mathcal{G}(z):= \bigl( X{X}^\dagger-z \bigr)^{-1}.\nonumber
\end{eqnarray}
Since the nonzero eigenvalues of $X{X}^\dagger$ and ${X}^\dagger X$ are
identical and $ X {X}^\dagger$ has $M-N$ more (or $N-M$ less) zero
eigenvalues,
%
%
%e2.2 #&#
\begin{equation}
\label{yz10} \operatorname{Tr}G(z)-\operatorname{Tr}\mathcal{G}(z)=
\frac{M-N }{z}.
\end{equation}

We will often need to consider minors of $X$ defined below:

%
%
%de2.1 #&#
\begin{definition}[(Minors)]
For $\mathbb T \subset\{1, \ldots, N\}$ we define $X^{(\mathbb T)}$ as
the $(M \times(N - |\mathbb T |))$ minor of $X$ obtained by removing
all columns of $X$ indexed by $i \in\mathbb T$. Note that we keep the
names of
indices of $X$ when defining $X^{(\mathbb T)}$,% For $\mathbb T \subset
\[
\bigl(X^{(\mathbb T)} \bigr)_{ij}:= \mathbf1 (j \notin\mathbb T)
X_{ij}.
\]
\end{definition}

The quantities $G^{(\mathbb T)}(z)$, $\mathcal{G}^{(\mathbb T)}(z)$,
$\lambda _\alpha^{(\mathbb T)}$, $\mathbf{u}_\alpha^{(\mathbb T)}$,
$\mathbf {v}_\alpha ^{(\mathbb T)}$, etc. are defined similarly
using~$X^{(\mathbb T)}$. Furthermore, we abbreviate $(i) = (\{i\})$ as
well as $(i \mathbb T) =\break  (\{i\} \cup\mathbb T)$. We also set
%
%
%e2.3 #&#
\begin{equation}
\label{eqnmTdef} m^{(\mathbb T)}(z):= \frac{1}{N} \sum
_{i\notin\mathbb T } G^{(\mathbb
T)}_{ii}(z).
\end{equation}
We denote the $i$th column of $X$ by $\mathbf{x}_i$, which is an
$M\times1$ vector. Recall $\lambda_+, \lambda_{-}$ from
(\ref{eqnlambdapm}). For
$z = E + i \eta$, set%
%
%e2.4 #&#
\begin{equation}
\label{defkappa} \kappa:= \min\bigl(|\lambda_+ - E|, |E - \lambda_{-}|\bigr).
\end{equation}
Throughout the paper we will use the letters $C,C_\zeta, c$ to denote
generic positive constants whose precise value may change from one
occurrence to the next but independent of everything else.

Define the Green function of ${X}^\dagger X$ by
%
%
%e2.5 #&#
\begin{equation}
\label{eqngreen} G_{ij}(z) = \biggl(\frac{1}{{X}^\dagger
X-z}
\biggr)_{ij}, \qquad z=E+i\eta, \qquad E\in\mathbb{R}, \qquad\eta>0.
\end{equation}
The Stieltjes transform of the empirical eigenvalue distribution of
${X}^\dagger X $ is
given by%
%
%e2.6 #&#
\begin{equation}
m(z): = \frac{1}{N} \sum_j
G_{jj}(z) = \frac{1}{N} \operatorname{Tr}\frac{1}{{X}^\dagger X-z}.
\label{mNdef}
\end{equation}
We will be working in the regime
%
%
%e2.7 #&#
\begin{equation}
\label{eqndefd} d:=d_N:=N/M, \qquad\lim_{ N\to\infty
} d
\neq{0,\infty}.
\end{equation}
For our results at the hard-edge (smallest eigenvalues) and for bulk
universality results, we will further require that $\lim_{N \to\infty}
d_N \neq1$. Define
%
%
%e2.8 #&#
\begin{equation}
\label{eqnlambdapm} \lambda_\pm:= (1\pm\sqrt{d} )^2.
\end{equation}
The Marchenko--Pastur law \cite{MP} (henceforth abbreviated by MP) is
given by
%
%
%e2.9 #&#
\begin{equation}
\label{defrhow} \varrho_c(x)=\frac{1}{2\pi
d}\sqrt{
\frac{ [(\lambda_+-x)(x-\lambda_-) ]_+}{x^2}}.
\end{equation}
We define $m_c(z)$, $z\in\mathbb{C}$, as the Stieltjes transform of
$\varrho_c$,
that is,%
%
%e2.10 #&#
\begin{equation}
\label{defmW} m_c(z)=\int_{\mathbb{R}}
\frac{\varrho_c(x)}{(x-z)}\,\mathrm{d}x.
\end{equation}
The function $m_c$ depends on $d$ and has the closed form expression
%
%
%e2.11 #&#
\begin{equation}
\label{mwz=} m_c(z)=\frac{1-d-z +i \sqrt{(z-\lambda_-)(\lambda
_+-z)}}{2 d z},
\end{equation}
where $\sqrt{\mbox{ }}$ denotes the square root on the complex plane
whose branch cut is the negative real line. One can check that $m_c(z)$
is the unique solution of the equation
\[
m_c(z) + \frac{1}{z - (1-d) + z d m_c(z)} = 0
\]
with $\Im m_c(z) > 0$ when $\Im{z} > 0$. %Recall $\lambda_\pm=(1\pm
%d^{1/2})^2$ from \eqref{eqnlambdapm}.
%Explicit calculation shows $m_W(z)$ is the Stieltjes transform of the
%Marchenko-Pastur density given in \eqref{defrhow}.
Define the \textit{normalized empirical counting function} by
%
%
%e2.12 #&#
\begin{equation}
{\mathfrak n}(E):= \frac{1}{N}\# \{ \lambda_j\geq E\}.
\label{deffn}
\end{equation}
Let%
%
%e2.13 #&#
\begin{equation}
n_c(E): = \int_{E}^{\infty}
\varrho_c(x)\,\mathrm{d}x \label{nsc}
\end{equation}
so that $1 - n_c(\cdot)$ is the distribution function of the MP law.

By the singular value decomposition of $X$, there exist orthonormal
bases $\{{\mathbf u}_1, {\mathbf u}_2,\ldots, {\mathbf u}_{M }\}
\subset\mathbb{C}^{M}$ and $\{{\mathbf v}_1, \ldots, {\mathbf v}_{N }\}
\subset\mathbb{R}^N$
such that %\sidenote{Since the entries of $X$ are real, why $u_i, v_i
%
%
%e2.14 #&#
\begin{equation}
\label{eqnSVD} X = \sum_{\alpha= 1}^{M} \sqrt{
\lambda_\alpha} {\mathbf u_\alpha} \mathbf{v}_\alpha^\dagger
= \sum_{\alpha= 1}^{N} \sqrt{
\lambda_\alpha} {\mathbf u_\alpha} \mathbf{v}_\alpha^\dagger,
\end{equation}
where $\lambda_1\geq\lambda_2\cdots\lambda_{\max\{M,N\}}\geq0$,
$\lambda_\alpha=0$ for $ \min\{N, M\}+1\leq\alpha\leq\max\{N,M\} $, and
we let $\mathbf{v}_\alpha=0 $ if $\alpha>N$ and $\mathbf{u}_\alpha =0 $
for $\alpha>M$. We also define the classical location of the
eigenvalues with
$\varrho_c$ as follows:%
%
%e2.15 #&#
\begin{equation}
\label{defgj} \int_{\gamma_j}^{\lambda_+}
\varrho_c(x) \,\mathrm{d}x =\int_{\gamma_j}^{+\infty}
\varrho_c(x) \,\mathrm{d}x =j/N.
\end{equation}
Define the parameter%
%
%e2.16 #&#
\begin{equation}
\label{eqnvarphi} \varphi:= (\log N)^{\log\log N}.
\end{equation}
For $\zeta\geq 0$, define the set
%
%
%e2.17 #&#
\begin{equation}
\label{defS} \qquad\mathbf{S}(\zeta):= \bigl\{z \in\mathbb{C}\dvtx
{\bolds{1}_{d>1}}(\lambda_-/5) \leq E \leq5\lambda_+, \varphi^\zeta
N^{-1} \leq\eta \leq10(1+d) \bigr\}.
\end{equation}
Note that {$m_c\sim1$} in $\mathbf{S}(0)$. Also the cases ${ d}>1$ and
${ d}<1$ are not symmetric in the above definition. Actually the proof
of universality in the case ${d}>1$ is much harder, since it has many
zero eigenvalues. This issue can be easily avoided if matrix entries
are independent since ${X}^\dagger X$ and $X{X}^\dagger$ have the same
nonzero eigenvalues. Since, in the strong Marcenko--Pastur law
established next section, \emph{we do not assume independence unlike
previous works}, the proof is more difficult.
%As used in \cite{}, we define high probability events as follows.
%
%
%de2.2 #&#
\begin{definition}[(High probability events)]\label{defhp}
Let $\zeta> 0$.
We say that an %$N$-dependent%
event $\Omega$ holds with \emph{$\zeta$-high probability} if there
exists a~constant $C > 0$ such that
%
%
%e2.18 #&#
\begin{equation}
\label{highprob} \mathbb{P} \bigl(\Omega^c \bigr) \leq N^C
\exp \bigl(-\varphi^\zeta \bigr)
\end{equation}
for large enough $N$.
\end{definition}
The next lemma collects the main identities of the resolvent matrix
elements $G_{ij}^{(\mathbb T) }$ and $\mathcal{G}^{(\mathbb
T)}_{ij}(z)$.

%
%
%le2.3 #&#
\begin{lemma}[(Resolvent identities)]\label{lemmaresolventid}
%
%
%e2.19 #&#
%e2.20 #&#
%e2.21 #&#
\begin{eqnarray}
\qquad\qquad G_{ii}(z) &=& { 1 \over- z - z \langle{\mathbf x}_i,\mathcal{G}^{(i)}(z)
{\mathbf x}_i\rangle},\quad\mbox{i.e., } \bigl\langle{ \mathbf x}_i,\mathcal{G}^{(i)}(z) {\mathbf
x}_i \bigr\rangle=\frac{-1}{z G_{ii}(z)}-1, \label{eqnGii}
\\
%where $\langle x_i |\mG^{(i)}(z)| x_i\rangle=\sum_{kl}X_{ik}
G_{ij}(z) &=& z G_{ii}(z)
G_{jj}^{(i)} (z) \bigl\langle{\mathbf x}_i,
\mathcal{G} ^{(ij)}(z) {\mathbf x}_j \bigr\rangle, \qquad i
\neq j, \label{eqnGij}
\\
G_{ij}(z) &=& G_{ij}^{(k)}(z) +
\frac{G_{ik} (z)
G_{kj}(z)}{G_{kk}(z)},\qquad i,j \neq k\label{eqnGijGijk}.
\end{eqnarray}
%
%For $1\leq i\leq N$,
%=\mG(z) + {\langle\mG(z) x_i, x_i |\mG\rangle\over1- \langle{
%,\qquad
% \mG
%=\mG^{(i)}-\frac{\mG^{(i)} |x_i\rangle\langle x_i |\mG^{(i)}}{1+
% Furthermore,
% \langle x_i |\mG(z)| x_i\rangle=1+z G_{ii}
%and
% \begin{eqnarray}\label{iGj}
\end{lemma}
%
%Note: It also holds for any $X^{(T)}$.

\begin{pf} The proof is straightforward and needs only elementary
linear algebra; see
%follows from Lemma \ref{lemtbtinvmat} and Lemma \ref{lemxxt} above
%and
Lemma 3.2 of \cite{EYYrigid}.
\end{pf}

%s3 #&#
\section{Strong Marchenko--Pastur law} \label{secMPlaw}
Our goal in this section is to estimate the following quantities:
%
%
%e3.1 #&#
\begin{equation}
\label{defLambda} \qquad \Lambda_d: = \max_k
|G_{kk}-m_c|, \qquad\Lambda_o:=\max
_{k\ne\ell} |G_{k\ell}|, \qquad\Lambda:=|m-m_c|,
\end{equation}
where the subscripts refer to ``diagonal'' and ``off-diagonal'' matrix
elements. All these quantities depend on the spectral parameter $z$ and
on $N$, but for simplicity we suppress this in the notation.

For simplicity of exposition, henceforth in this section we focus on
the $\lim_{N}d_N\neq1$ case. The proof of the distribution of the
largest eigenvalue in the case $\lim_{N \to\infty}d_N=1$ is a~simple
extension of our proof of the case $\lim_{N\to\infty}d_N \in
(0,\infty)\setminus\{1\}$. Therefore, we will give only a~brief
discussion at the end of Section~\ref{proUE}.

The following is the main result of this section and our main technical
tool for establishing universality. It holds for both real and complex
valued entries. The proof of the results in this section is given in
Sections~\ref{ld-MPlaw}--\ref{secp45}.

%
%
%th3.1 #&#
\begin{theorem}[(Strong local Marchenko--Pastur law)] \label{451}
Let $X = [x_{ij}]$ with entries $x_{ij}$ satisfying (\ref{eqnXmat}) and
(\ref{eqnXmatexpbd}), and let $\lim_{N \to\infty} d_N \in (0,\infty)
\setminus\{1\} $. For any $\zeta>0$ there exists a~constant $C_\zeta$
such that the following events hold with $\zeta$-high probability:
\begin{longlist}[(iii)]
\item[(i)] The Stieltjes transform of the empirical eigenvalue
    distribution of $H $ satisfies
%for all $A$ with $ A_0/2 \le A \le A_0$
%
%
%e3.2 #&#
\begin{equation}
\label{Lambdafinal} \bigcap_{z \in{ \mathbf{S}}(C_\zeta)} \biggl\{ \Lambda(z)
\leq\varphi^{C_\zeta} \frac{1}{N\eta} \biggr\}.
\end{equation}

\item[(ii)] The individual matrix elements of
the Green function satisfy%
%
%e3.3 #&#
\begin{equation}
\label{Lambdaofinal} \bigcap_{z \in{ \mathbf{S}}(C_\zeta)} \biggl\{
\Lambda_o(z) +\Lambda_d \leq\varphi^{C_\zeta}
\biggl(\sqrt{\frac{\Im m_c(z) }{N\eta}}+ \frac{1}{N\eta} \biggr) \biggr\}.
\end{equation}

\item[(iii)] The smallest nonzero and largest eigenvalues of
$X^\dagger X$ satisfy%
%
%e3.4 #&#
\begin{equation}
\label{443} \lambda_- -N^{-2/3}\varphi^{C_\zeta}\leq\min
_{j\leq\min\{M, N\}} \lambda_j \leq\max_{j }
\lambda_j \le\lambda_++N^{-2/3}\varphi^{C_\zeta}.
\end{equation}

\item[(iv)] Delocalization of the eigenvectors of $X^\dagger X$,
%
%
%e3.5 #&#
\begin{equation}
\label{4444} \max_{\alpha\dvtx \lambda_\alpha\neq0}\|\mathbf {v}_\alpha\|
_\infty\leq\varphi^{C_\zeta}N^{-1/2}.
\end{equation}
\end{longlist}
\end{theorem}

%this paper we prove the following
%local Marcenko--Pastur law that provides essentially the optimal
%estimate
%{\it uniformly} in $E = \Re z$.

%We will estimate not only
%the deviation of $m(z)$ from $\mw(z)$, but also
%the deviation of each diagonal matrix element of the resolvent,
%$G_{kk}(z)$,
%from $\mw(z)$. Moreover, we show that the off-diagonal
%elements of the resolvent are small.
%
%Let
%$$
% v_k:= G_{kk}- \mw, \qquad m:=\frac{1}{N}\sum_{k=1}^N G_{kk}, \qquad
% \sum_{k=1}^N v_k = m-\mw.
%$$
%
%
%re3.2 #&#
\begin{remark}
To our knowledge, there are two weaker versions of the above theorem
previously established in \cite{GotzTikh04,ESYY}. In \cite{ESYY} the
error term obtained in (\ref{Lambdafinal}) is of order
$(N\eta)^{-1/2}/(\kappa+(N\eta)^{-1/2})^{1/2}$ [see (\ref{defkappa})] and
similar comments apply for the results in \cite{GotzTikh04}, whereas we
need the above stronger estimates for our work, especially for edge
universality.
\end{remark}

The main theorem above is then used to obtain the following results:
%
%
%th3.3 #&#
\begin{theorem}[(Rigidity of the eigenvalues of covariance matrix)]\label{452}
Recall $\gamma_j$ in~(\ref{defgj}). Let $X = [x_{ij}]$ with entries
$x_{ij}$ satisfying (\ref{eqnXmat}) and (\ref{eqnXmatexpbd}) and
$\lim_{N \to\infty} d_N \in(0,\infty) \setminus\{1\} $. For any $1\leq
j\leq N$, let
\[
\widetilde j=\min \bigl\{\min\{N, M\} + 1 -j, j \bigr\}.
\]

For any $\zeta>0$ there exists a~constant $C_\zeta$ such that
%
%
%e3.6 #&#
\begin{equation}
\label{resrig} |\lambda_j-\gamma_j|\leq
\varphi^{C_\zeta} N^{-2/3} \widetilde j^{-1/3}
\end{equation}
and
%
%
%e3.7 #&#
\begin{equation}
\label{resrig2} \bigl|\mathfrak n (E)-n_c(E) \bigr|\leq
\varphi^{C_\zeta} N^{-1}
\end{equation}
hold with $\zeta$-high probability for any $1\leq j\leq N$.
\end{theorem}

The above two results are stated under the assumption that the matrix
entries are independent. The independence assumption (of the elements
in each column vector of $X$) required in Theorems \ref{451} and
\ref{452} can be replaced with a~large deviation criteria as will be
explained below.

Let us first recall the following large deviation lemma for independent
random variables; see \cite{EYYBulkuni}, Appendix B for a~proof.
%
%
%le3.4 #&#
\begin{lemma}[(Large deviation lemma)] \label{lemlargdev} Suppose $a_i$
are independent, mean $0$ complex variables, with $\mathbb{E}|a_i|^2 =
\sigma^2$ and have a~sub-exponential decay as in (\ref{eqnXmatexpbd}).
Then there exists a~constant $\rho\equiv\rho(\vartheta) > 1$ such that,
for any $\zeta> 0$ and for any $A_i \in\mathbb{C}$ and $B_{ij}
\in\mathbb{C}$, the bounds
%
%
%e3.8 #&#
%e3.9 #&#
%e3.10 #&#
\begin{eqnarray}
\Biggl|\sum_{i=1}^M a_i
A_i\Biggr| &\leq& ( \log M)^{\rho\zeta\log\log M} \sigma\| A\|, \label{eqnld1}
\\
\Biggl|\sum_{i=1}^M \widebar a_i
B_{ii} a_i - \sum_{i=1}^M
\sigma^2 B_{ii} \Biggr| &\leq& ( \log M)^{\rho\zeta\log\log M}
\sigma^2 \Biggl(\sum_{i=1}^M
|B_{ii}|^2 \Biggr)^{1/2}, \label{eqnld2}
\\
\Biggl|\sum_{i \neq j} \widebar a_i
B_{ij} a_j\Biggr| &\leq& ( \log M)^{\rho\zeta\log
\log M}
\sigma^2 \biggl(\sum_{i \neq j}
|B_{ij}|^2 \biggr)^{1/2} \label{eqnld3}
\end{eqnarray}
hold with $\zeta$-high probability.
\end{lemma}

%
%
%re3.5 #&#
\begin{remark}
When $M \sim N$, equation (\ref{eqnld1}) yields that for any $\zeta>
0$, $|\sum_{i=1}^M a_i A_i| \leq\varphi^{C_\zeta} \sigma\|A\|$ for some
$C_\zeta> 0$ with $\zeta$-high probability. Here $\varphi$ is as
defined in (\ref{eqnvarphi}).
\end{remark}

Next we extend Theorems \ref{451} and \ref{452} by relaxing the
independence assumption.
%
%
%th3.6 #&#
\begin{theorem} \label{thmlargedev}
Let $X = [x_{ij}]$ be a~random matrix with $\mathbb{E}(x_{ij}^2) = 1/M$
and $\lim_{N \to\infty} d_N \in(0,\infty) \setminus\{1\} $.
%entries satisfying \eqref{eqnXmat}
Assume that the column vectors of the matrix $X$ are mutually
independent. Furthermore, suppose that for any fixed $j \leq N$, the
random variables defined by $a_i = x_{ij}, 1 \leq i \leq M$, satisfy the
large deviation bounds (\ref{eqnld1}), (\ref{eqnld2}) and
(\ref{eqnld3}), for any $A_i \in\mathbb{C}$ and $B_{ij} \in\mathbb{C}$
and some $\zeta> 0$. Then the conclusions of Theorems \ref{451} and
\ref{452} hold for the random matrix~$X$.
\end{theorem}

Thus Theorem \ref{thmlargedev} extends the universality results to
a~large class of matrix ensembles. For instance, let $h_{ij}$ be
a~sequence of i.i.d. random variables, and set
%
%
%e3.11 #&#
\begin{equation}
\label{eqnstudtstat} x_{ij} = {h_{ij} \over\sqrt{\sum_{i=1}^M
h^2_{ij}}}, \qquad1 \leq i \leq
M, 1 \leq j \leq N.
\end{equation}
Thus the entries of the column vector $(x_{1j}, x_{2j}, \ldots,
x_{Mj})$ are not independent, but exchangeable. Clearly
$\mathbb{E}(x^2_{ij}) = {1 \over M}$. The random variables $x_{ij}$
given by (\ref{eqnstudtstat}) are called self normalized sums and arise
in various statistical applications. For instance, the matrix $X =
[x_{ij}]$ constructed above is called the correlation matrix (see
\cite{John01,PillYin1102}) and is often preferred in applications such
as principal component analysis (PCA) due to the scale invariance of
the correlation matrix.

\begin{pf*}{Proof of Theorem \ref{thmlargedev}}
In the proofs of Theorems \ref{451} and \ref{452}, we use only the
large deviation properties of $a_i = x_{ij}$ and the\vspace*{1pt} fact that
$\mathbb{E}(x_{ij}^2) = 1/M$, instead of independence and
sub-exponential decay. Therefore the proofs of Theorems \ref{451} and
\ref{452} in fact yield Theorem \ref{thmlargedev}.
\end{pf*}

%s4 #&#
\section{Universality of eigenvalues at edge}
\label{proUE} In this section we give the proof of edge universality
stated in Theorem \ref{twthm}. For simplicity, we focus on the case
$\lim_{N \to\infty} d_N \in(0,\infty) \setminus\{1\}$ first and return
to the $\lim_{N \to\infty} d_N = 1$ at the end of this section. The
proof is loosely based on Theorem 2.4 of \cite{EYYrigid} which is an
analogous result for Wigner matrices, but in our case \emph{there is
a~key difference}: the entries within the same column of the matrix $H
= {X}^\dagger X$ are \emph{dependent}. To address this difficulty, we
give a~novel argument involving the Green function comparison. In the
following we consider the largest eigenvalue $\lambda_1$, but the same
argument applies to the smallest nonzero eigenvalue as well. Also for
the rest of this section, let us fix a~constant $\zeta> 0$.

For any $E_1\le E_2$ let
\[
\mathcal{N}(E_1, E_2): = \#\{ E_1\le
\lambda_j\le E_2\}
\]
denote the number of eigenvalues of the covariance matrix ${X }^\dagger
X $ in $[E_1, E_2]$ where $X $ is a~random matrix whose entries satisfy
(\ref{eqnXmat}) and~(\ref{eqnXmatexpbd}). By Theorems~\ref{451}~and~\ref{452} (rigidity of eigenvalues), there exists a~positive
constant $C_\zeta$ such that
%
%
%e4.1 #&#
%e4.2 #&#
\begin{eqnarray}
|\lambda_1 -\lambda_+ | &\leq&\varphi^{C_\zeta}N^{-2/3},
\label{6-1}
\\
\mathcal{N} \bigl(\lambda_+-2\varphi^{C_\zeta}N^{-2/3}, \lambda
_++2\varphi ^{C_\zeta}N^{-2/3} \bigr) &\leq&\varphi^{2C_\zeta}
\label{6-2}
\end{eqnarray}
hold with $\zeta$-high probability. Using these estimates, we can
assume that the parameter $s$ in (\ref{tw}) satisfies
%
%
%e4.3 #&#
\begin{equation}
\label{25} -\varphi^{C_\zeta} \le s \le\varphi^{C_\zeta
}.
\end{equation}
Set%
%
%e4.4 #&#
\begin{equation}
\label{defEL} E_\zeta:= \lambda_++2\varphi^{C_\zeta}N^{-2/3}
\end{equation}
and for any $E\le E_\zeta$ define $\chi_E:= {\mathbf1}_{[E, E_\zeta]}$
to be the characteristic function of the interval $[E,
E_\zeta]$. For any $\eta>0$ we define%
%
%e4.5 #&#
\begin{equation}
\label{thetam} \theta_\eta(x):=\frac{\eta}{\pi(x^2+\eta^2)} =
\frac{1}{\pi} \Im\frac{1}{x-i\eta}
\end{equation}
to be an approximate delta function on scale $\eta$. In the following
elementary lemma we compare the sharp counting function $\mathcal{N}(E,
E_\zeta)= \operatorname{Tr}\chi _E(H)$ by its approximation smoothed on
scale $\eta$. Notice that for any $\ell> 0$,
\[
\operatorname{Tr}\chi_{E-\ell} \ast\theta_\eta(H)=N
\frac{1}\pi \int_{E-\ell}^{E_\zeta}\Im m(y+i\eta) \,\mathrm{d}y.
\]
Let us fix $\varepsilon> 0$ and set%
%
%e4.6 #&#
\begin{equation}
\label{defeta1} \eta_1 = N^{-2/3 - 9 \varepsilon}.
\end{equation}

%
%le4.1 #&#
\begin{lemma}\label{lem21}
%Suppose that the assumptions of Theorem \ref{twthm} hold
%and $L$, $\phi$ satisfy \eqref{6-1} and \eqref{6-2}.
For any $\varepsilon>0$, set $\ell_1:=N^{-2/3-3\varepsilon}$. Then for
any $E$
satisfying%
%
%e4.7 #&#
\begin{equation}
\label{E-2N} |E-\lambda_+|\leq\tfrac{3} 2 \varphi^{C_\zeta}N^{-2/3},
\end{equation}
where the constant $C_\zeta$ is as in (\ref{6-1})--(\ref{defEL}), the
bound
%
%
%e4.8 #&#
\begin{equation}
\label{610} \bigl|\operatorname{Tr}\chi_E(H) - \operatorname{Tr}
\chi_E \ast \theta_{\eta_1} (H)\bigr| \le C \bigl( N^{-2\varepsilon}
+ \mathcal{N}(E- \ell_1, E+\ell_1) \bigr)
\end{equation}
holds with $\zeta$-high probability.
%This estimate holds
%for both the $\bv$ and $\bw$ ensembles.
\end{lemma}
\begin{pf} %{Proof of Lemma \ref{lem21}}
From inequalities (\ref{6-1}), (\ref{6-2}) above, and (6.13) and (the
first line of) (6.17) of \cite{EYYrigid} we obtain
%
%
%e4.9 #&#
\begin{eqnarray}
\label{6-22}
&& \bigl|\operatorname{Tr}\chi_E(H) - \operatorname{Tr}
\chi_E \ast\theta _{\eta_1} (H)\bigr|\nonumber
\\
&&\qquad \le C \bigl( \mathcal{N}(E-
\ell_1, E+ \ell_1) +N^{- 5\varepsilon} \bigr)
\\
&&\quad\qquad{} +C N \eta_1 (E_\zeta-E) \int_\mathbb{R}
\frac{1} {y^2 + \ell_1^2 } \Im m( E-y+i\ell_1) \,\mathrm{d}y.\nonumber
\end{eqnarray}
By definition, $ \int_\mathbb{R}\Im m( E-y+i\ell_1)\,\mathrm{d}y=O(1)$.
For any fixed small enough $c>0$,
\[
\int_{ |y|\geq\varepsilon} \frac{1} {y^2 + \ell_1^2 } \Im m( E-y+i
\ell_1) \,\mathrm{d}y =O \bigl(c^{-2} \bigr).
\]
On the interval $ |y| \leq c$ we use (\ref{Lambdafinal}), that is,
\[
\Im m( E-y+i\ell_1) \le\Im m_c( E-y+i
\ell_1) + \frac{\varphi^{C_\zeta}} { N \ell_1}
\]
and the elementary estimate $ \Im m_c( E-y+i\ell_1)\le C\sqrt{\ell_1+ |
E-y -\lambda_+ |}$. Using the definitions of $\ell_1$ and $\eta_1$ it
can be shown that
(see inequality (6.18) of \cite{EYYrigid})%
\[
N \eta_1 (E_\zeta-E) \int_\mathbb{R}
\frac{1} {y^2 + \ell_1^2 } \Im m( E-y+i\ell_1) \,\mathrm{d}y\leq
N^{-2\varepsilon}.
\]
Now the lemma follows from (\ref{6-22}).
\end{pf}

Let $q\dvtx\mathbb{R}\to\mathbb{R}_+$ be a~smooth cutoff function such
that
\begin{eqnarray*}
q(x) &=& 1\qquad\mbox{if }|x| \le1/9,
\\
q(x) &=& 0\qquad\mbox{if }|x| \ge2/9
\end{eqnarray*}
and we assume that $q(x)$ is decreasing for $x\ge0$. Then we have the
following corollary for Lemma \ref{lem21} (which is the counterpart of
Corollary 6.2 in \cite{EYYrigid}):

%
%
%co4.2 #&#
\begin{corollary} \label{23} Let $\ell_1$ be as in Lemma \ref{lem21},
and set $\ell: = \frac{1}{2}\ell_1 N^{ 2\varepsilon} =
\frac{1}{2}N^{-2/3 -
\varepsilon}$. Then for all $E$ such that%
%
%e4.10 #&#
\begin{equation}
\label{condE-} |E-\lambda_+|\leq\varphi^{C_\zeta}N^{-2/3},
\end{equation}
where the constant $C_\zeta$ is as in (\ref{6-1})--(\ref{defEL}), the
inequality
%
%
%e4.11 #&#
\begin{equation}
\label{41new} \qquad \operatorname{Tr}\chi_{E+ \ell} \ast\theta_{\eta
_1}
(H) - N^{-\varepsilon} \le\mathcal{N}(E, \infty) \le\operatorname {Tr}
\chi_{E- \ell} \ast \theta_{\eta_1} (H) + N^{-\varepsilon}
\end{equation}
holds with $\zeta$-high probability. Furthermore, there exists $N_0
\in\mathbb{N}$ independent of $E$ such
that for all $N \geq N_0$,%
%
%e4.12 #&#
\begin{eqnarray}\label{67E}
&& \mathbb{E}q \bigl(\operatorname{Tr}\chi_{E-\ell} \ast
\theta_{\eta_1} (H) \bigr)
\nonumber\\[-8pt]\\[-8pt]
&&\qquad \le\mathbb{P} \bigl(\mathcal{N}(E, \infty) = 0
\bigr)
\le \mathbb{E}q \bigl(\operatorname{Tr}\chi_{E+\ell} \ast
\theta_{\eta_1} (H) \bigr) + Ce^{-\varphi^{C_\zeta}}.\nonumber
\end{eqnarray}
%
%as long as \eqref{condE-} holds.
\end{corollary}

\begin{pf} For any $E$ satisfying (\ref{condE-}) we have
$E_\zeta-E \gg\ell$ thus
$|E-\lambda_+-\ell|N^{2/3}\le\frac{3}{2}\varphi^{C_\zeta}$ [see
(\ref{E-2N})]; therefore (\ref{610}) holds for $E$ replaced with $y\in
[E-\ell, E]$\vadjust{\goodbreak} as well. We thus obtain
\begin{eqnarray*}
\operatorname{Tr}\chi_{E } (H) & \le& \ell^{-1} \int
_{E- \ell}^E \,\mathrm{d}y \operatorname{Tr}
\chi_y (H)
\\
& \le& \ell^{-1} \int_{E- \ell}^E
\,\mathrm{d}y \operatorname{Tr}\chi_y \ast\theta_{\eta_1} (H)
\\
&&{} + C\ell^{-1} \int_{E- \ell}^E \,\mathrm{d}y \bigl[ N^{-2\varepsilon} +
\mathcal{N}(y-\ell_1, y+ \ell_1) \bigr]
\\
& \le& \operatorname{Tr}\chi_{E- \ell} \ast\theta_{\eta_1} (H) +
CN^{- 2\varepsilon}+ C \frac{\ell_1}{\ell} \mathcal{N}(E- 2\ell, E + \ell)
\end{eqnarray*}
holds with $\zeta$-high probability. From (\ref{resrig2}),
(\ref{condE-}), $\ell_1/\ell=2N^{-2\varepsilon}$ and $ \ell\le
N^{-2/3}$, we gather that
\[
\frac{\ell_1}{\ell} \mathcal{N}(E- 2\ell, E+ \ell) \le N^{1-2\varepsilon} \int
_{E-2\ell}^{E+\ell}\varrho_c(x)\,\mathrm{d}x +
N^{-2\varepsilon
}(\log N)^{L_1} \le\frac{1}{2}N^{-\varepsilon}
\]
holds with $\zeta$-high probability, where we estimate the explicit
integral using the fact the integration domain is in
a~$CN^{-2/3}\varphi^{C_\zeta}$-vicinity of the edge at $\lambda_+$. We
have thus proved
\[
\mathcal{N}(E, E_\zeta) = \operatorname{Tr}\chi_{E } (H)
\le \operatorname{Tr} \chi_{E- \ell} \ast\theta_{\eta_1} (H) +
N^{-\varepsilon}.
\]
Using\vspace*{-1pt} (\ref{6-1}) we can replace $\mathcal{N}(E,
E_\zeta)$ by $\mathcal{N}(E, \infty)$ with a~change of probability of
at most $O(e^{-\varphi^{C_\zeta}})$. This proves the upper bound of
(\ref{41new}), and the lower bound can be proved similarly.

When event (\ref{41new}) holds, the condition $\mathcal{N}(E, \infty )
= 0$ implies that $\operatorname{Tr}\chi_{E+\ell} \ast\theta_{\eta_1}
(H) \leq1/9$.
Thus we have%
%
%e4.13 #&#
\begin{equation}
\label{Pleft} \mathbb{P} \bigl(\mathcal{N}(E, \infty)=0 \bigr)\leq \mathbb{P}
\bigl(\operatorname{Tr} \chi_{E+\ell} \ast\theta_{\eta_1} (H) \leq1/9
\bigr)+Ce^{-\varphi^{C_\zeta}}.
\end{equation}
Together with the Markov inequality, this proves the upper bound in
(\ref{67E}). For the lower bound, we use%
\begin{eqnarray*}
\mathbb{E}q \bigl(\operatorname{Tr}\chi_{E-\ell} \ast\theta
_{\eta_1} (H) \bigr) &\le& \mathbb{P} \bigl( \operatorname{Tr}
\chi_{E-\ell} \ast\theta_{\eta_1} (H)\le 2/9 \bigr)
\\
&\le&\mathbb{P} \bigl( \mathcal{N}(E,\infty)\le2/9+N^{-\varepsilon
} \bigr) =
\mathbb{P} \bigl( \mathcal{N}(E, \infty)=0 \bigr),
\end{eqnarray*}
where we used the upper bound from (\ref{41new}) and the fact that
$\mathcal{N}(E,\infty)$ is an integer. This completes the proof of
Corollary \ref{23}.
\end{pf}

%s4.1 #&#
\subsection{Green function comparison theorem}\label{secgreen}

Let $X^{\mathbf{v}} = [x^{\mathbf{v}} _{ij}]$, with the entries
$x^{\mathbf{v}}_{ij}$ satisfying (\ref{eqnXmat}) and
(\ref{eqnXmatexpbd}), $H^{\mathbf{v}} = {X^{\mathbf{v}}}^\dagger
X^{\mathbf{v}}$, and let $G^{\mathbf {v}}(z) =
({X^{\mathbf{v}}}^\dagger X^{\mathbf{v}} - z)^{-1} = (
H^{\mathbf{v}}-z)^{-1} $ be the Green function corresponding to
$X^{\mathbf{v}}$. Define the matrices $X^{\mathbf{w}}$, $H^{\mathbf
{w}}$ and the Green function $G^{\mathbf{w}}(z)$ analogously. Define
$m^{\mathbf{v}}(z) = {1 \over N} \operatorname{Tr} G^{\mathbf{v}}(z)$
and $m^{\mathbf{w}}(z) = {1 \over N} \operatorname
{Tr}G^{\mathbf{w}}(z)$. The operators $\mathbb{E}^{\mathbf{v}},
\mathbb{E}^{\mathbf{w}}$ denote the expectations under the
distributions of $X^{\mathbf{v}}$ and $X^{\mathbf{w}}$, respectively.

Also notice from (\ref{thetam}) that $\theta_\eta(H)= \frac{1}{\pi }
\Im m(i\eta)$. Corollary \ref{23} bounds the probability of
$\mathcal{N}(E,\infty)=0$ in terms of the expectations of two
functionals of Green functions. In this subsection, we show that the
difference between the expectations of these functionals with respect
to the two ensembles $X^{\mathbf{v}}$ and $X^{\mathbf{w}}$ is
negligible assuming their second moments match. The precise statement
is the following Green function comparison theorem on the edges. All
statements are formulated for the upper spectral edge $\lambda_+$, but
identical arguments hold for the lower spectral edge $\lambda_-$ as
well.

%
%
%th4.3 #&#
\begin{theorem}[(Green function comparison theorem on the edge)] \label{GFCT}
%Suppose that the assumptions of Theorem \ref{twthm}%, including
% hold.
Let $F\dvtx\mathbb{R}\to\mathbb{R}$ be a~function whose derivatives
satisfy
%
%
%e4.14 #&#
\begin{equation}
\label{gflowder} \max_{x}\bigl|F^{(\alpha)}(x)\bigr| \bigl(|x|+1
\bigr)^{-C_1} \leq C_1,\qquad\alpha=1, 2, 3, 4
\end{equation}
with some constant $C_1>0$. Then there exists $\varepsilon_0>0$, $N_0
\in\mathbb{N}$ depending only on $C_1$ such that for any
$\varepsilon<\varepsilon_0$ and $N \geq N_0$ and for any real numbers
$E$, $E_1$ and $E_2$ satisfying
\[
|E-\lambda_+|\leq N^{-2/3+\varepsilon}, \qquad|E_1-\lambda_+|\leq
N^{-2/3+\varepsilon}, \qquad|E_2-\lambda_+|\leq N^{-2/3+\varepsilon}
\]
and $\eta= N^{-2/3-\varepsilon}$, we have%
%
%e4.15 #&#
\begin{equation}
\label{maincomp} \bigl|\mathbb{E}^\mathbf{v}F \bigl( N \eta\Im m^{\mathbf{v}}
(z) \bigr) - \mathbb{E}^\mathbf{w}F \bigl( N \eta\Im m^{\mathbf{w}} (z)
\bigr) \bigr| \le C N^{-1/6+C \varepsilon},\quad z=E+i\eta\hspace*{-35pt}
\end{equation}
and%
%
%e4.16 #&#
\begin{eqnarray}
\label{c1}
\qquad && \biggl|\mathbb{E}^\mathbf{v}F \biggl(N \int
_{E_1}^{E_2} \,\mathrm{d}y\, \Im m^{\mathbf{v}} (y +i
\eta) \biggr)-\mathbb {E}^\mathbf{w}F \biggl(N \int_{E_1}^{E_2}
\,\mathrm{d}y\, \Im m^{\mathbf{w}} (y+i\eta) \biggr) \biggr|
\nonumber\\[-16pt]\\[-2pt]
&&\qquad \leq C N^{-1/6+C \varepsilon}.\nonumber
\end{eqnarray}
\end{theorem}

Theorem \ref{GFCT} holds in much greater generality. We state the
following extension which can be used to prove (\ref{twa}), the
generalization of Theorem \ref{twthm}. The class of functions $F$ in
the following theorem can be enlarged to allow some polynomially
increasing functions similar to (\ref{gflowder}). But for our
application of the above theorem to prove (\ref{twa}), the following
form is sufficient.

%
%
%th4.4 #&#
\begin{theorem} \label{GFCT2}
Suppose that the assumptions of Theorem \ref{twthm} hold. Fix any
$k\in\mathbb{N}_+$ and let $F\dvtx \mathbb{R}^k \to\mathbb{R}$ be a
bounded smooth function with bounded derivatives. Then there exists
$\varepsilon_0>0$, $N_0 \in \mathbb{N}$ depending only on $C_1$ such
that for any $\varepsilon<\varepsilon_0$ and $N \geq N_0$, there exists
$\delta>0$ such that for any sequence of real numbers $E_k < \cdots<
E_{1}< E_0 $ with $|E_j-\lambda_+|\le N^{-2/3+\varepsilon}$,
$j=0,1,\ldots, k$, and $\eta= N^{-2/3 - \varepsilon}$ we have
%
%
%e4.17 #&#
\begin{eqnarray}
\label{c11}
\qquad && \biggl|\mathbb{E}^{\mathbf{v}} F \biggl(N \int
_{E_1}^{E_0} \,\mathrm{d}y\, \Im m (y +i\eta), \ldots, N
\int_{E_k}^{E_0} \,\mathrm{d}y\, \Im m (y+i\eta) \biggr)
\nonumber\\[-8pt]\\[-8pt]
&&\hspace*{172pt}{} - \mathbb{E}^{\mathbf{w}} F \bigl(m^\mathbf{v} \rightarrow
m^\mathbf{w} \bigr) \biggr| \leq N^{-\delta},\nonumber
\end{eqnarray}
where in the second term the arguments of $F$ are changed from
$m^\mathbf{v}$ to $m^\mathbf{w}$ and all other parameters remain
unchanged.
\end{theorem}

\begin{pf}
The proof of Theorem \ref{GFCT2} is similar to that of Theorem
\ref{GFCT} and will be omitted.
\end{pf}

Before proceeding further, let us state the following theorem which
gives sufficient criteria for proving edge universality\vadjust{\goodbreak} for matrix
ensembles of the form ${Y}^\dagger Y$ for various types of data
matrices $Y$. Let $Y_{M \times N} = [y_{ij}], Z_{M\times N} = [z_{ij}]$
be two matrix ensembles, and set $H^Y = {Y}^\dagger Y, H^Z =
{Z}^\dagger Z$. Define the corresponding Green functions $G^Y = (H^Y -
z)^{-1}, G^{Z} = (H^Z - z)^{-1}$ and denote their respective empirical
Stieltjes transforms by $m^Y$, $m^Z$.
%
%
%th4.5 #&#
\begin{theorem} \label{thmsuffedgeuni}
Assume that the matrices $Y, Z$ satisfy the conclusions stated in items
\textup{(i)}, \textup{(ii)} and \textup{(iii)} of Theorem \ref{451}.
Furthermore, assume that $m^Y$ and $m^Z$ satisfy the conclusions of
Theorems \ref{GFCT} and \ref{GFCT2}. Then the asymptotic eigenvalue
distribution of the matrices $H^Y,H^Z$ at the edge are identical; that
is, the conclusions of Theorem \ref{twthm} are satisfied with
$X^{\mathbf{v}} = Y$ and $X^{\mathbf{w}}= Z$.
\end{theorem}

%
%
%re4.6 #&#
\begin{remark}
Thus our results can be used to show edge universality for cases far
beyond covariance matrices. In \cite{PillYin1102} we use Theorem
\ref{thmsuffedgeuni} to prove the edge universality of correlation
matrices.
\end{remark}

\begin{pf*}{Proof of Theorem \ref{thmsuffedgeuni}}
An inspection of the proofs will reveal that, for the arguments used in
our application of the Green function comparison method to go through,
all we need are the strong MP law and the rigidity of eigenvalues
[items~\textup{(i)}, \textup{(ii)} and \textup{(iii)} of
Theorem~\ref{451}] and Theorems~\ref{GFCT} and~\ref{GFCT2}.
\end{pf*}

Recall that all discussion so far in this section has been under the
assumption that $\lim_{N \to\infty} d_N \in(0,\infty) \setminus \{1\}$.
Now we first prove Theorem \ref{twthm} when $\lim_{N \to \infty} d_N
\in(0,\infty) \setminus\{1\}$, assuming that Theorem \ref{GFCT} holds
and then give the proof of Theorem \ref{GFCT}. Finally we return to
prove Theorem \ref{twthm} for $\lim_{N \to\infty}  d_N = 1$ at the end
of this section.

\begin{pf*}{Proof of Theorem \ref{twthm} for the case $\lim_{N \to
\infty } d_N = (0,\infty) \setminus\{1\}$}
%Suppose $\lim_{N \to\infty} d_N = (0,\infty) \setminus\{1\}$.
Define $ E_\zeta$ as in (\ref{defEL}) with a~constant $C_\zeta$ such
that (\ref{6-1}) and (\ref{6-2}) hold. Therefore we can assume that
(\ref{25}) holds for the parameter $s$. Let $E:=\lambda_++sN^{-2/3}$ so
that $|E-\lambda_+|\leq\varphi^{C_\zeta}N^{-2/3}$. Using (\ref{67E}),
for
any sufficiently small $\varepsilon>0$, we have%
\[
\mathbb{E}^\mathbf{w}q \bigl(\operatorname{Tr}\chi_{E-\ell} \ast
\theta_{\eta_1} (H) \bigr) \le\mathbb{P}^\mathbf{w} \bigl(
\mathcal{N}(E, \infty) = 0 \bigr)
\]
with
\[
\ell: = \tfrac{1}{2}N^{-2/3-\varepsilon}, \qquad\eta_1:=
N^{-2/3-9\varepsilon}.
\]
Recall that by definition
\[
\operatorname{Tr}\chi_{E-\ell} \ast\theta_{\eta_1} (H)=N
\frac
{1}\pi \int_{E-\ell}^{E_\zeta}\Im m(y+i
\eta_1 )\,\mathrm{d}y.
\]
Bound (\ref{c1}) applied to the case $E_1=E-\ell$ and $E_2= E_\zeta$
shows that there exists $\delta>0$, such that
%
%
%e4.18 #&#
\begin{equation}
\label{645} \mathbb{E}^\mathbf{v}q \bigl(\operatorname{Tr}\chi
_{E-\ell} \ast \theta_{\eta_1} (H) \bigr)\leq\mathbb{E}^\mathbf{w}q
\bigl(\operatorname{Tr} \chi_{E-\ell} \ast\theta_{\eta_1} (H)
\bigr)+N^{-\delta}.\vadjust{\goodbreak}
\end{equation}
%
%(Note that $9\varepsilon$ plays the role of the $\varepsilon$ in the
%Green function comparison theorem).\tcr{Jun: for the rest of this
%proof, please check where one uses $\eta_1$ vs $\eta_0$ and correct
%it, and I will check.}
Then applying the right-hand side of (\ref{67E}) in Lemma \ref{23} to
the left-hand side of~(\ref{645}), we have
\[
\mathbb{P}^\mathbf{v} \bigl(\mathcal{N}(E-2\ell, \infty) = 0 \bigr) \leq
\mathbb{E}^\mathbf{v}q \bigl(\operatorname{Tr}\chi_{E-\ell} \ast
\theta_{\eta_1} (H) \bigr)+C\exp{ \bigl(-c\varphi^{O(1)} \bigr)}.
\]
Combining these inequalities, we have
%
%
%e4.19 #&#
\begin{equation}
\label{Pbv} \mathbb{P}^\mathbf{v} \bigl(\mathcal{N}(E- 2\ell, \infty)
= 0 \bigr) \leq\mathbb{P} ^\mathbf{w} \bigl(\mathcal{N}(E, \infty) = 0
\bigr)+2N^{-\delta}
\end{equation}
for sufficiently small $\varepsilon>0$ and sufficiently large $N$.
Recalling that $E= \lambda_++sN^{-2/3}$, this proves the first
inequality of (\ref{tw}) and, by switching the roles of $\mathbf{v},
\mathbf{w}$, the second inequality of (\ref{tw}) as well. This
completes the proof of Theorem \ref{twthm}.
\end{pf*}

\begin{pf*}{ Proof of Theorem \ref{GFCT}}
%The proof is similar to that of Lemma \ref{comparison}.
We need to compare the matrices $H^{\mathbf{v}}$ and $H^\mathbf{w}$.
Instead of replacing the matrix elements one by one ($NM$ times) and
comparing their successive differences, the \emph{key new idea} here is
to estimate the successive difference of matrices which differ by
a~column. Indeed for $1\leq\gamma\leq N$, denote by $X_\gamma$ the
random matrix whose $j$th column is the same as that of $X^{\mathbf
{v}}$ if $j < \gamma$ and that of $X^{\mathbf{w}}$ otherwise; in
particular $X_0 = X^{
\mathbf{v}}$ and $X_{N} = X^{\mathbf{w}}$. As before, we define%
\[
H_\gamma=X_\gamma^\dagger X_\gamma.
\]
We will compare $H_{\gamma-1}$ with $H_\gamma$ using the following
lemma.
%Let
%$$Q= X_{\gamma-1}^{(\gamma)}= X_{\gamma}^{(\gamma)}, \qquad R =
For simplicity, we denote
\[
\widetilde m{}^{(i)}(z)=m^{(i)}(z)-(Nz)^{-1}.
\]

%
%le4.7 #&#
\begin{lemma} \label{lemgreenfned}
For any random matrix $X$ whose entries satisfy (\ref{eqnXmat}) and
(\ref{eqnXmatexpbd}), if $|E-\lambda_+|\leq N^{-2/3+\varepsilon}$ and
$N^{-2/3}\gg\eta\geq N^{-2/3-\varepsilon}$ for some $\varepsilon>0$,
then we have
%
%
%e4.20 #&#
\begin{equation}
\label{wdb}
\qquad\qquad \mathbb{E}F \bigl( N\eta\Im m(z) \bigr) - \mathbb{E}F \bigl( N
\eta\Im\widetilde m{}^{(i)}(z) \bigr)= A \bigl(X^{(i)},
m_1, m_2 \bigr)+N^{-7/6+C\varepsilon},
\end{equation}
where the functional $A(X^{(i)}, m_1, m_2 )$ depends only on the
distribution of $X^{(i)}$ and the first two moments $m_1, m_2 $ of
$\sqrt M x_{ji}=\sqrt M (X)_{ji}$ $(1\leq j\leq M)$.
\end{lemma}

Notice that $X_\gamma^{(\gamma)}$ is equal to
$X_{\gamma-1}^{(\gamma)}$. We also have that the first two moments of
the entries of $X^{\mathbf{v}}$ and $X^{\mathbf{w}}$ are identical.
Thus Lemma
\ref{lemgreenfned} implies that%
%
%e4.21 #&#
\begin{equation}
\label{ygz3} \qquad\mathbb{E}F \biggl( \eta\Im\operatorname{Tr}\frac
{1} { H_{\gamma-1}-z}
\biggr) - \mathbb{E}F \biggl( \eta\Im\operatorname{Tr}{1 \over H_\gamma- z}
\biggr)= O \bigl(N^{-7/6+C\varepsilon} \bigr).
\end{equation}
%
%where $A(X^{(\gamma)}_\gamma, m_1, m_2 )$ only depends on the
%distribution of $X^{(\gamma)}_\gamma$ and $m_1, m_2 $.
Now the proof of Theorem \ref{GFCT} now can be completed via a~simple
telescoping argument.% As \eqref{ygz1}, it implies \eqref{maincomp}.
%The \eqref{c1} can be proved similarly.
Thus to finish the proof of Theorem \ref{GFCT}, all that needs to be
shown is Lemma \ref{lemgreenfned} which is proven below.
\end{pf*}
\begin{pf*}{Proof of Lemma \ref{lemgreenfned}}
Fix $\zeta>0$, $\varepsilon> 0$ and, without loss of generality, assume
that $i=1$. Recall that $N^{-2/3} \gg\eta\geq N^{-2/3 - \varepsilon}$
and $|E - \lambda_+| \leq N^{-2/3 + \varepsilon}$. First, we claim the
following bounds for $G^{(1)}$ and $\mathcal{G}^{(1)}$:
%
%
%e4.22 #&#
%e4.23 #&#
\begin{eqnarray}
\label{af1} \bigl| \bigl\langle\mathbf{x}_{1 }, \bigl(
\mathcal{G}^{(1)}(z) \bigr)^2 \mathbf{x}_1 \bigr
\rangle \bigr|&\leq& N^{1/3+C\varepsilon}, \qquad z = E + i\eta
\\
\bigl| \bigl[\mathcal{G}^{(1)}(z) \bigr]_{ij}
\bigr|&\leq& N^{ C\varepsilon},
\nonumber\\[-8pt]\label{af2} \\[-8pt]
\bigl| \bigl[ \bigl[ \mathcal{G}^{(1)}(z)
\bigr]^2 \bigr]_{ij} \bigr|&\leq& N^{ 1/3+C\varepsilon}, \qquad z
= E + i \eta\nonumber
\end{eqnarray}
with $\zeta$-high probability for some $C>0$. In the above, $\mathbf
{x}_1$ denotes the first column of the matrix $X$. In (\ref{af2}) we
allow $i = j$. The proof of these bounds is postponed to the end.

Now using (\ref{eqnGii}) and (\ref{eqnGijGijk}), we have
%
%
%e4.24 #&#
\begin{eqnarray}
\operatorname{Tr}G -\operatorname{Tr}G^{(1)} +z^{-1}& =&
\bigl(G_{11}+z^{-1} \bigr)+\frac{
\langle\mathbf{x}_{1 }, X^{(1)}G ^{(1)} G ^{(1)}X^{(1)\dagger}
\mathbf{x}_1\rangle}{- z - z ( \mathbf{x}_1, \mathcal{G}^{(1)}(z)
\mathbf{x}_1)}
\nonumber\\[-8pt] \label{26mm}\\[-8pt]
%, \mG^{(1)}(z) \bx_1)}
%, \mG^{(1)}(z) \bx_1)}\\\nonumber=&
&=& z
G_{11} \bigl\langle\mathbf{x}_{1 }, \bigl(
\mathcal{G}^{(1)} \bigr)^2 (z) \mathbf{x}_1
\bigr\rangle.\nonumber
\end{eqnarray}
Define the quantity $B$ to be
%
%
%e4.25 #&#
\begin{equation}
\label{215a} B =-z m_c \biggl[ \bigl\langle\mathbf{x}_1,
\mathcal{G}^{(1)}(z) \mathbf{x}_1 \bigr\rangle- \biggl(
\frac
{-1}{z m_c(z)}-1 \biggr) \biggr].
\end{equation}
By (\ref{eqnGii}),
\[
B =-z m_c \biggl[ \biggl(\frac{-1}{z G_{11}(z)}-1 \biggr)- \biggl(
\frac
{-1}{z m_c(z)}-1 \biggr) \biggr]=\frac{m_c-G_{11}}{G_{11}}.
\]
From (\ref{Lambdaofinal}),
%we also have that
% holds with an extremely large probability. ($m_W^\mG$ is the
%classical value of $\tr\mG$)
%By the definition of $m_W$ and above bound on $( \bx_1, \mG^{(1)}(z)
we obtain that
%
%
%e4.26 #&#
\begin{equation}
\label{216a} | B| \leq N^{-1/3+2\varepsilon}\ll1,
\end{equation}
with $\zeta$-high probability. Therefore, we have the identity
%
%
%e4.27 #&#
\begin{equation}
\label{217a} G_{11}=\frac{m_c}{ B+1}=m_c\sum
_{k\geq0 }(-B)^k.
\end{equation}

Define $y$ with the left-hand side of (\ref{26mm}),%
%
%e4.28 #&#
\begin{equation}
y :=\eta \bigl( \operatorname{Tr}G - \operatorname{Tr}G^{(1)} +
z^{-1} \bigr) \label{defy} %\mu&:= \eta( \Im\tr G^{(1)} - {1 \over z}) = N\eta\Im m
\end{equation}
so that we have%
%
%e4.29 #&#
\begin{equation}
\label{eqndefutily} N\eta\Im m(z) = N\eta\Im\widetilde m{}^{(1)}(z) + y.
\end{equation}
Using
(\ref{26mm}) and (\ref{217a}) we obtain%
\[
y =\eta z G_{11} \bigl\langle\mathbf{x}_{1 }, \bigl(
\mathcal{G}^{(1)} \bigr)^2 \mathbf{x}_1 \bigr
\rangle= \sum_{k=1}^\infty y_k,
\qquad y_k:=\eta z m_c(-B)^{k-1} \bigl\langle
\mathbf{x}_{1 }, \bigl( \mathcal{G}^{(1)} \bigr)^2
\mathbf{x}_1 \bigr\rangle.
\]

Since $z$ and $m_c$ are $O(1)$, together with (\ref{af1}) and
(\ref{216a}) we see that the bounds
%
%
%e4.30 #&#
\begin{equation}
\label{boundyk} |y_{k}|\leq O \bigl(N^{-k/3+C\varepsilon
} \bigr)\quad
\mbox{and}\quad|y|\leq O \bigl(N^{-1/3+C\varepsilon
} \bigr)
\end{equation}
hold with $\zeta$-high probability. Consequently, using (\ref
{eqndefutily}), the expansion
%
%
%e4.31 #&#
\begin{eqnarray}
\label{jkall} && F \bigl( N\eta\Im m(z) \bigr) - F \bigl( N\eta \Im\widetilde
m{}^{(1)}(z) \bigr)
\nonumber\\[-8pt]\\[-8pt]
&&\qquad= \sum_{k=1}^3
{1 \over k!} F^{(k)} \bigl( N\eta\Im\widetilde
m{}^{(1)}(z) \bigr) (\Im y)^k+O \bigl(N^{-4/3+C\varepsilon} \bigr)\nonumber
\end{eqnarray}
holds with $\zeta$-high probability.

Now we estimate each of the three terms ($k=1,2,3$) on the right-hand
side of~(\ref{jkall}) individually. First, using (\ref{boundyk}) we
obtain that
%
%
%e4.32 #&#
\begin{equation}
\label{bdrd}  F^{(3)} \bigl( N\eta\Im\widetilde m{}^{(1)}(z)
\bigr) ( \Im y)^3 =F^{(3)} \bigl( N\eta\Im\widetilde
m{}^{(1)}(z) \bigr) (\Im y_1)^3+O
\bigl(N^{-4/3+C\varepsilon} \bigr)\hspace*{-35pt}
\end{equation}
holds with $\zeta$-high probability. Moreover, we
have%
%
%e4.33 #&#
\begin{eqnarray}
\label{eqnsumedgesep}\quad  \mathbb{E}_1(\Im y_1)^3 &=&
\mathbb{E}_1 (\eta z m_c)^3 \bigl\langle
\mathbf{x}_{1 }, \bigl( \mathcal{G}^{(1)} \bigr)^2
\mathbf{x}_1 \bigr\rangle^3
\nonumber\\[-8pt]\\[-8pt]
&=& (\eta z m_c)^3\sum_{k_1, \ldots, k_6 =
1}^M
\mathbb{E}_1 \Biggl(\prod_{i=1}^6x_{k_i1}
\Biggr) \prod_{i=1}^3 \bigl[ \bigl(
\mathcal{G}^{(1)} \bigr)^2 \bigr]_{k_{2i-1}, k_{2i} },\nonumber
\end{eqnarray}
where $\mathbb{E}_1$ is the expectation value with respect to $\mathbf
{x}_1$, the first column of $X$. Recall that $m_k$ denotes the $k$th
moment of $\sqrt M x_{j1}$. If there is an index $k_i $ which is
different from all the others in the product $\prod_{i=1}^6x_{k_i1}$,
then
\[
\mathbb{E}_1 \Biggl(\prod_{i=1}^6x_{k_i1}
\Biggr) = 0=m_1
\]
and if each $k_i$
appears exactly twice, then%
\[
\mathbb{E}_1 \Biggl(\prod_{i=1}^6x_{k_i1}
\Biggr) =m_2^3.
\]
Isolating the above two cases from the sum (\ref{eqnsumedgesep}),
we have%
\begin{eqnarray*}
\mathbb{E}_1(\Im y_1)^3&=&\widetilde A_3 \bigl(X^{(1)}, m_1, m_2
\bigr)
\\
&&{}+(\eta z m_c)^3\sum_{\mathcal{A} } \mathbb{E}_1
\Biggl(\prod_{i=1}^6x_{k_i1} \Biggr) \bigl[ \bigl(\mathcal{G}^{(1)}\bigr)^2
\bigr]_{k_1k_2} \bigl[ \bigl(\mathcal{G}^{(1)}\bigr)^2 \bigr]_{k_3k_4} \bigl[
\bigl(\mathcal{G}^{(1)}\bigr)^2 \bigr]_{k_5k_6},
\end{eqnarray*}
where $\mathcal{A}$ denotes the set of indices $k_i \in\{1,2,\ldots,
M\}$ such that (1) no $k_i$ appears exactly once in the product
$\prod_{i=1}^6x_{k_i1}$ and (2) there is an\vspace*{1pt} index $k_i$
which appears at least three times. Clearly, the functional $\widetilde
A_3(X^{(1)}, m_1, m_2)$ depends only on $X^{(1)}$, $m_1$ and
$m_2$. Furthermore, it readily follows that%
\[
\#\mathcal{A} \leq CN^2.
\]
Then using (\ref{af2}) and the bounds on $m_k$'s, it follows that
%
%
%e4.34 #&#
\begin{equation}
\label{rb55} \mathbb{E}_1(\Im y_1)^3=
\widetilde A_3 \bigl(X^{(1)}, m_1,
m_2 \bigr)+O \bigl(N^{-2+C\varepsilon} \bigr).
\end{equation}
It is easy to prove that $|N\eta\Im\widetilde m{}^{(1)}|\leq
N^{C\varepsilon}$ with $\zeta$-high probability. Using (\ref{bdrd}) and
the fact that $ \widetilde m{}^{(1)}$ depends only on $X^{(1)}$, we
have
%
%
%e4.35 #&#
\begin{equation}
\label{jk3} \quad \mathbb{E}F^{(3)} \bigl( N\eta\Im\widetilde m{}^{(1)}(z) \bigr) (\Im y)^3 =A_3
\bigl(X^{(1)}, m_1, m_2 \bigr)+O
\bigl(N^{-4/3+C\varepsilon} \bigr),
\end{equation}
where $ A_3(X^{(1)}, m_1, m_2)$ depends only on the distribution of
$X^{(1)}$, $m_1$ and~$m_2$.

Now we estimate the term with $F^{(2)}$ in (\ref{jkall}). As in
(\ref{bdrd}), we have
%
%
%e4.36 #&#
\begin{eqnarray}\label{bdrd20}
\qquad && F^{(2)} \bigl( N\eta\Im\widetilde m{}^{(1)}(z)
\bigr) (\Im y)^2
\nonumber\\[-8pt]\\[-8pt]
&&\qquad  =F^{(2)} \bigl( N\eta\Im\widetilde m{}^{(1)}(z) \bigr) \bigl[(\Im y_1)^2+2(\Im
y_1) (\Im y_2) \bigr]+O \bigl(N^{-4/3+C\varepsilon} \bigr).\nonumber
\end{eqnarray}
By definition,%
\begin{eqnarray*}
&& \mathbb{E}_1(\Im y_1) ^2+2(\Im
y_1) ( \Im y_2)
\\
&&\qquad =C_1(z)\eta^2
\bigl\langle \mathbf{x}_{1
}, \bigl( \mathcal{G}^{(1)} \bigr)
\mathbf{x}_1 \bigr\rangle \bigl\langle\mathbf{x}_{1 }
\bigl( \mathcal{G}^{(1)} \bigr)^2 \mathbf{x}_1
\bigr\rangle^2+C_2(z) \eta^2 \bigl\langle
\mathbf{x}_{1 } \bigl( \mathcal{G}^{(1)} \bigr)^2
\mathbf{x}_1 \bigr\rangle^2,
\end{eqnarray*}
where $C_1(z)$, $C_2(z)=O(1)$ are constants which depend only on $z$
and $m_c(z)$. Using the bounds on $\mathcal{G}^{(1)}$ in (\ref{af2}),
as in (\ref{rb55}), we have
\[
\mathbb{E}_1 \bigl[(\Im y_1) ^2+(\Im
y_1) (\Im y_2) \bigr]=\widetilde A_{2 }
\bigl(X^{(1)}, m_1, m_2 \bigr)+O
\bigl(N^{-5/3+C\varepsilon} \bigr),
\]
where $ \widetilde A_2(X^{(1)}, m_1, m_2)$ depends only on the
distribution of $X^{(1)}$, $m_1$ and $m_2$. Then with (\ref{bdrd20}),
as in (\ref{jk3}), we conclude that
%
%
%e4.37 #&#
\begin{equation}
\label{jk2} \quad\mathbb{E}F^{(2)} \bigl( N\eta\Im\widetilde m{}^{(1)}(z) \bigr) (\Im y )^2=A_2
\bigl(X^{(1)}, m_1, m_2 \bigr)+O
\bigl(N^{-4/3+C\varepsilon} \bigr)
\end{equation}
for some functional $A_2$ which depends only on the distribution of
$X^{(1)}$, $m_1$ and~$m_2$.

Finally we estimate the term $F^{(1)}$ in (\ref{jkall}). As in
(\ref{bdrd}), we have%
%
%e4.38 #&#
\begin{eqnarray}\label{bdrd2}
&& F^{(1)} \bigl( N\eta\Im\widetilde m{}^{(1)}(z)
\bigr) (\Im y)^2
\nonumber\\[-8pt]\\[-8pt]
&&\qquad  =F^{(1)} \bigl( N\eta\Im\widetilde m{}^{(1)}(z) \bigr) [ \Im y_1 + \Im y_2 +\Im
y_3 ]+O \bigl(N^{-4/3+C\varepsilon} \bigr).\nonumber
\end{eqnarray}
A similar argument as in (\ref{jk2}) and (\ref{jk3}) yields
%
%
%e4.39 #&#
\begin{equation}
\label{jk1} \mathbb{E}F^{(1)} \bigl( N\eta\Im\widetilde m{}^{(1)}(z) \bigr) (\Im y) =A_1 \bigl(X^{(1)},
m_1, m_2 \bigr)+O \bigl(N^{-4/3+C\varepsilon} \bigr).
\end{equation}
Inserting (\ref{jk1}), (\ref{jk2}) and (\ref{jk3}) into (\ref{jkall}),
we obtain (\ref{wdb}). Now to complete the proof of Lemma
\ref{lemgreenfned} we need to prove (\ref{af1}) and (\ref{af2}).

For (\ref{af1}), using the large deviation lemma (Lemma
\ref{lemlargdev}), we obtain that for any $\zeta>0$,%
%
%e4.40 #&#
\begin{eqnarray}
\label{eqnrigidmGcalc}
\bigl| \bigl( \mathbf{x}_{1 } \bigl( \mathcal{G}^{(1)}
\bigr)^2 \mathbf{x}_1 \bigr)\bigr|& \leq&
\varphi^{C_\zeta
} \bigl(N^{-1}\operatorname{Tr}\bigl|
\mathcal{G}^{(1)} \bigr|^4 \bigr)^{1/2}\nonumber
\\
& \leq&
\varphi^{C_\zeta} \biggl(\frac
{1}{N}\sum_{\alpha}
\frac{1}{|\lambda^{(1)}_\alpha-z|^4} \biggr)^{1/2}
\nonumber\\[-8pt]\\[-8pt]
&\leq& \varphi^{C_\zeta} \biggl(\frac{1}{N\eta^2}\sum
_{\alpha}\frac{1}{|\lambda
^{(1)}_\alpha-z|^2} \biggr)^{1/2}\nonumber
\\
&=& \varphi^{C_\zeta} \biggl(\frac{1}{N\eta ^3} \Im m^{(1)}(z)\biggr)^{1/2}\nonumber
\end{eqnarray}
with $\zeta$-high probability. Then using (\ref{Lambdafinal}) %and
we have (\ref{af1}). For (\ref{af2}), we note that
\[
\mathcal{G}^{(1)}=\frac{1}{X^{(1)}(X^{(1)})^\dagger-z}.
\]
Comparing with (\ref{eqngreen}), we see that the pair $\{\mathcal{G}
^{(1)},(X^{(1)})^\dagger\} $ plays the role of $\{G, X \}$. Since
$\sqrt{\frac{M }{N-1}}(X^{(1)})^\dagger$ is just an $(N-1)\times M$
random data matrix, whose entries have variance $(N-1)^{-1}$, the
results in (\ref{Lambdaofinal}) also hold for $\mathcal{G}^{(1)}$ with
slight changes. One can easily obtain that
\[
\max_{ij}\bigl| \bigl[ \mathcal{G}^{(1)}
\bigr]_{ij}\bigr| \leq C, \qquad\max_{i \neq j} \bigl| \bigl[
\mathcal{G}^{(1)} \bigr]_{ij}\bigr|\leq CN^{-1/3+C\varepsilon}
\]
with $\zeta$-high probability showing (\ref{af2}) and finishing the
proof of Lemma \ref{lemgreenfned} and consequently we have proved
Theorem \ref{GFCT}.
\end{pf*}
\begin{pf*}{Proof of Theorem \ref{twthm} for the case $\lim_{N \to
\infty } d_N = 1$}
%{\subsection{Proof of Theorem \ref{twthm} when $\lim_{N \to\infty}
%d_N=1$}
Note that this proof holds only for the largest eigenvalues. Without
loss of generality, set $1/2\leq d_N\leq2$. First, in the proof of
estimates in (\ref{Lambdafinal}) and (\ref{Lambdaofinal}) of $m(z)$, we
never used the assumption $\lim_{N \to\infty} d_N\neq1$ directly. We
only needed the property of $m_c(z)$ listed in Lemma \ref{lemmw}.
%%which is implied by $\lim_{N \to\infty} d_N\neq1$.
One can easily check that if $\Re z\ge\varepsilon$ for some constant
$\varepsilon$ independent of $N$, then $m_c(z)$ also satisfies the
properties in Lemma \ref{lemmw}, even if $\lim_{N \to\infty} d_N=1$.
Therefore for any fixed $\varepsilon>0$, expressions
(\ref{Lambdafinal}) and (\ref{Lambdaofinal}) still hold with $\zeta
$-high probability if we replace $\bigcap_{z \in{ \mathbf{S}}(C_\zeta)}
$ with $\bigcap_{z \in{ \mathbf{S}}(C_\zeta)\ \mathrm{and}\ \Re
z\ge\varepsilon} $.

Next, in step 1 in the proof of (\ref{443}), using the estimate of
$m(z)$ from (\ref{Lambdafinal}) and~(\ref{Lambdaofinal}), and
properties on $m_c(z)$ in Lemma \ref{lemmw}, we obtain that for any
$\zeta>0$, there exists some $D_\zeta>0$ such that%
%
%e4.41 #&#
\begin{equation}
\label{eqnmaxlj} \max\{\lambda_j\dvtx \lambda_j\leq5
\lambda_+ \} \leq\lambda_++N^{-2/3}\varphi^{4D_\zeta}
\end{equation}
holds with $\zeta$-high probability. In the proof, we used only the
estimates of $m(z)$ from (\ref{Lambdafinal}) and (\ref{Lambdaofinal})
for $z \in{ \mathbf{S}}(C_\zeta)$ and $ \Re z\in[\lambda_+, 5\lambda_+]
$. Now,\vadjust{\goodbreak} using our modified version of (\ref{Lambdafinal}) and
(\ref{Lambdaofinal}) (obtained by replacing $\bigcap_{z \in{ \mathbf
{S}}(C_\zeta)}$ with $\bigcap_{z \in{ \mathbf{S}}(C_\zeta)\
\mathrm{and}\  \Re z\ge\varepsilon} $), (\ref{eqnmaxlj}) can be easily
extended to the case \mbox{$\lim_{N \to\infty}  d_N=1$}.

Now,\vspace*{1pt} we claim that when $\lim_{N \to\infty}  d_N=1$,
$\lambda _1\leq 5\lambda_+$ holds with $\zeta$-high probability. The $M
\times N$ data matrix can be considered as a~minor of a~matrix
$\widetilde X$, which (1) is an $M\times\widetilde N$ matrix with
$\lim_{\widetilde N \to\infty} \widetilde N/M\ge1+c$ for some fixed
$c>0$, (2)~satisfies the condition of Theorem \ref{thmmain}. Let
$\lambda_1$, $\widetilde\lambda_1$ be the largest eigenvalue of
$X^\dagger X $ and $\widetilde X^\dagger\widetilde X $. By definition
and Theorem \ref{thmmain}, for small enough $c$ we have
\[
\lambda_1\leq\widetilde\lambda_1\leq5 \lambda_+.
\]

Combining the above two statements we obtain that for any $\zeta>0$,
there exists some $D_\zeta>0$ such that with $\zeta$-high probability
%
%
%e4.42 #&#
\begin{equation}
\label{jjp} \lambda_1\leq\lambda_++N^{-2/3}
\varphi^{4D_\zeta}.
\end{equation}

Likewise, step 2 [formula (\ref{fixE1E2})] in the proof of (\ref{443})
can also be extended to
%
%
%e4.43 #&#
\begin{eqnarray}
\bigl| \bigl({\mathfrak n}(E_1)-{\mathfrak n}(E_2) \bigr)-
\bigl( n_c(E_1)-n_c(E_2) \bigr) \bigr|
\le\frac{C (\log N)\varphi
^{C_\zeta}}{N},\nonumber
\\
\eqntext{E_1, E_2\in[ \lambda_+/2,
\lambda_+ ]}
\end{eqnarray}
since the proof relies only on the estimates of $m(z)$ for $ z \in{
\mathbf{S}}(C_\zeta)$ and $\Re z\in[E_1, E_2]$ given by our modified
version of (\ref{Lambdafinal}) and (\ref{Lambdaofinal}). Together with
(\ref{jjp}), we obtain that for any fixed $\varepsilon>0$, the rigidity
result (\ref{resrig}) holds for $j\le(1-\varepsilon) N$, and
(\ref{resrig2}) holds for $E\ge\varepsilon$.

Therefore, we conclude that (\ref{6-1}) and (\ref{6-2}) hold with
$\zeta $-high probability for the case $\lim_{N \to\infty} d_N=1$. Now
to obtain Theorem \ref{twthm} when $\lim_{N \to\infty}  d_N=1$, one
needs only to repeat the argument in this section. We note that in the
proof of (\ref{6-22}), we used (\ref{Lambdafinal}), but only for $z$'s
such that $\Re z$ is very close to $\lambda_+$, which is covered by our
modified version of (\ref{Lambdafinal}). Similarly for Corollary
\ref{23}, we used (\ref{resrig2}) but only for $E$'s which are very
close to $\lambda_+$. Therefore, we obtain Theorem \ref{twthm} in the
case $\lim_{N \to\infty}  d_N=1$.
\end{pf*}
%
%s5 #&#
\section{Universality of eigenvalues in bulk}
\label{proUB} In this section, our goal is to prove Theorem
\ref{thmmain}. This follows from our key technical result in
Section~\ref{secMPlaw} and the usual arguments using the ergodicity of
the Dyson Brownian motion mentioned in the
\hyperref[sec1]{Introduction}. Throughout this section we assume that
$\lim_{N \to\infty}  d_N \in(0,\infty) \setminus\{1\}$. Again, we note
that our arguments are valid for both real and complex valued entries.
%Although we assume the matrix $X$ to have only real valued entries,
%anal assume that $X^\bw$ is real Wishart matrix.

First, we consider a~flow of random matrices $X_t$ satisfying the
following
matrix valued stochastic differential equation%
%
%e5.1 #&#
\begin{equation}
\label{OUflow} \mathrm{d}X_t = \frac{1}{\sqrt{M}}\,\mathrm{d}
\beta_t - \frac{1}{2} X_t \,\mathrm{d}t,
\end{equation}
where $\beta_t$ is a~real matrix valued process whose elements are
standard real valued independent Brownian motions. The initial
condition $X_0 = X = [x_{ij}]$ satisfies (\ref{eqnXmat}) and
(\ref{eqnXmatexpbd}). For any fixed $t\ge0$, the distribution of $X_t$
coincides with that of
% The DBM (with the minor modification of replacing the Brownian motion
%by
% the Ornstein-Uhlenbeck (OU) process) is generated by the flow
%
%
%e5.2 #&#
\begin{equation}
\label{matrixdbm} {X_t} \stackrel{d} {=} e^{-t/2}
X_0 + \bigl(1-e^{-t} \bigr)^{1/2} V,
\end{equation}
where
%$ H_0$ is the initial hermitian Wigner matrix and
$V$ is a~real matrix with Gaussian entries which have mean $0$ and
variance $1/M$. The singular values of the matrix $X_t$ also satisfy
a~system of coupled SDEs which is also called the Dyson Brownian motion
(with a~drift in our case).
More precisely, let%
%
%e5.3 #&#
\begin{eqnarray}\label{H}
\mu&=&\mu_N(\mathrm{d} {\mathbf w}) = \frac{e^{-\mathcal{H}^\beta
_W({\mathbf w})}}{Z_\beta}\,\mathrm{d} {
\mathbf w},\nonumber
\\
\mathcal{H}^\beta_W({\mathbf w}) & =& \beta \Biggl[ \sum
_{i=1}^N \frac{w_{i}^{2}}{2d
} -\frac{1}{N} \sum_{i< j} \log\bigl|w_{j}^2
- w_{i}^2\bigr|
\\
&& \hspace*{11pt}{}  - \biggl( \frac{1}d-1 +
\frac{1-\beta^{-1}}{N} \biggr) \sum_{i=1}^N
\log|w_i| \Biggr]\nonumber
\end{eqnarray}
denote the joint distribution of the singular values of $X $ when the
matrix $X $ has independent Gaussian entries (so that ${X}^\dagger X $
is a~Wishart random matrix). In (\ref{H}), the constant $\beta$ takes
values $\{1,2\}$ with $\beta=2$ for complex entries and $\beta=1$ for
real valued entries. Also, $Z_\beta$ is the normalization constant so
that $\mu$ is a~probability measure.
%In this section, we often use the notation $x_j$ instead of $
%the singular values to follow the notations of \cite{ESYY}.
Denote the distribution of the singular values at time $t$ by $f_t
({\mathbf w})\mu(\mathrm{d}{\mathbf w})$.
Then $f_t$ satisfies%
%
%e5.4 #&#
\begin{equation}
\label{dy} \partial_{t} f_t =
\mathcal{L}^Wf_t,
\end{equation}
where%
%
%e5.5 #&#
\begin{eqnarray}
\label{LW2}
\mathcal{L}^W&=& L_{\beta,N}^W=
\sum_{i=1}^N \frac{1}{2N}
\partial_{i}^{2}
+\sum_{i=1}^N
\biggl(- \frac{\beta
w_{i}}{2d} + \frac{\beta}{N}\sum
_{j\ne i} \frac{w_i}{w_i^2 - w_j^2}
\nonumber\\[-8pt]\\[-8pt]
&&\hspace*{112pt}{} + \frac{1}{2} \biggl(\beta
\biggl(\frac{1}d-1 \biggr) + \frac{\beta-1}{N} \biggr)
\frac{1}{w_j} \biggr) \partial_{i}.\nonumber
\end{eqnarray}
For any $n\ge1$ we define the $n$-point correlation functions
(marginals) of the probability measure
$f_t\,\mathrm{d}\mu$ by%
%
%e5.6 #&#
\begin{equation}
p^{(n)}_{t,N}(w_1, w_2, \ldots,
w_n) = \int_{\mathbb{R}^{N-n}} f_t(\mathbf{w})
\mu(\mathbf{w}) \,\mathrm{d}w_{n+1}\cdots\mathrm{d}w_N.
\label{corr}
\end{equation}
With a~slight abuse of notation, we will sometimes also use $\mu$ to
denote the density of the measure $\mu$ with respect to the Lebesgue
measure. The correlation functions of the equilibrium measure are
denoted by
%
%
%e5.7 #&#
\begin{equation}
p^{(n)}_{\mu,N}(w_1, w_2, \ldots,
w_n) = \int_{\mathbb{R}^{N-n}} \mu(\mathbf{w})
\,\mathrm{d}w_{n+1} \cdots\mathrm{d}w_N. \label{correq}
\end{equation}

Now we are ready to prove the \emph{strong local ergodicity of the
Dyson Brownian motion} which states that the correlation functions of
the Dyson Browian motion $p^{(n)}_{t,N}$ and those of the equilibrium
measure $p^{(n)}_{\mu,N}$ are close:
%
%
%th5.1 #&#
\begin{theorem} \label{thmmainfk}
%[Strong local ergodicity of Dyson Brownian motion]
Let $X = [x_{ij}]$ with entries $x_{ij}$ satisfying (\ref{eqnXmat}) and
(\ref{eqnXmatexpbd}). Let $E\in[\lambda_- +r, \lambda_+ - r] $ with
some $ r> 0$. Then for any $\varepsilon'> 0$, $\delta>0$, $0 < b=b_N<
r/2$, any integer $n\ge1$ and for any compactly supported continuous
test
function $O\dvtx\mathbb{R}^n\to\mathbb{R}$ we have%
%
%e5.8 #&#
\begin{eqnarray}
\label{abstrthm2}
&& \sup_{t\ge N^{-1 + \delta+ \varepsilon'}} \Biggl| \int_{E-b}^{E+b}
\frac{\mathrm{d}
E'}{2b} \int_{\mathbb{R}^n} \,\mathrm{d}\alpha_1
\cdots \mathrm{d}\alpha_n O(\alpha_1,\ldots,
\alpha_n) \frac{1}{\varrho_c(E)^n}\bigl( p_{t,N}^{(n)} - p_{\mu, N}
^{(n)} \bigr)\hspace*{-26pt}\nonumber
\\
&&\hspace*{152pt}{} \times  \biggl(E'+\frac{\alpha_1}{N\varrho_c(E)}, \ldots,
E'+\frac{\alpha_n}{ N\varrho_c(E)} \biggr) \Biggr|\hspace*{-26pt}
\\
&&\qquad  \leq C_n N^{2 \varepsilon'} \bigl[ b^{-1} N^{ - 1+\varepsilon' } + b^{-1/2}
N^{-\delta/2} \bigr],\nonumber
\end{eqnarray}
where $ p_{t,N}^{(n)}$ and $ p_{\mu,N}^{(n)}$, (\ref{corr}) and
(\ref{correq}),\vspace*{1pt} are the correlation functions of the
eigenvalues of the Dyson Brownian motion flow (\ref{matrixdbm}) and
those of the equilibrium measure, respectively, and $C_n$ is
a~constant.
\end{theorem}

%
%
%re5.2 #&#
\begin{remark}
Notice that if we choose $\delta= 1- 2 \varepsilon'$ and thus $t = N^{-
\varepsilon'}$, then we can set $b \sim N^{-1+8\varepsilon'}$ so that
the right-hand side of (\ref{abstrthm2}) vanishes as $N
\rightarrow\infty$. From the MP law we know that the spacing of the
eigenvalues in the bulk is $O(N^{-1})$ and thus we see that Theorem
\ref{thmmainfk} yields universality with almost no averaging in~$E$.
\end{remark}

\begin{pf*}{Proof of Theorem \ref{thmmainfk}}
The proof follows from the main result in \cite{ESYY} (Theorem 2.1)
which states that the local ergodicity of Dyson Brownian
motion~(\ref{abstrthm2}) holds for $t \ge N^{-2\frak{a}+\delta}$ for
any $\delta> 0$
provided that there exists an $\frak{a}>0$ such that%
%
%e5.9 #&#
\begin{equation}
\sup_{ t \ge N^{-2\frak{a}}} \frac{1}{N} \mathbb{E}\sum
_{j=1}^N \bigl(\lambda_{j} (t)-
\gamma_{j} \bigr)^2 \le CN^{-1-2\frak{a}} \label{assum3}
\end{equation}
holds with a~constant $C$ uniformly in $N$. Here $\sqrt{\lambda_j(t)}$
is the singular value of the matrix $X_t$ given in (\ref{matrixdbm}).
%$\E_t$ is the expectation w.r.t. Dyson Brownian motion at the time
%$t$.
Condition (\ref{assum3}) is a~simple consequence of (\ref{resrig}) as
long as $\frak{a}<1/2$.

Strictly speaking, there are four assumptions in the hypothesis of
Theorem 2.1 in \cite{ESYY}. Assumptions I~and~II of Theorem 2.1 in
\cite{ESYY} are automatically satisfied in the setting that the Dyson
Brownian motion is generated by flows on the covariance matrix
ensembles. Assumption~IV of Theorem of \cite{ESYY} states that the
local density of the singular values of $X_t$ in the scale larger than
$N^{-1+c}$ for any $c>0$, is bounded above by a~constant. As in
\cite{ESYY} this follows from the large deviation estimate
(\ref{Lambdafinal}) since a~bound on $\Im m(z)$, $z=E+i\eta$, can be
easily used to prove an upper bound on the local density of eigenvalues
in a~window of size $\eta$ about $E$. As usual, the additional
condition in \cite{ESYY} on the entropy $S_\mu(f_{t_0})\le CN^m$ for
some constant $m$ for $t_0= N^{-2\frak{a}}$, holds due to the
regularization property of the Ornstein--Uhlenbeck process. Thus for
a~given $0<\varepsilon'<1$, choosing $\frak{a}= 1/2 - \varepsilon'/2, A
= \varepsilon'$ in the second part of Theorem 2.1 in \cite{ESYY} and
using (\ref{resrig}), we obtain (\ref{assum3}) and the proof is
finished.
\end{pf*}

For any $\varepsilon>0$, applying Theorem \ref{thmmainfk} with
$\delta=1-2\varepsilon$, $\varepsilon'=\varepsilon$ and
\mbox{$b=-1+8\varepsilon$,} we obtain universality for all ensembles
with the matrix elements distributed according to $ M^{-1/2} \xi_t $
with
%
%
%e5.10 #&#
\begin{equation}
\label{matrixdbmfake} \xi_t = e^{-t/2} \xi_0 +
\bigl(1-e^{-t} \bigr)^{1/2} \xi_G,
\end{equation}
where the matrix $\xi_G $ has independent Gaussian random variables
with mean $0$ and variance $1 $, $t\sim N^{-\varepsilon}$, and the
initial condition $\xi_0$ has entries satisfying our conditions
(\ref{eqnXmat}) and (\ref{eqnXmatexpbd}). In other words, for $t \sim
N^{-\varepsilon}$ the random matrices $\xi_t$ which are distributed
according to (\ref{matrixdbmfake}) have the same correlation functions
as that of the matrix with Gaussian entries, averaged on a~length of
$O(N^{-1+8\varepsilon})$. Thus in order to prove Theorem \ref{thmmain},
it remains to find a~random matrix $\widetilde\xi_t$ of the form
(\ref{matrixdbmfake}) (with time $t = N^{-\varepsilon}$) whose
eigenvalue correlation functions well approximate that of the spectrum
of the given matrix $X$ satisfying (\ref{eqnXmat}) and
(\ref{eqnXmatexpbd}).
%Theorem \ref{thmmainfk} yields
%universality for all ensembles with the matrix elements distributed
%according to $ M^{-1/2} \xi_t $ with
%where $\xi_G $ are independent Gaussian random variables with mean
%$0$ and variance $1 $ and $t\sim N^{-\varepsilon}$ and the initial
%condition $\xi_0$
%has entries satisfying our conditions \eqref{eqnXmat} and
%In other words, for $t \sim N^{-\varepsilon}$ the random matrices $
%distributed according to \eqref{matrixdbmfake} have the same
%correlation functions
%as that of the Wishart matrix.
%Thus in order to prove Theorem \ref{thmmain},
%it remains to find a~random matrix $\wt\xi_t$ of the form
%whose eigenvalue correlation functions well approximate that of the
%spectrum of the given matrix $X$ satisfying \eqref{eqnXmat} and

The requirements on entries of the matrix $\widetilde{\xi}_t$ are just
mean zero, variance one and subexponential decay; however, it turns out
that for any fixed $X$ and $\varepsilon$, one may find a~$\widetilde
\xi_0$ such that $\widetilde\xi_t$ satisfies (\ref{matrixdbmfake}),
with $t \sim N^{-\varepsilon}$, and the entries
$[\widetilde\xi_t]_{ij}$ have mean $0$, variance $1$ and the
\emph{same} third moment as those of the (rescaled) initial condition
$\sqrt{M} X$. Moreover $\widetilde\xi _t$ can be chosen in such a~way
so that its entries have fourth moment very close to those of $X$. More
precisely,
% Let $M_k(U)$ denote the $k$th moment of a~random variable $U$. It is
%easy to show that
%M_4(U)-M_3^2(U)-1\ge0
% if $U$ has mean zero, variance one.
Lemma~3.4 in \cite{EYYgenwig} yields that for any given matrix $X$
satisfying (\ref{eqnXmat}) and (\ref{eqnXmatexpbd}) and $t\sim
N^{-\varepsilon}$, there exists a~matrix $\widetilde\xi_t $ of the form
(\ref{matrixdbmfake}) such that for $1\leq k\leq3$,
\[
\mathbb{E} \sqrt M x_{ij}^k=\mathbb{E}[\widetilde
\xi_t]^k_{ij}, \qquad \bigl|\mathbb{E} (\sqrt
M x_{ij})^4-\mathbb {E}[\widetilde \xi_t]^4_{ij}
\bigr|\leq Ct\sim N^{-\varepsilon}.
\]

Now to finish the proof of Theorem \ref{thmmain}, it remains only to
show that that the correlation functions of the eigenvalues of two
matrix ensembles at a~fixed energy [i.e., for a~fixed value of $E =
\Re(z)$] are identical up to the scale $1/N$ provided that the first
four moments of the matrix elements of these two ensembles are almost
identical in above sense. To achieve this, as shown for the Wigner
matrices~\cite{EYYBulkuni} (see Sections 8.6--8.13 of
\cite{EYYBulkuni}), it is enough to show that the corresponding Green
functions are close for these two matrix ensembles. This is the content
of the following theorem which we call, following \cite{EYYBulkuni},
the Green function comparison theorem.

Recall the matrices $X^{\mathbf{v}}, X^{\mathbf{w}}, H^{\mathbf{v}},
H^{\mathbf{w}}$ and the Green functions $G^{\mathbf{v}},
G^{\mathbf{w}}$ from Section~\ref{proUE}.
%
%
%th5.3 #&#
\begin{theorem}\label{comparison}
Assume that the first three moments of $x^{\mathbf{v}}_{ij}$ and
$x^{\mathbf{w}}_{ij}$ are identical, that is,
\[
\mathbb{E} \bigl(x^{\mathbf{v}}_{ij} \bigr)^{u} =
\mathbb{E} \bigl(x^{\mathbf{w}}_{ij} \bigr)^{u}, \qquad0
\le u\le3
\]
and the difference between the fourth moments of $x^{\mathbf{v}}_{ij}$
and $x^{\mathbf{w}}_{ij}$ is much less than 1, say
%
%
%e5.11 #&#
\begin{equation}
\label{4match} \bigl|\mathbb{E} \bigl( \sqrt M x^{\mathbf{v}}_{ij}
\bigr)^{4 }- \mathbb{E} \bigl( \sqrt M x^{\mathbf{w}}_{ij}
\bigr)^{4 } \bigr|\leq N^{-\delta}
\end{equation}
for some given $\delta>0$. Let $\varepsilon>0$ be arbitrary, and choose
an $\eta$ with $N^{-1-\varepsilon}\le\eta\le N^{-1}$. For any sequence
of positive integers $k_1, \ldots, k_n$, set complex parameters
\[
z^m_j = E^m_j \pm i \eta,
\qquad j = 1, \ldots, k_i, \qquad m = 1, \ldots, n,
\]
with an arbitrary choice of the $\pm$ signs %\sn{why do we care about
%negative values of $\eta$?}
and $ \lambda_{-} + \kappa\le|E^m_j| \le\lambda_+ - \kappa$ for some
\mbox{$\kappa>0$.} Let $F(x_1, \ldots, x_n)$ be a~function such that
for any multi-index $\alpha=(\alpha_1, \ldots,\alpha_n)$ with $1\le
|\alpha| = \sum|\alpha_i| \le5$ and for any $\varepsilon'>0$
sufficiently small, we have
%
%
%e5.12 #&#
%e5.13 #&#
\begin{eqnarray}
\max \Bigl\{\bigl|\partial^{\alpha}F(x_1, \ldots,
x_n)\bigr|\dvtx \max_j|x_j|\leq
N^{\varepsilon'} \Bigr\} &\leq& N^{C_0\varepsilon'}, \label{lowder}
\\
\max \Bigl\{\bigl|\partial^\alpha F(x_1, \ldots,
x_n)\bigr|\dvtx \max_j|x_j|\leq
N^2 \Bigr\} &\leq& N^{C_0}\label{highder}
\end{eqnarray}
for some constant $C_0$.

Then there is a~constant $C_1$, depending on $\alpha$, $\sum_i k_i$ and
$C_0$ such that for any $\eta$ with $N^{-1-\varepsilon}\le\eta \le
N^{-1}$ and for any choices of the signs in the imaginary part
of~$z^m_j$,
%
%
%e5.14 #&#
\begin{eqnarray}
\label{eqnmaincompbulk}
&& \Biggl|\mathbb{E}F \Biggl( \frac
{1}{N^{k_1}}\operatorname{Tr}
\Biggl[\prod_{j=1}^{k_1} G^{\mathbf{v}}
\bigl(z^1_{j} \bigr) \Biggr], \ldots, \frac{1}{N^{k_n}}
\operatorname{Tr} \Biggl[ \prod_{j=1}^{k_n}
G^{\mathbf{v}} \bigl(z^n_{j} \bigr) \Biggr] \Biggr)\,{-}\,
\mathbb{E}F \bigl( G^{\mathbf{v}} \,{\to}\, G^{\mathbf{w}} \bigr) \Biggr|\hspace*{-31pt}
\nonumber\\[-4pt]\\[-12pt]
&&\qquad\le C_1 N^{-1/2 + C_1 \varepsilon}+C_1
N^{-\delta+ C_1
\varepsilon},\nonumber
\end{eqnarray}
where in the second term the arguments of $F$ are changed from the
Green functions of $H^{\mathbf{v}}$ to $H^{\mathbf{w}}$, and all other
parameters remain unchanged.
\end{theorem}

Once again we note the equivalence of (\ref{abstrthm2}) and
(\ref{eqnmaincompbulk}) as discussed in \cite{EYYBulkuni}
(Sections~8.6--8.13). The only difference is that in \cite{EYYBulkuni},
the equivalence is proved for Wigner matrices, but the arguments are
easily adapted for covariance matrices. Thus to complete the proof of
Theorem \ref{thmmain}, all that remains is Theorem \ref{comparison}
which is proved below.

\begin{pf*}{Proof of Theorem \ref{comparison}}
The proof is very similar to Lemma~2.3 of \cite{EYYBulkuni}. The only
differences are a~few simple linear algebraic identities. Therefore, we
will only prove the simple case of $k=1$ and $n=1$.

Fix a~bijective ordering map on the index set of the independent matrix
elements,
\[
\phi\dvtx \bigl\{(i, j)\dvtx 1\le i\le M, 1\leq j \le N \bigr\} \to\{1, \ldots,
MN \} %\qquad\gamma(N): =MN
\]
and define the family of random matrices $X_\gamma$, $0 \leq\gamma \leq
MN$,
\begin{eqnarray*}
{[X_\gamma]}_{ij}&=& \bigl[X^{\mathbf{v}}
\bigr]_{ij},\qquad\phi(i,j) > \gamma,
\\
&=& \bigl[X^{\mathbf{w}} \bigr]_{ij}, \qquad\phi(i,j)\le\gamma.
\end{eqnarray*}
%
%the random matrix whose matrix elements $x_{ij}$ follow
%the $v$-distribution if $\phi(i,j)\le\gamma$ and they follow the
%$w$-distribution
%otherwise;
In particular we have $ X_0 = X^{\mathbf{v}}$ and $ X_{MN} =
X^{\mathbf{w}}$. Denote
$H_\gamma$, $G_\gamma$ and $\mathcal{G}_\gamma$ as%
\[
H_\gamma=X^\dagger_\gamma X _\gamma, \qquad
G_\gamma=(H_\gamma-z)^{-1},\qquad
\mathcal{G}_\gamma= \bigl(X_\gamma X^\dagger_\gamma-z
\bigr)^{-1}.
\]
First, using the delocalization result (\ref{4444}) and the rigidity of
eigenvalues (\ref{resrig}), it is easy to have the following estimate
on the matrix elements of the resolvent:
%
%
%e5.15 #&#
\begin{equation}
\label{basic4} \max_{ \gamma} \max_{ k, l } \max
_{\eta\geq N^{-1-\varepsilon}}\ \max_{\kappa\geq c } \bigl|
\bigl[G_\gamma(z) \bigr]_{k l } \bigr|+ \bigl| \bigl[
\mathcal{G}_\gamma(z) \bigr]_{k l } \bigr|\le N^{C\varepsilon}
\end{equation}
with $\zeta$-high probability for any $\zeta>0$. For instance, for
$\gamma= 0$, we have the identity $ G_0(z) = \sum_{\alpha=1}^N {
\mathbf{v}_{\alpha}^\dagger {\mathbf{v}}_{\alpha} \over
\lambda_\alpha- z} $ where $\lambda_\alpha, {\mathbf{v}}_\alpha$ are
the eigenvalues and eigenvectors of $H_0$. By the delocalization result
(\ref{4444}), we obtain%
\[
\bigl|G_0(z)\bigr| \leq{\varphi^{C_\zeta} \over N} \sum
_{\alpha=1}^N {1 \over|\lambda_\alpha
- z|}.
\]
We write the above sum as%
%
%e5.16 #&#
\begin{equation}
\label{eqnrigidcalcest} \sum_\alpha{1 \over|\lambda_\alpha- z|}
= \sum_k\sum_{\alpha\in I_k}
{1 \over|\lambda_\alpha- z|} \leq\sum_k
|I_k| {1 \over|\lambda_\alpha- E|+\eta},
\end{equation}
where $I_k$ is the set of all $\alpha$ such that
\[
{N^{-1}2^{K-1} \leq(\lambda_\alpha-E )\leq
N^{-1}2^K}.
\]
By the rigidity of eigenvalues we obtain that $|I_K|\leq C 2^K$ with
$\zeta$-high probability. Substituting this bound in (\ref
{eqnrigidcalcest}) yields the estimate (\ref{basic4}).

Recall that $\mathbf{x}_i$ denotes the $i$th column of $X$. For $1\leq
i\leq N$, using straightforward algebra, it is easy to check that
%$(\mG^{i})^{-1}= (\mG)^{-1}-x_i^\dagger x_i$
%
%
%e5.17 #&#
\begin{equation}
\label{mGi} \qquad \mathcal{G}^{(i)}_{kl} =
\mathcal{G}_{kl}+ \frac{(\mathcal{G}\mathbf{x}_i )_k
(\mathbf{x}_i^\dagger\mathcal{G})_l}{1- \langle\mathbf
{x}_i,\mathcal{G}(z) \mathbf{x}_i\rangle},\qquad\mathcal{G}_{kl}
=\mathcal{G}^{(i)}_{kl}- \frac{(\mathcal{G}^{(i)} \mathbf{x}_i)_k( \mathbf{x}_i
^\dagger\mathcal{G}^{(i)})_l}{1+ \langle\mathbf{x}_i, \mathcal
{G}^{(i)} (z) \mathbf{x}_i\rangle}.
\end{equation}
From (\ref{eqnGii}) we obtain
%
%
%e5.18 #&#
%e5.19 #&#
\begin{eqnarray}
\label{mGii} \bigl\langle\mathbf{x}_i,\mathcal{G}^{(i)}
(z) \mathbf{x}_i \bigr\rangle&=&-1+\frac{-1}{z G_{ii}}, \qquad \bigl
\langle \mathbf{x}_i,\mathcal{G}(z) \mathbf{x}_i \bigr
\rangle=1+z G_{ii},
\\
\mathcal{G}\mathbf{x}_i &=&\frac{\mathcal{G}^{(i)} \mathbf{x}_i } {
1+ \langle\mathbf{x}_i, \mathcal{G}^{(i)} (z)
\mathbf{x}_i\rangle} =-z G_{ii}
\mathcal{G}^{(i)} \mathbf{x}_i.\label{98}
\end{eqnarray}
Furthermore, from (\ref{eqnGij}) it follows that
\begin{eqnarray*}
\bigl\langle\mathbf{x}_i,\mathcal{G}^{(i )}
\mathbf{x}_j \bigr\rangle&=& \bigl\langle\mathbf{x}_i,
\mathcal{G}^{(i j)} \mathbf{x}_j \bigr\rangle-
\frac{\langle\mathbf{x}_i, \mathcal{G}^{(ij)}
\mathbf{x}_j\rangle\langle\mathbf{x}_j,\mathcal{G}
^{(ij)}\mathbf{x}_j\rangle}{1+ \langle\mathbf{x}_j,\mathcal
{G}^{(ij)} \mathbf{x}_j\rangle}
\\
&=& \frac{\langle\mathbf{x}_i,\mathcal{G}^{(ij)} \mathbf
{x}_j\rangle}{1+ \langle\mathbf{x}_j,\mathcal{G}
^{(ij)} \mathbf{x}_j\rangle} =-z G_{jj}^{(i)} \bigl\langle
\mathbf{x}_i,\mathcal{G}^{(ij)} \mathbf{x}_j
\bigr\rangle=- \frac{ G_{ij}
}{G_{ii}}.
\end{eqnarray*}
Similarly
%
%
%e5.20 #&#
\begin{equation}
\label{910} \langle\mathbf{x}_i,\mathcal{G}\mathbf{x}_j
\rangle%=
%= \frac{\langle\bx_i|\mG^{(i )} |\bx_j\rangle}{1+ \langle\bx_i |
=-z G_{ii} \bigl\langle\mathbf{x}_i, \mathcal{G}^{(i )}
\mathbf{x}_j \bigr\rangle= zG_{ij},
\end{equation}
which implies that
%
%
%e5.21 #&#
\begin{equation}
\label{iGj} \bigl\langle\mathbf{x}_i, \mathcal{G}^{(i)}
\mathbf{x}_j \bigr\rangle=\frac{ G_{ij} }{G_{ii}}, \qquad\langle
\mathbf{x}_i,\mathcal{G}\mathbf{x}_j
\rangle=-zG_{ij}.
\end{equation}
Let $x_i$ be the $i$th row of $X$. By symmetry, the above identities
also hold if one switches $\{G,\mathbf{x}_i\}$ and $\{\mathcal
{G},x_i\}$.

Combining the above identities with (\ref{basic4}), we obtain the bound
%
%
%e5.22 #&#
\begin{eqnarray}
\label{basic5}
\quad && \max_{ \gamma} \max_{ k,l } \max
_{\eta\geq N^{-1-\varepsilon}}\ \max_{\kappa\ge c} \bigl|
\bigl[G_\gamma(z) \bigr]_{k l } \bigr|+ \bigl|
\bigl[X_\gamma G_\gamma(z) \bigr]_{k l } \bigr|+ \bigl|
\bigl[ G_\gamma X_\gamma^\dagger(z)
\bigr]_{k l } \bigr|
\nonumber\\[-8pt]\\[-8pt]
&&\qquad{} + \bigl| \bigl[X_\gamma G_\gamma
X_\gamma^\dagger(z) \bigr]_{k l } \bigr|\le
N^{C\varepsilon},\nonumber
\end{eqnarray}
with $\zeta$-high probability.

Consider the telescopic sum of differences of expectations
%
%
%e5.23 #&#
\begin{eqnarray}
\label{tel} &&\mathbb{E}F \biggl( \frac{1}{N}\operatorname {Tr}
\frac{1} {
H^{\mathbf{w}}-z} \biggr) - \mathbb{E}F \biggl( \frac{1}{N}
\operatorname{Tr}\frac{1} {
H^{\mathbf{v}}-z} \biggr)
\nonumber\\[-8pt]\\[-8pt]
&&\qquad= \sum_{\gamma=1}^{MN} \biggl[
\mathbb{E}F \biggl( \frac{1}{N}\operatorname{Tr}\frac{1} {
H_\gamma-z} \biggr) -
\mathbb{E}F \biggl( \frac{1}{N}\operatorname{Tr} \frac{1} {
H_{\gamma-1}-z} \biggr)
\biggr].\nonumber
\end{eqnarray}
Let $E^{(ij)}$ denote the matrix whose matrix elements are zero
everywhere except at the $(i,j)$ position, where it is 1, that is,
$E^{(ij)}_{k\ell}=\delta_{ik}\delta_{j\ell}$. Fix a~$\gamma\ge1$, and
let $(i,j)$ be determined by $\phi(i, j) = \gamma$. We will compare
$H_{\gamma-1}$ with $H_\gamma$. Note that these two matrices differ
only in the $(i,j)$ matrix element, and they can be written as
\[
X_{\gamma-1} = Q + V, \qquad V:= x^{\mathbf{v}}_{ij}
E^{(ij)}, \qquad X_\gamma= Q + W, \qquad W:=
x^{\mathbf{w}}_{ij} E^{(ij)}
\]
with a~matrix $Q$ that has zero matrix element at the $(i,j)$ position.
Define the Green functions
\[
R = \frac{1}{Q^\dagger Q-z}, \qquad S= \frac{1}{H_{\gamma-1}-z}, \qquad T=
\frac{1}{H_{\gamma}-z}.
\]
The following lemma is at the heart of the Green function comparison
first established in \cite{EYYBulkuni} (subsequently used in
\cite{EYYgenwig,EYYrigid,EKYY12}) which states that the difference of
smooth functionals of Green functions of two matrices which differ by
a~single entry can be bounded above as a~function of its first four
moments.
\end{pf*}
%

%
%le5.4 #&#
\begin{lemma} \label{lemGCheart}
%Now we are going to prove the following result.
Let $m_k$ be the $k$th moment of $ \sqrt M x^{\mathbf{v}}_{ij}$, then
%
%
%e5.24 #&#
\begin{eqnarray}
\label{ygz1} % \E\left[ F \left( \frac{1}{N}\tr\frac1 { H_{\gamma-1}-z} \right)
&& \mathbb{E} \biggl[ F \biggl( \frac{1}{N}
\operatorname{Tr}S \biggr) - F \biggl( \frac{1}{N}\operatorname{Tr}R \biggr)
\biggr]
\nonumber\\[-8pt]\\[-8pt]
&&\qquad = A(Q, m_1, m_2, m_3)+N^{-5/2+C\varepsilon}+
\widetilde A(Q)m_4\nonumber
\end{eqnarray}
for a~functional $A(Q, m_1, m_2, m_3)$ which depends only on the
distribution of $Q$ and $m_1, m_2, m_3$. The constant $\widetilde A(Q)$
depends only on the
distribution of $Q$ and satisfies the bound%
\[
\bigl|\widetilde A(Q)\bigr|\leq N^{-2+C\varepsilon}.
\]
\end{lemma}

Before giving the proof of Lemma \ref{lemGCheart}, let us use it to
conclude the foregoing argument in the proof of Theorem
\ref{comparison}. Note that the matrices $H_\gamma$ and $Q$ also differ
by one entry, and therefore applying Lemma \ref{lemGCheart} yields
%
%
%e5.25 #&#
\begin{eqnarray}
\label{ygz2} % \E F \left( \frac{1}{N}\tr\frac1 { H_\gamma-z} \right)
&& \mathbb{E} \biggl[ F \biggl( \frac{1}{N}
\operatorname{Tr}T \biggr) - F \biggl( \frac{1}{N}\operatorname{Tr}R \biggr)
\biggr]
\nonumber\\[-8pt]\\[-8pt]
&&\qquad = A(Q, m_1, m_2, m_3)+N^{-5/2+C\varepsilon}+
\widetilde A(Q)m'_4,\nonumber
\end{eqnarray}
where $m'_4$ is the fourth moment of $ \sqrt M x^{\mathbf{w}}_{ij}$ (by
hypothesis, the first three moments of $x^{\mathbf{w}}_{ij}$ are
identical to those of $x^{\mathbf {v}}_{ij}$). Since $|m'_4-m _4|\leq
N^{-\delta}$ by hypothesis, we have
\[
\mathbb{E}F \biggl( \frac{1}{N}\operatorname{Tr}\frac{1} {
H_\gamma-z} \biggr)
- \mathbb{E}F \biggl( \frac{1}{N}\operatorname{Tr}\frac{1} { H_{\gamma-1}-z}
\biggr) \le C N^{-5/2 + C \varepsilon}+C N^{-2 -\delta+ C \varepsilon}.
\]
Using the above estimate and summation over $\gamma$ yields [see
(\ref{tel})]
\[
\mathbb{E}F \biggl( \frac{1}{N}\operatorname{Tr}\frac{1} {
H^{\mathbf{v}} - z} \biggr)
- \mathbb{E}F \biggl( \frac{1}{N}\operatorname{Tr}\frac{1} { H^{\mathbf{w}} - z}
\biggr)\le C N^{-1/2+ C \varepsilon}+C N^{-\delta+ C \varepsilon},
\]
obtaining precisely what we set out to show in (\ref{eqnmaincompbulk}).
The proof can be easily generalized to functions of several variables.
Thus to conclude the proof of Theorem~\ref{comparison}, we just need to
give the proof of Lemma~\ref{lemGCheart}.

\begin{pf*}{Proof of Lemma \ref{lemGCheart}}
We first claim that the estimate (\ref{basic4}) holds for the Green
function $R$ as well. To see this, from the resolvent expansion we
obtain
\begin{eqnarray*}
R &=& S + S \bigl(V^\dagger X+X^\dagger V+V^\dagger V \bigr)S +
\cdots+ \bigl[ S \bigl(V^\dagger X+X^\dagger V+V^\dagger V \bigr)
\bigr]^9S
\\
&&{} + \bigl[ S \bigl(V^\dagger X+X^\dagger V+V^\dagger V \bigr)
\bigr]^{10} R.
\end{eqnarray*}
Since the matrix $V$ has only at most one nonzero entry, when computing
the $(k,\ell)$ matrix element of the matrix identity above, each term
is a~finite sum involving matrix elements of $S$, $XS$, $SX^\dagger$,
$XSX^\dagger$ or $R$ (only for the last term) and $x^{\mathbf
{v}}_{ij}$. Using the bound (\ref{basic5}) for the $S$ matrix elements,
the subexponential decay for $x^{\mathbf{v}}_{ij}$ and the trivial
bound $|R_{ij}| \le\eta^{-1}$, we obtain that the estimate
(\ref{basic4}) holds for~$R$. Similarly by expanding $XR$, $RX$ and
$\mathit{XRX}$, we can obtain (\ref{basic5}) for $XR$, $RX$ and
$\mathit{XRX}$, $QR$, $RQ$ and $QRQ$.

Now we prove (\ref{ygz1}). By the resolvent expansion,
%
%
%e5.26 #&#
\begin{eqnarray}
\label{SR-r} S &=& R - R \bigl(V^\dagger Q+Q^\dagger
V+V^\dagger V \bigr)R + \cdots
\nonumber\\[-8pt]\\[-8pt]
&&{} - \bigl[ R \bigl(V^\dagger
Q+Q^\dagger V+V^\dagger V \bigr) \bigr]^9R+O
\bigl(N^{-4} \bigr)\nonumber
\end{eqnarray}
holds with extremely high probability. Thus we may write%
\[
\frac{1}{N} \operatorname{Tr}S = \frac{1}{N} \operatorname {Tr}R+
\sum_{k \leq20} y_k+O \bigl(N^{-4}
\bigr), %
\]
where $y_k$ is the sum of the terms in (\ref{SR-r}), in which there are
exactly $k$ $V$'s. Recall that $m_k$ is the $k$th moment of $\sqrt M
x_{ij}$, which is $O(1)$ if $k=O(1)$. The terms $y_k$ satisfy the bound
[with ${ K}=(k_1, k_2, \ldots,k_n )$ and $|{ K}|:=\sum_ik_i$]
%
%e5.27 #&#
\begin{eqnarray}\label{52666}
|y_k| &\leq& N^{C\varepsilon}N^{-k/2},\nonumber
\\
\mathbb{E}_{\mathbf{v}} y_{k_1}y_{k_2}\cdots
y_{k_n} &=&N^{-|{ K}|/2} m_{| K|} z_K(Q),
\\
\bigl|z_K(Q)\bigr|&\leq& N^{C\varepsilon}\nonumber
\end{eqnarray}
for some $z_K(Q)$ depending only on the distribution $Q$, and the last
inequality holds with $\zeta$-high probability. Here $\mathbb
{E}_{\mathbf{v}} $ is the expectation value with respect to the
distribution of the entries of the matrix $X^{ \mathbf{v}}$. Then we
have
%
%
%e5.28 #&#
\begin{eqnarray}
\label{temp66}
&& \mathbb{E}F \biggl(\frac{1}{N}\operatorname{Tr}
\frac{1} {
H_{\gamma-1}-z} \biggr)
\nonumber\\[-8pt]\\[-8pt]
&&\qquad = \mathbb{E}\sum_{n=0}^4
\frac{1}{n!}F^{(n)} \biggl( \frac{1}{N} \operatorname{Tr}R
\biggr) \biggl(\sum_{k\leq20} y_k
\biggr)^n+O \bigl(N^{-5/2+C\varepsilon} \bigr).\nonumber
\end{eqnarray}
From (\ref{52666}) we obtain%
\begin{eqnarray*}
&& \mathbb{E}F \biggl(\frac{1}{N}\operatorname{Tr}\frac{1} {
H_{\gamma-1}-z} \biggr)
\\
&&\qquad = \mathbb{E}\sum_{n=0}^4
\frac{1}{n!}F^{(n)} \biggl( \frac{1}{N} \operatorname{Tr}R
\biggr) \biggl(\sum_{k_1, \ldots,
k_ n } N^{-|K|/2}
m_{|K|} z_K(Q) \biggr) +O \bigl(N^{-5/2+C\varepsilon} \bigr)
\\
&&\qquad =B+O \bigl(N^{-5/2+C\varepsilon} \bigr)+A(Q, m_1, m_2,
m_3)+\widetilde A(Q)m_4,
\end{eqnarray*}
where $A(Q, m_1, m_2, m_3)$ depends only on the distribution of
$Q$ and $m_1, m_2, m_3$ and%
\begin{eqnarray*}
B &=&\mathbb{E}\sum_{n=0}^4
\frac{1}{n!}F^{(n)} \biggl( \frac{1}{N} \operatorname{Tr}R
\biggr) \biggl(\sum_{k_1,
\ldots, k_ n\dvtx|K|\geq5, k_i\leq20} N^{-|K|/2}
m_{|K|} z_K(Q) \biggr),
\\
\widetilde A(Q) &=&\mathbb{E}\sum_{n=0}^4
\frac{1}{n!}F^{(n)} \biggl( \frac{1}{N} \operatorname{Tr}R
\biggr) \biggl( \sum_{k_1, \ldots, k_ n\dvtx {|K|=4}} N^{-2}
z_K(Q) \biggr).
\end{eqnarray*}
In the above $K = \sum_{i}k_i$. Now
it remains only to prove%
\[
|B|\leq O \bigl(N^{-5/2+C\varepsilon} \bigr), \qquad\widetilde A(Q) \leq O
\bigl(N^{-2+C\varepsilon} \bigr).
\]
Using the estimate (\ref{basic5}) for $R$ and the derivative bounds
(\ref{lowder}) for the typical values of $ \frac{1}{N} \operatorname
{Tr}R$, we see that $F^{(n)} ( \frac{1}{N} \operatorname{Tr}R )$
($n\leq4$) are bounded by $N^{C\varepsilon}$ with \mbox{$\zeta$-}high
probability. Similarly $z_K$ ($k_i\leq20$) is also bounded by
$N^{C\varepsilon}$ for some $C>0$ with $\zeta$-high probability. Now we
define $\Xi_g$ as the good set where these quantities are bounded by
$N^{C\varepsilon}$. Furthermore, using (\ref{highder}) and the
definition of $z_K$, we know that $F^{(n)} ( \frac{1}{N} \operatorname
{Tr}R )$ and $z_K$ are bounded by $N^C$ for some $C>0$ in $\Xi_g^c$.
Since $\Xi_g^c$ has a~very small probability by
(\ref{basic5}), we have%
\[
\widetilde A(Q) = \mathbb{E}_{\Xi_g} \sum_{n=0}^4
\frac{1}{n!}F^{(n)} \biggl( \frac{1}{N} \operatorname{Tr}R
\biggr) \biggl( \sum_{k_1, \ldots, k_ n\dvtx {|K|=4}} N^{-2}
z_K(Q) \biggr)+O \bigl(N^{-5/2+C\varepsilon} \bigr).
\]
Then with the bounds on $F^{(n)}$ and $z_K$ in $\Xi_g$, we obtain
$\widetilde A(Q) \leq O(N^{-2+C\varepsilon})$. Similarly with
$m_{|K|}\leq O(1)$, we have $\widetilde B\leq O(N^{-5/2+C\varepsilon})$
completing the proof of Lemma \ref{lemGCheart} and thereby also
finishing the proof of Theorem \ref{comparison}.
\end{pf*}

%s6 #&#
\section{A priori bound for the strong local Marcenko--Pastur
law}\label{ld-MPlaw} Our goal in this section is to prove the following
weaker form of Theorem \ref{451},
and in Section~\ref{secp45} we will use this a~priori bound to obtain
the stronger form as claimed in Theorem~\ref{451}. Throughout this
section, we will assume that $\lim_{N \to\infty}  d_N \in(0,\infty)
\setminus\{1\}$.

%
%
%th6.1 #&#
\begin{theorem}\label{thmdetailed}
Let $X = [x_{ij}]$ with the entries $x_{ij}$ satisfying (\ref{eqnXmat})
and (\ref{eqnXmatexpbd}). For any $\zeta>0$ there exists a~constant
$C_\zeta$ such that the following event holds with $\zeta$-high
probability:
%
%
%e6.1 #&#
\begin{equation}
\label{mainlsresult} \bigcap_{z \in{ \mathbf{S}}(C_\zeta)} \biggl\{
\Lambda_d(z) +\Lambda_o(z) \leq\varphi^{C_\zeta}
\frac{1}{(N\eta)^{1/4}} \biggr\}.
\end{equation}
\end{theorem}

%s6.1 #&#
\subsection{A roadmap for the reader} For conveying the key ideas of
the computations involved in this section, we first give a~brief
outline of the proof of \mbox{Theorem}~\ref{thmdetailed}. For the
reader's convenience, we also indicate the corresponding
theorems/lemmas in
which the estimates mentioned below are proved. % \sn{In this
%subsection we have used $\eta\geq O(1)$ notation and $\eta\sim1$.
%Please check and modify if necessary}

The proof of Theorem~\ref{thmdetailed} proceeds via ``self-consistent
equations'' explained below. Let us fix $\zeta> 0$.
%initiated in \tcr{Jun:put references here}. Fix $\zeta> 0$.
By definition it follows that
\[
m(z) = {1 \over N} \sum_{i}
G_{ii}(z) = {1 \over N} \sum_{i}
{1
\over-z - z(1/M) \operatorname{Tr}\mathcal{G}^{ (i)} - Z_i },
\]
where
%
%
%e6.2 #&#
\begin{eqnarray}
Z_i&:=&z \bigl\langle\mathbf{x}_i,\mathcal{G}^{(i)}
\mathbf {x}_i \bigr\rangle-\frac
{z}{M}\operatorname{Tr}
\mathcal{G}^{(i)}. \label{eqnZi}
\end{eqnarray}
We will first establish Theorem~\ref{thmdetailed} for $\Im z
={\eta\sim1}$. For ${\eta\sim
1}$, the empirical Stieltjes transform satisfies%
\[
m(z)  = {1 \over N} \sum_{i}
{1 \over1 - z-d-z d m(z)+Y_i }, \qquad\max_i|Y_i|\leq
\varphi^{C_\zeta} \Psi
\]
with $\zeta$-high probability (see
Lemma~\ref{lemselfcon1}) where%
%
%e6.3 #&#
\begin{eqnarray}
\Psi&:=&\sqrt{\frac{\Im m_c+\Lambda}{N\eta}}.\label{eqndefpsi} %%v=v(z)=m(z)-m_W(z), \qquad[ Z]=\frac{1}{N}\sum_{i} Z_i,\qquad
\end{eqnarray}
%

%re6.2 #&#
\begin{remark}\label{remsizeofff}
Notice that when $m_c+ \Lambda \leq O(1)$, we have
%
%
%e6.4 #&#
\begin{equation}
\label{sizeofff2} \Psi\le O(N\eta)^{-1/2}.
\end{equation}
\end{remark}
Consequently, we deduce that for {$\eta\sim1$}, the function $m(z)$
satisfies the ``self-consistent'' equation
%
%
%e6.5 #&#
\begin{equation}
\label{eqnselfconintro1} m(z)  = {1 \over1 - z-d-z
d m(z) } +O \bigl(
\varphi^{C_\zeta} \Psi \bigr)
\end{equation}
with $\zeta$-high probability. Notice that the above equation satisfied
by $m(z)$ is nearly identical to the\vadjust{\goodbreak} fixed point equation satisfied by
the Stieltjes transform of the \mbox{MP-}law, namely
%
%
%e6.6 #&#
\begin{equation}
\label{eqnselfconintro2} m_c(z) + \frac{1}{z - (1-d) + z d m_c(z)} = 0
\end{equation}
with $\Im m_c> 0$ when $\Im z > 0$. From (\ref{eqnselfconintro1}) and
(\ref{eqnselfconintro2}), we immediately deduce that (Lemma
\ref{lemselfcon1}) for $\eta\sim1$, with $\zeta$-high probability,%
%
%e6.7 #&#
\begin{equation}
\label{eqnLambintroapbd} |m - m_c| = \Lambda(z) \leq\varphi ^{C_\zeta}
\frac{1}{(N\eta)^{1/4}}.
\end{equation}

We now use (\ref{eqnLambintroapbd}) to establish Theorem
\ref{thmdetailed} for $\eta\sim1$. To this end, we identify the
following ``bad sets'' (improbable events).
% and show
%that they indeed have negligible probability.
%Let $\xi$ be such that $N\eta\geq(\log N)^{2\xi}$.
For $z \in\mathbf{S}(0)$, define
%
%
%e6.8 #&#
\begin{equation}
\qquad \Omega(z,K):= \Bigl\{ \max \Bigl\{\Lambda_o(z), \max
_{i}\bigl|G_{ii}(z) - m(z) \bigr|, \max_{i}
|Z_i| \Bigr\}\geq K\Psi(z) \Bigr\}. \label{eqexcep} %\Omega_{d}(z, K) &:= \Big\{ \max_{ij}|G_{ii}(z) -
%m(z) |\geq K\Psi(z)
\end{equation}
Then the event (Lemma~\ref{coretaeqO1})
%
%
%e6.9 #&#
\begin{equation}
\label{eqngpww3intro} \bigcap_{z \in{ \mathbf{S}}(0),
\eta\sim1 }\Omega \bigl(z,
\varphi^{C_\zeta} \bigr)^c
\end{equation}
holds with $\zeta$-high probability. Here $A^c$ denotes the complement
of the set $A$. The estimate (\ref{eqngpww3intro}) coupled with
(\ref{eqnLambintroapbd}) immediately establishes Theorem
\ref{thmdetailed} for ${\eta\sim1}$.

Before proceeding, we notice the following important point. When $\eta$
is not assumed to be $\sim1$, a~statement analogous to
(\ref{eqngpww3intro}) holds with a~different assumption. Set
%
%
%e6.10 #&#
%e6.11 #&#
\begin{eqnarray}
\mathbf{B}(z) &:=& \bigl\{ \Lambda_o(z) + \Lambda_d(z)
> ( \log N)^{-1} \bigr\}, \label{eqndefB}
\\
\Gamma(z,K) &:=& \Omega(z,K)^c \cup\mathbf{B}(z). \label{eqneventG}
\end{eqnarray}
In Lemma~\ref{lemexpnullset} we show that
%
%
%e6.12 #&#
\begin{equation}
\label{eqngpwwintro} \bigcap_{z \in{ \mathbf{S}}(C_\zeta)}\Gamma \bigl(z,
\varphi^{C_\zeta} \bigr)
\end{equation}
holds with $\zeta$-high probability. It can also be shown that for
$\eta\sim1$, the event $\mathbf{B}^c(z)$ holds with $\zeta$-high
probability.

For proving the result for all $z \in\mathbf{S}(C_\zeta)$ (i.e., for
all
$\eta\geq\varphi^\zeta N^{-1}$) we proceed as follows. For a~function $u(z)$, define its ``deviance'' to be%
%
%e6.13 #&#
\begin{equation}
\label{eqndeffm} \mathcal D(u) (z):= \bigl( {u}^{-1}(z) +z d u(z)
\bigr)- \bigl( {m_c}^{-1}(z) +z d m_c(z)
\bigr).%= {m}^{-1} +z d m + z
%+ d - 1.
\end{equation}
Clearly, $\mathcal D(m_c) = 0$. The plan is to show that $|\mathcal
D(m)| \approx0$ and, therefore, \mbox{$|m_c- m| \approx0$.}

%The proof of Theorem~\ref{thmdetailed} for small $\eta$ relies on
%the following argument.
More precisely, suppose that for two numbers $L,K$ satisfying $\varphi
^L \geq\break  K^2 (\log N)^4$ and for some $A \subset\bigcap_{z
\in\mathbf{S}(L)}
\Gamma(z,K) \bigcap_{\eta\sim1} \mathbf{B}^c(z)$ {(%Note: it means
%that if $
i.e., $A$ is not in the bad sets of $z$ such that $\Im z\sim1$)} one
has the bound
%
%
%e6.14 #&#
\begin{equation}
\label{eqnintroDmbd} \bigl|\mathcal D(m) (z)\bigr| \leq\delta(z)+\infty
1_{\mathbf{B}(z) }\qquad\forall z \in{ \mathbf{S}}(L),
\end{equation}
where $\delta\dvtx \mathbb{C} \mapsto\mathbb{R}_+$ is a~continuous
function, decreasing in $\Im z$ and $|\delta(z)| \leq(\log N)^{-8}$.
Then, via a~continuity argument, we show in Lemma~\ref{lemfm} that
from~(\ref{eqnintroDmbd}) one indeed has the
following stronger conclusion:%
%
%e6.15 #&#
\begin{equation}
\label{eqnintrostconcl} \Lambda\leq C(\log N) \frac{
\delta(z)}{\sqrt{\kappa+\eta+\delta}}\qquad\forall z \in{
\mathbf{S}}(L)
\end{equation}
and $A \subset\bigcap_{z \in\mathbf{S}(L)} \mathbf{B}^c(z)$
%{(with $A \subset\bigcap_{z \in\b S(L)} \Gamma(z,K) \bigcap_{\eta\sim1}
[i.e., $A$ is contained in the bad sets of $z$ for all $z
\in\mathbf{S}(L)$]. This estimate with a~brief additional argument will
yield that for large enough $C$ and $z \in\mathbf{S}(\varphi^C)$, we
have $\Lambda = o(1)$ and $\Omega(z,\varphi^{C_\zeta})^c$ holds with
$\zeta$-high probability. These two conclusions immediately yield
Theorem~\ref{thmdetailed}.\looseness=1

%s6.2 #&#
\subsection{Preliminary estimates} %\label{secprelim}
We start with the following elementary lemma whose proof is standard:
%
%
%le6.3 #&#
\begin{lemma} \label{lemtbtinvmat} For\vspace*{1pt} any rectangular matrix $M$,
and partition matrices $A,B$ and $D$ of $M$ given by $M = {A\ \, B\choose B^\dagger\  D}$, we have the following identity:%
\[
M^{-1} = \pmatrix{
U^{-1} & -U^{-1} B D^{-1}
\cr
-D^{-1}{B}^\dagger U^{-1} & D^{-1} +
D^{-1} {B}^\dagger U^{-1} B D^{-1}
}, \qquad U = A - B D^{-1}{B}^\dagger.
\]
\end{lemma}
%
%
%le6.4 #&#
\begin{lemma} \label{lemxxt} For any $z$ not in the spectrum of
${X}^\dagger X$,
we have%
\[
X \bigl(X^\dagger X-z \bigr)^{-1}X^\dagger=I+z \bigl(X
X^\dagger-z \bigr)^{-1}.
\]
\end{lemma}

\begin{pf}
Indeed from the SVD decomposition given in (\ref{eqnSVD}), we have
\begin{eqnarray*}
X \bigl({X}^\dagger X -z \bigr)^{-1}
{X}^\dagger &=& \sum_{\alpha}
{\lambda_\alpha\over\lambda_\alpha- z} \mathbf {u}_{\alpha} \mathbf{u}_\alpha^\dagger
\\
&=& \sum_{\alpha} \biggl(1 +{ z \over\lambda_\alpha- z}
\biggr) \mathbf{u}_{\alpha} \mathbf{u}_\alpha^\dagger =
I+z \bigl(X X^\dagger-z \bigr)^{-1}
\end{eqnarray*}
and the lemma is proved.
\end{pf}
%
%matrix elements $G_{ij}^{(\mathbb T) }$ and $\mG^{(\mathbb
%T)}_{ij}(z)$.
%{\mathbf Note: one can choose $\zeta=\varepsilon\log N/(\log\log N)$,
%i.e., $(\log N)^\zeta=N^\varepsilon$. Make this precise..}
%See ??.

We record the following properties of $m_c$ without proof.
%
%
%le6.5 #&#
\begin{lemma}[(Properties of $m_c$)] \label{lemmw}
For $z = E + i \eta\in{\mathbf S}(0)$ we have the following bounds:
%
%
%e6.16 #&#
%e6.17 #&#
\begin{eqnarray}
\bigl|m_c(z)\bigr|&\sim& 1, \qquad\bigl|1- m_c^2(z)\bigr|\sim\sqrt{
\kappa+\eta}, \label{smallz}
\\
\label{esmallfake} \Im m_c(z) &\sim&\cases{ \displaystyle
\frac{\eta
}{\sqrt{\kappa+\eta}}, &\quad if $\kappa\ge\eta$ and $|E| \notin[\lambda_-,
\lambda_+]$,
\vspace*{4pt}\cr
\sqrt{\kappa+ \eta}, &\quad if $\kappa\le\eta$ or $|E| \in[
\lambda_-, \lambda_+]$.}
\end{eqnarray}
%
%where $A\sim B$ denotes $C^{-1}B\leq A\leq CB$ for some constants $C$.
Furthermore%
%
%e6.18 #&#
\begin{equation}
\label{sizeofim} \frac{\Im m_c(z)}{N\eta} \ge O \biggl({1\over
N} \biggr)
\quad\mbox{and}\quad\partial_\eta\frac{\Im m_c(z)}{\eta
} \leq0.
\end{equation}
\end{lemma}

Recall $\mathbf{B}(z)$ from (\ref{eqndefB}).
%
%
%le6.6 #&#
\begin{lemma}[(Rough bounds of $\Lambda_o^ {(\mathbb T)}$ and $\Lambda
_d^ {(\mathbb T)}$)] \label{lemGkkmest} Fix $\mathbb T
\subset\{1,2,\ldots, N\}$ such that $|\mathbb T| = O(1)$. For
$z\in{\mathbf S}(0)$, there exists a~constant $C = C_{|\mathbb T|}$
such that
the following estimates hold in $\mathbf{B}^c(z)$:%
%
%e6.19 #&#
%e6.20 #&#
%e6.21 #&#
\begin{eqnarray}
\max_{k\notin\mathbb T }\bigl|G^{(\mathbb T)}_{kk}- G
_{kk}\bigr| &\leq& C \Lambda_o^2, \label{eqngkkmest1}
\\
{1 \over C} \leq\bigl|G^{(\mathbb T)}_{kk}\bigr|& \leq& C,
\label{eqngkkmest2}
\\
\Lambda_o^ {(\mathbb T)} &\leq& C \Lambda_o.
\label{eqngkkmest3}
\end{eqnarray}
\end{lemma}

\begin{pf}
For $\mathbb T = \varnothing$, (\ref{eqngkkmest1}) and
(\ref{eqngkkmest3}) follow from definition, and (\ref{eqngkkmest2})
follows from the definition of $\mathbf{B}(z) $ and (\ref{smallz}). For
nonempty $\mathbb T$, one can prove the lemma using an induction on
$|\mathbb T|$. For example, for $|\mathbb T|=1$, using
(\ref{eqnGijGijk}) we can show that
%
%
%e6.22 #&#
\begin{equation}
\label{220bk} \bigl|G_{kk}(z) - G^{(\mathbb T)}_{kk}(z)\bigr|
\leq C\Lambda_o^2,
\end{equation}
which implies bound (\ref{eqngkkmest1}). A similar argument will yield
(\ref{eqngkkmest2}) and (\ref{eqngkkmest3}).
\end{pf}

On the other hand, when $\eta\sim1$, a~bound similar to
(\ref{eqngkkmest2}) holds without the assumption of $\mathbf{B}^c$.
%
%
%le6.7 #&#
\begin{lemma}[(Rough bounds for $G_{kk}$ for $\eta\sim1$)] \label{lemetaO1bd}
Fix $\mathbb T \subset\{1,2,\ldots, N\}$ such that $|\mathbb T| =
O(1)$. For any $z\in{\mathbf S}(0)$ and $\eta\sim1$, we have the bound
\[
\max_{i} \bigl|G^{(\mathbb T)}_{ii}(z)\bigr| \leq C
\]
for some $C>0$ and $1 \leq i \leq N$.
\end{lemma}

\begin{pf} Let us show the result first for $|\mathbb T| = \varnothing$.
By definition,
\[
| G_{ii}|= \biggl|\sum_{\alpha}
\frac{\mathbf{u}_\alpha
(i)\widebar\mathbf{u}_
\alpha(i)}{\lambda_\alpha-z} \biggr|\leq\frac{1}{\eta} \sum_{\alpha}
\mathbf{u}_\alpha(i)\widebar \mathbf{u}_ \alpha(i) \leq
\frac{1}{\eta}\leq C,
\]
where in the second inequality we have used $|\lambda_\alpha-z |\geq
\Im z=\eta$. The claim for a~general $\mathbb T$ follows similarly.
\end{pf}

Recall from (\ref{eqexcep}) and (\ref{eqneventG}), the event
\[
\Gamma \bigl(z,\varphi^{C_\zeta} \bigr) = \Omega \bigl(z,\varphi
^{C_\zeta} \bigr)^c \cup\mathbf{B}(z).\vadjust{\goodbreak}
\]
Define the events%
%
%e6.23 #&#
\begin{eqnarray}\label{eqnOoOd}
\Omega_{o}(z, K) &:=& \bigl\{ \Lambda_0
\geq K\Psi(z) \bigr\},\nonumber
\\
\Omega_{d}(z, K) &:=& \Bigl\{ \max_{i}\bigl|G_{ii}(z)
- m(z) \bigr|\geq K\Psi(z) \Bigr\},
\\
\Omega_{Z}(z, K) &:=& \Bigl\{ \max_{i}|Z_i
|\geq K\Psi(z) \Bigr\}.\nonumber
\end{eqnarray}
Note: {$\Omega_{d}(z, K) $ is defined with $m$, not $m_c$.} Set
\[
\Omega(z,K)=\Omega_{o}(z, K)\cup\Omega_{d}(z, K) \cup
\Omega_{Z}(z, K).
\]

%
%le6.8 #&#
\begin{lemma}\label{lemexpnullset}
For any $\zeta>0$ there exists a~constant $C_\zeta$ such that
%
%
%e6.24 #&#
\begin{equation}
\label{gpww} \bigcap_{z \in{ \mathbf{S}}(C_\zeta)}\Gamma \bigl(z,
\varphi^{C_\zeta} \bigr)
\end{equation}
holds with $\zeta$-high probability.
\end{lemma}

\begin{pf} We need to prove only that there exists a~uniform constant
$C_\zeta$ such that for any $z \in{ \mathbf{S}}(C_\zeta)$ the event
%
%
%e6.25 #&#
\begin{equation}
\label{gpww2} \Gamma \bigl(z, \varphi^{C_\zeta} \bigr)
\end{equation}
holds with $\zeta$-high probability. It is clear that (\ref{gpww})
follows from (\ref{gpww2}) and the fact that
%
%
%e6.26 #&#
\begin{equation}
\label{jjNC} |\partial_zG_{ij}|\leq
N^C, \qquad\eta> N^{-1}.
\end{equation}
%
%Then we are going to show $\Omega_{o(d)}^c\cup B $ holds \hp{\zeta}.

Note $ \Gamma(z, K) =(\Omega_o^c\cup\mathbf{B}) \cap(\Omega
_d^c\cup\mathbf{B})\cap (\Omega_Z^c\cup\mathbf{B})$. First we shall
prove that the $\Omega _o^c\cup\mathbf{B} $ holds with $\zeta$-high
probability. Using formula (\ref{eqnGij}) and the fact that $|G|^2 =
G^*G$, we infer that there exists a~constant $C_\zeta$ such that with
$\zeta$-high probability,
%
%
%e6.27 #&#
\begin{eqnarray}
\label{227pp}
\Lambda_o(z) &\leq& C|z|\max_{i \neq j}
\bigl| \bigl\langle\mathbf{x}_i, \mathcal{G}^{(ij)}
\mathbf{x}_{j} \bigr\rangle \bigr| \leq\varphi^{C_\zeta}
{|z| \over N} \biggl( \sum_{k,l} \bigl|
\mathcal{G}^{(ij)}_{kl}\bigr|^2 \biggr)^{1/2}\nonumber
\\
&\leq& \varphi^{C_\zeta}{|z| \over N} \bigl( \operatorname{Tr}\bigl|
\mathcal{G}^{(ij)}\bigr|^2 \bigr)^{1/2}
\\
& \leq&\varphi^{C_\zeta} |z|\sqrt{\frac{\Im\operatorname
{Tr}\mathcal{G}^{(ij)}}{N^2\eta}} \qquad\mbox{in }
\mathbf{B}^c(z),\nonumber
\end{eqnarray}
where in the last step we used the identity $\eta^{-1}
\Im\operatorname{Tr}\mathcal{G}^{(ij)} = \operatorname{Tr}|\mathcal
{G}^{(ij)}|^2$. Using the identity
%Using the identity \sidenote{ ${1 \over|z|}$ = $O(1)$ ??}
%
%
%e6.28 #&#
\begin{equation}
\label{eqndtrGmG} \operatorname{Tr}G^{(\mathbb T)}(z)-\operatorname {Tr}
\mathcal{G}^{(\mathbb T)}(z)= \frac{M-N+|\mathbb T| }{z},
\end{equation}
formula (\ref{eqngkkmest1}) and $\Im(z^{-1}) =\eta|z|^{-2}$, we deduce
that with $\zeta$-high probability
\[
\Lambda_o(z) \leq\varphi^{C_\zeta}%|z|
\sqrt
{\frac{\Im m_c+\Lambda+\Lambda_o^2 }{N\eta}+\frac{1}{N}}\qquad \mbox{in }\mathbf{B}^c(z).
\]
For the above choice of $C_\zeta$, for $z\in{ \mathbf{S}}(3C_\zeta )$,
with $\Im m_c\leq O(1)$, the bound
%
%
%e6.29 #&#
\begin{equation}
\label{eqnL0aprbd} \Lambda_o(z) \leq\varphi^{C_\zeta}%|z|
\sqrt{\frac{\Im m_c+\Lambda}{N\eta}+\frac{1}{N}}+o(\Lambda _o)\qquad
\mbox{in }\mathbf{B}^c(z)
\end{equation}
holds with $\zeta$-high probability. From (\ref{eqnL0aprbd}) and
(\ref{sizeofim}) it follows that
$\Omega_o^c\cup\mathbf{B} $ holds with $\zeta$-high probability. %
%Finally the claimed bound
%for $\Lambda_d$ in \eqref{eqnboundlambdao} follows \eqref{eqnximgxi}
%and \eqref{eqnGii}.

A similar argument using the large deviation lemma will give
%
%
%e6.30 #&#
\begin{eqnarray}\label{eqnZibd}
|Z_i| = |z| \biggl| \bigl\langle\mathbf{x}_i,
\mathcal{G}^{(i )} \mathbf{x}_{i} \bigr\rangle-
\frac
{1}{M}\operatorname{Tr}\mathcal{G}^{(i)} \biggr|\leq|z| \varphi
^{C_\zeta} \sqrt{\Im\operatorname{Tr}\mathcal{G}^{(i)} \over N^2 \eta} \leq%
%|z|
\varphi^{C_\zeta} \Psi\nonumber
\nonumber\\[-14pt]\\[-10pt]
\eqntext{\mbox{in }\mathbf{B}^c(z)}
\end{eqnarray}
holds with $\zeta$-high probability implying that%
\[
\max_{i}|Z_i| \leq\varphi^{C_\zeta} \Psi
\]
and therefore $\Omega_Z^c\cup\mathbf{B} $ holds with $\zeta$-high
probability.

Finally notice that $\max_{i} |G_{ii} - m| \leq\max_{i \neq j} |G_{ii}
- G_{jj}|$. From (\ref{eqnGii}) we obtain that
\begin{eqnarray*}
|G_{ii} - G_{jj}| &\leq& \biggl|{ 1 \over- z - z \langle{\mathbf
x}_i,\mathcal{G}^{(i)}(z) {\mathbf x}_i\rangle} -
{ 1 \over- z - z \langle{\mathbf x}_j,\mathcal{G}^{(j)}(z) {\mathbf
x}_j\rangle} \biggr|
\\
&\leq&|G_{ii}G_{jj}| \biggl(|Z_i -
Z_j| + {|z| \over M} \bigl|\operatorname{Tr}
\mathcal{G}^{(i)} - \operatorname{Tr} \mathcal{G}^{(j)}\bigr|
\biggr)
\\
&\leq& C \bigl(\varphi^{C_\zeta} \Psi+ \Lambda_o^2+N^{-1}
\bigr)\qquad\mbox{in }\mathbf{B}^c(z)
\end{eqnarray*}
holds with $\zeta$-high probability, where the last inequality follows
from (\ref{eqnZibd}), (\ref{yz10}), (\ref{eqngkkmest1}) and
(\ref{eqngkkmest2}). Thus we have shown that $\Omega_d^c\cup\mathbf {B}
$ holds with $\zeta$-high probability, and the lemma is proved.
%The lemma now follows from \eqref{eqnL0aprbd} and \eqref{sizeofff2}.
\end{pf}
On the other hand, in the case of $\eta\sim1 $, a~result similar to
Lemma~\ref{lemexpnullset} holds without the assumption of
$\mathbf{B}^c$.

%
%
%le6.9 #&#
\begin{lemma}\label{coretaeqO1}
%Let $\Omega_o(z)$ and $\Omega_d(z)$ be as in \eqref{eqexcep}.
For any $\zeta>0$, there exists a~constant $C_\zeta$ such that the
event
%
%
%e6.31 #&#
\begin{equation}
\label{gpww3} \bigcap_{z \in{ \mathbf{S}}(0), \eta\sim1} \Omega \bigl(z,
\varphi^{C_\zeta} \bigr) ^c % \bigcap\{\max_i|Z_i|\leq\varphi^{C_
\end{equation}
holds with $\zeta$-high probability.
\end{lemma}

\begin{pf}
From (\ref{jjNC}) we see that we need only to prove (\ref{gpww3}) for
fixed $z$. First we note in this case, that is, $\eta\sim1$, we have
$\Im m_c\sim1$ and from Lemma~\ref{lemetaO1bd} we have
$\Lambda=O(1)$ and therefore%
%
%e6.32 #&#
\begin{equation}
\label{ckf} \Psi\sim N^{-1/2}.
\end{equation}
As in (\ref{227pp}) and Lemma
\ref{lemetaO1bd} we obtain that%
\[
\Lambda_o\leq\varphi^{C_\zeta} \sqrt{\frac{\Im\operatorname
{Tr}\mathcal{G}^{(ij)}}{N^2
}}\leq
\varphi^{C_\zeta} N^{-1/2}\leq\varphi^{C_\zeta} \Psi
\]
with $\zeta$-high probability. The estimate for $Z_i$ can be proved as
in (\ref{eqnZibd}) using Lemma~\ref{lemetaO1bd}. The estimate for
$\Omega_d$ [see (\ref{eqnOoOd})] can also be proved similarly using the
identity
\[
\operatorname{Tr}\mathcal{G}^{(i)}-\operatorname{Tr}\mathcal
{G}^{(j)}=\operatorname{Tr}G^{(i)}-\operatorname{Tr}
G^{(j)}=O( \eta)^{-1},
\]
which follows from Cauchy's interlacing
theorem of eigenvalues, that is,%
%
%e6.33 #&#
\begin{equation}
\label{232t} \bigl| m-m^{(i)} \bigr|\leq(N\eta)^{-1}
\end{equation}
and the proof is finished.
\end{pf}

%s6.3 #&#
\subsection{Self-consistent equations}
In Section~\ref{secprelim}, we have bounded $\Lambda_o$ and\allowbreak
$\max_i(G_{ii}-m)$ in terms of $m_c$, $\eta$ and $\Lambda$ in
$\mathbf{B}^c$ (we do not need the event $\mathbf{B}^c$ when
${\eta\sim1}$). In this subsection, we will give the desired bound for
$\Lambda$ and show that the event $\mathbf{B}^c$ holds with
$\zeta$-high probability.

First we give the bound for $\Lambda$ in the case of $\eta\sim1$.
%
%
%le6.10 #&#
\begin{lemma} \label{lemselfcon1} For any $\zeta>0$, there exists a~constant $C_\zeta$ such that
%
%
%e6.34 #&#
\begin{equation}
\label{gpww4} \bigcap_{z \in{ \mathbf{S}}(0), \eta=
10(1+d)}\Lambda(z)\leq
\varphi^{C_\zeta}N^{-1/4}
\end{equation}
holds with $\zeta$-high probability.
\end{lemma}

\begin{pf}
By the definition of $Z_i$ given in formulas (\ref{eqnZi}) and
(\ref{eqnGii}),%
%
%e6.35 #&#
\begin{equation}
\label{fgll} \bigl(G_{ii}(z) \bigr)^{-1} = -z - z
{1 \over M} \operatorname{Tr}\mathcal{G}^{ (i)} -
Z_i.
\end{equation}
Using (\ref{eqndtrGmG}) and (\ref{232t}), we obtain that if
$\eta\sim1$,
%
%
%e6.36 #&#
\begin{equation}
\label{eqnximtrgfake} \biggl| z \frac{1}{M}\operatorname {Tr}
\mathcal{G}^{(i)}- z d m(z)+1-d \biggr|\leq C N^{-1}.
\end{equation}
Together with $|Z_i| \leq\varphi^{C_\zeta}\Psi$ [see (\ref{gpww3})],
estimate (\ref{eqnximtrgfake}) implies that
\[
m(z)  = {1 \over N} \sum_{i}
{1 \over1 - z-d-z d m(z)+Y_i }, \qquad\max_i|Y_i|\leq
\varphi^{C_\zeta} \Psi\leq O \bigl(\varphi^{C_\zeta}N^{-1/2}
\bigr).
\]
It thus follows that $ |m(z)|\sim1$ for $\eta\sim1$ with $\zeta$-high
probability. Then using the fact that
$\sum_{i}(G_{ii} - m) = 0$ we obtain that%
\[
\sum_i \bigl(G_{ii}(z) \bigr)^{-1}=m^{-1}(z) +O \Bigl(\max_i
|G_{ii}-m| \Bigr)^2.
\]
Recall $\mathcal D$ in (\ref{eqndeffm}). Using (\ref{fgll}),
(\ref{ckf}) and the bound $|Z_i| + |G_{ii}-m| \leq
\varphi^{C_\zeta}\Psi$ [see (\ref{gpww3})], and we have
\[
\mathcal D(m) = \delta(z), \qquad\bigl|\delta(z)\bigr| \leq\varphi^C \Psi \leq
O \bigl( \varphi^C N^{-1/2} \bigr).
\]
The two solutions $m_1,m_2$ of the equation $\mathcal D(m) = \delta(z)$
for a~given
$\delta(\cdot)$ are given by%
%
%e6.37 #&#
\begin{eqnarray}
\label{eqnDmtwosolbdfake} m_{1,2} &=&\frac{\delta(z) + 1-d-z \pm i
\sqrt{(z-\lambda_{-,\delta})(\lambda_{+,\delta}-z)}}{2 d z}, %%m_2(z)=\frac{\delta(z) + 1-d-z +i \sqrt{(z-\lambda_{-,\delta})(
%m_1 - m_2 &= {\sqrt{(z-\lambda_{-,\delta})(z-\lambda_{+,\delta})}
\nonumber\\[-8pt]\\[-8pt]
\lambda_{\pm, \delta} &=& 1 + d \pm2 \sqrt{d - \delta(z)} - \delta(z), \qquad
|\lambda_{\pm,\delta} - \lambda_{\pm}| = O(\delta).\nonumber
\end{eqnarray}
Therefore, we obtain $m=m_1$ or $m_2$. It is easy to see that
$|m_1-m_2|\geq O(1)$, since $\eta\sim1$. Since $m(z)$ is continuous
with respect to $E$ (for fixed $\eta$), $m=m_1$ (say) for $E=0$ implies
that $m=m_1$ for all $E=O(1)$. Using this fact and $\Im m >0$, we
obtain that $m(z) =\frac{\delta(z) + 1-d-z + i \sqrt{(z-\lambda
_{-,\delta})(\lambda_{+,\delta}-z)}}{2 d z}$, %. Comparing with
and thus we obtain (\ref{gpww4}) and the proof of the lemma is
complete.
\end{pf}

Now combining (\ref{gpww3}) with (\ref{gpww4}), we have proved that for
any $\zeta>0$, there exists a~constant $C_\zeta$ such that, for
$\eta=10(1+d)$, formula (\ref{mainlsresult}) holds with $\zeta$-high
probability. It immediately follows that the event
%
%
%e6.38 #&#
\begin{equation}
\label{gpww5} \bigcap_{z \in{ \mathbf{S}}(0), \eta=
10(1+d)}\mathbf{B}^c(z)
\end{equation}
holds with $\zeta$-high probability for any $\zeta> 0$.

Now we prove (\ref{mainlsresult}) for general {$\eta>0 $}. Recall the
\emph{deviance} function from (\ref{eqndeffm}), $Z_i$ from
(\ref{eqnZi}) and set%
%
%e6.39 #&#
\begin{equation}
\label{eqnZbar} [Z] = {1 \over
N} \sum
_{i=1}^N Z_i.
\end{equation}
Recall the set $\mathbf{B}(z)$ from (\ref{eqndefB}) and $\Gamma (z,K)$
from Lemma~\ref{lemexpnullset}.
%
%
%le6.11 #&#
\begin{lemma} \label{corDm} Fix $1\leq K\leq(\log N)^{-1}(N\eta
)^{1/2} $. Then, on the set $\Gamma(z, K)$, we have the bound
\[
\bigl| \mathcal D(m)\bigr| \leq{\bigl|[Z]\bigr|}+O \bigl(K^2\Psi^2 \bigr) +
\infty1_{\mathbf{B}(z)}.% -
%O( [Z] )+ C(\log N)^{2\xi}\Psi^2 + \Lambda_o^3
\]
\end{lemma}

\begin{pf}
Using (\ref{eqnGii}), (\ref{eqngkkmest1}), (\ref{eqndtrGmG}) and the
definition of $m_c$, on the set $\Gamma(z, K)$, we obtain a~more
precise version of (\ref{fgll}),
%
%
%e6.40 #&#
\begin{eqnarray}
G_{ii}(z)^{-1} = m_c(z)^{-1} + z d \bigl[m_c(z) - m(z) \bigr]
-Z_i + O \bigl(K^2 \Psi^2 \bigr) +O \bigl(N^{-1} \bigr)\nonumber
\\
\eqntext{\mbox{in } { \mathbf{B}^c}\cap \Omega^c,}
\end{eqnarray}
where $\Omega:=\Omega(z, K)$. Then
%
%
%e6.41 #&#
\begin{eqnarray}
\label{eqncordmbd1} \quad G^{-1}_{ii} -m^{-1} &=& \mathcal
D(m) -Z_i + O \bigl(K^2 \Psi^2 \bigr) + O
\bigl(N^{-1} \bigr)\qquad\mbox{in } {\mathbf{B}^c}\cap
\Omega^c
\end{eqnarray}
and averaging over $i$ yields
\[
{1 \over N} \sum_{i=1}^N
\bigl(G^{-1}_{ii} -m^{-1} \bigr) = \mathcal D(m) -
[Z] + O \bigl(K^2\Psi^2 \bigr) +O \bigl(N^{-1}
\bigr)\qquad\mbox{in } {\mathbf{B}^c}\cap\Omega^c.
\]
It follows from the assumptions $K\ll(N\eta)^{1/2} { \leq}
O(\Psi^{-1})$ that $G_{ii}-m=o(1)$. Expanding the left-hand side and
using the facts that $\sum_{i}(G_{ii} - m) = 0$,%
%
%e6.42 #&#
\begin{eqnarray}
\sum_{i=1}^N \bigl(G^{-1}_{ii}
-m^{-1} \bigr) &=& \sum_{i=1}^N
{G_{ii} - m \over
G_{ii} m} = {1 \over m^3} \sum
_{i=1}^N (G_{ii} - m)^2 +
\sum_{i=1}^N O \biggl(
{(G_{ii} - m)^3 \over m^4} \biggr)\nonumber
\\
\eqntext{\mbox{in } {\mathbf{B}^c}\cap \Omega^c.}
\end{eqnarray}
%
%Since $z \sim O(1)$, we have $ m^{-1} = O_p(1)$.
Together with (\ref{eqngkkmest2}) and (\ref{eqexcep}), it follows that
%
%
%e6.43 #&#
\begin{equation}
\label{eqncordmbd2} {1 \over N} \sum_{i=1}^N
\bigl(G^{-1}_{ii} -m^{-1} \bigr) \leq C (K \Psi
)^2\qquad\mbox{in } {\mathbf{B}^c}\cap
\Omega^c.
\end{equation}
Now the lemma follows from (\ref{eqncordmbd1}) and (\ref{eqncordmbd2}).
\end{pf}

%Recall the two solutions $m_1,m_2$ of the equation
%$\mathcal D(m) = \delta(z)$ for a~given $\delta(\cdot)$ are given by
%m_{1,2} &=\frac{\delta(z) + 1-d-z \pm i \sqrt{(z-\lambda_{-,\delta})(
%m_2(z)=\frac{\delta(z) + 1-d-z +i \sqrt{(z-\lambda_{-,\delta})(
%m_1 - m_2 &= {\sqrt{(z-\lambda_{-,\delta})(z-\lambda_{+,\delta})}
% \\
%, \qquad\quad
%|\lambda_{\pm,\delta} - \lambda_{\pm}| = O(\delta).

%and
%the function $u$ from Corollary~\ref{corDm}.
%
%
%le6.12 #&#
\begin{lemma}\label{lemfm} Let $K, L>0$ be two numbers such that
$\varphi^L\geq K^2(\log N)^4 $, and let $A$ be an event given by
%
%
%e6.44 #&#
\begin{equation}
\label{lY} A \subset\bigcap_{z\in S(L)}\Gamma(z, K)
\cap\bigcap_{z\in S(L), \eta=10(1+d)} \mathbf{B}^c(z).% \left\{
\end{equation}
Suppose that, in $A$, we have the bound%
\[
\bigl|\mathcal D(m) (z)\bigr| \leq\delta(z)+{\bolds\infty} 1_{\mathbf{B}(z)
}\qquad \forall
z \in{ \mathbf{S}}(L),
\]
where $\delta\dvtx \mathbb{C} \mapsto\mathbb{R}_+$ is a~continuous
function, decreasing in $\Im z$ and $|\delta(z)| \leq(\log N)^{-8}$.
Then for some constant $C>0$, the bound
%
%
%e6.45 #&#
\begin{equation}
\label{eqnsylx} \bigl| m(z )-m_c(z) \bigr|=\Lambda(z) \leq C(\log N)
\frac{ \delta
(z)}{\sqrt{\kappa+\eta+\delta}}\qquad\forall z \in{ \mathbf{S}}(L)
\end{equation}
holds in $A$ and
%
%
%e6.46 #&#
\begin{equation}
\label{eqnsylx2} A\subset\bigcap_{z\in{ \mathbf{S}}(L)}
\mathbf{B}^c(z).
\end{equation}
\end{lemma}
%
%
%re6.13 #&#
\begin{remark}
Formula (\ref{lY}) says that if $\Im z=10(1+d)$, then
$A\subset\Omega(z,K)^c$; that is, $A$ is not in the bad sets of such
$z$, and (\ref{eqnsylx2}) implies that $A$~is not in the bad sets of
all $z \in{\mathbf{S}}(L)$. The difficulty in the proof is that our
hypothesis yields the bound $\mathcal D(m) \leq\delta(z)$ only in the
set $\mathbf{B}^c$, but we need to prove (\ref{eqnsylx}) for both
$\mathbf{B}$ and $\mathbf{B}^c$.
\end{remark}

\begin{pf*}{Proof of Lemma \ref{lemfm}} %Let us fix $E=E_0$.
Let us first fix $E$ and define the set
\[
I_E = \biggl\{ \eta\dvtx \Lambda_o(E + i \widehat{
\eta}) + \Lambda_d(E+ i \widehat{\eta}) \leq{1 \over\log N}
\ \forall\widehat{\eta} \geq\eta, E +i\widehat{\eta} \in{ \mathbf{S}}(L)
\biggr\}.
\]
We first prove (\ref{eqnsylx}) for all $z = E + i \eta$ with $\eta \in
I_E$.
%The proof is simple in the case that $\kappa+\eta=O(1)$. % Note: in
%the case $\kappa+\eta=o(1)$, we have $|m_{sc}^2-1|=o(1)$
Define
\[
\eta_1 = \sup_{\eta\in I_E} \bigl\{\eta\dvtx \delta(E+i
\eta) \geq(\log N)^{-1} (\kappa+\eta) \bigr\}.
\]
Since $\delta$ is a~continuous decreasing function of $\eta$ by
assumption, $\delta(E+i\eta)\leq(\log N)^{-1} (\kappa+\eta_1)$ for
$\eta\geq\eta_1$. %and analogously for $\eta\leq\eta_1$.
Let $m_1$ and $m_2$ be the two solutions of the equation $\mathcal D(m)
= \delta(z)$ as given in (\ref{eqnDmtwosolbdfake}). Note by assumption
we do have $\mathcal|D(m)| \leq\delta(z)$ for $z=E+\eta i$ and $\eta
\in I_E $, since we are in $\mathbf{B}^c(z)$. Then it can be easily
verified that
%for fixed $z$ and $\delta(\cdot)$, using \eqref{eqnDmtwosolbd} it is
%easy to check that \sidenote{Jun: This is \\only true \\ for $z \geq
%O(1)$ \\?}
%
%
%e6.47 #&#
\begin{eqnarray}
\label{eqnm1m2} |m_1 - m_2| %\left\{\begin{array}{ccc}
&\geq& C\sqrt{
\kappa+\eta},\qquad\eta\geq\eta_1
\nonumber\\[-8pt]\\[-8pt]
&\leq& C(\log N)\sqrt{\delta(z)}, \qquad\eta\leq\eta_1.\nonumber
\end{eqnarray}
The difficulty here is that we do not know which of the two solutions $
m_1,m_2$ is equal to $m$. However for $\eta= O(1)$, we claim that $m =
m_1$. For $\eta= O(1)$, $|m - m_c|=\Lambda\leq\Lambda_d \ll1$. Also,
a~direct calculation using
(\ref{eqnDmtwosolbdfake}) gives%
%
%e6.48 #&#
\begin{equation}
\label{eqnmmwdiff} |m_1 - m_c| = C {\delta(z) \over\sqrt{\kappa+
\eta}}
\ll {1 \over
\log N}.
\end{equation}
Since $|m_1 -m_2| \geq C\sqrt{\kappa+ \eta}$ for $\eta= O(1)$ [see
(\ref{eqnm1m2})], it immediately follows that $m = m_1$ for $\eta=
O(1)$. Furthermore, since the functions $m_1, m_2$ and $m$ are
continuous and since $m_1 \neq m_2$ for $\eta> \eta_1$, it follows that
$m = m_1$ for $\eta\geq\eta_1$.
%%% \begin{comment}
%%%\begin{eqnarray} \eta_2=\min
%%%\left\{\eta\geq\eta_1: |m(z) - \mw(z)|\leq(\log N)^{-1/2}
%%%\sqrt{\kappa+\eta} \right\}.
%%%\end{eqnarray} Using \eqref{eqnm1m2} and initial condition, we can
%pick the
%%%\begin{eqnarray} |m(E+i\eta_2) -
%%%\mw(E+i\eta_2)|\leq C \frac{\delta}{\sqrt{\kappa+\eta}}\leq C(\log
%%%N)^{-1 } \sqrt{\kappa+\eta} \end{eqnarray} which implies that
%%%$\eta_2=\eta_1$.
%%%\end{comment}
Thus for $\eta\geq\eta_1$,%
\[
\bigl|m(z ) - m_c(z)\bigr| = \bigl|m_1(z) - m_c(z)\bigr|\leq C
\frac{\delta(z)}{\sqrt {\kappa+\eta}}\leq C \frac{\delta(z)}{\sqrt{\kappa+\eta+\delta}},
\]
where in the last step we have used $\delta\leq\kappa+\eta$.

For $ \eta\leq\eta_1$, we take advantage of the fact that the
difference $|m_1 -m_2|$ is the same order as the middle term of
(\ref{eqnmmwdiff}). Indeed, for $\eta\leq\eta_1$, if $m= m_2$ (say),
then using (\ref{eqnm1m2}),
\[
|m -m_c|\leq|m_2 - m_1| + |m_1
- m_c| \leq(\log N) \sqrt{\delta(z)}\leq C(\log N)
\frac{\delta(z)}{\sqrt {\kappa+\eta+\delta}}
\]
verifying (\ref{eqnsylx}) for $\eta\in I_E$.\vadjust{\goodbreak}

From the above computations for $\eta\sim1$, we know $I_E\neq
\varnothing$. Now we prove that $I_E$ is exactly the desired region,
that is, $[\varphi^LN^{-1}, 10(1+d)] $, and this will
verify~(\ref{eqnsylx2}). We argue by contradiction. Indeed, assume that
$I_E \neq[\varphi^LN^{-1}, 10(1+d)]$. Let $\eta_0=\inf I_E $. Then the
continuity assumption yields that
%
%
%e6.49 #&#
\begin{equation}
\label{eqnwyz} \Lambda_o(z_0)+\Lambda_d(z_0)=
(\log N)^{-1}, \qquad z_0=E+i\eta_0
\end{equation}
and thus $\Lambda(z_0) \leq\Lambda_d(z_0)\leq(\log N)^{-1}$. On the
other hand, from the calculations done above we deduce that
(\ref{eqnsylx})
holds for $\eta\in I_E$ and thus%
%
%e6.50 #&#
\begin{equation}
\label{99p} \Lambda(z_0)\leq(\log N)^{-3}.
\end{equation}
By definition,
\[
\bigl\{ \Lambda_o(z_0)+\Lambda_d(z_0)=
(\log N)^{-1} \bigr\} \cap\Gamma(z_0) \subset \bigl(
\Omega_o(z_0) \cup\Omega_d(z_0)
\bigr)^c
\]
and therefore
\[
\Lambda_o(z_0) + \max_{k}\bigl|G_{kk}(z_0)
- m(z_0)\bigr| \leq C K \Psi(z_0).
\]
From\vspace*{-1pt} the assumption $\varphi^L\ge K^2(\log N)^4$, we have
$\Psi(z_0) \leq\sqrt{{\Im m_c\over{N\eta}} +{\Lambda(z_0) \over N\eta}}
\ll K^{-1}(\log N)^{-2}$
%Then with the assumption of the lemma, one can easily obtain that $
%|m-m_W |\leq1/2(\log N)^{-1}$. Using lemma \eqref{l37}, we have
which immediately implies that $\Lambda_o(z_0)+\max_{k}|G_{kk}(z_0)
-\allowbreak m(z_0)| \ll(\log N)^{-1} $. Using this estimate and (\ref{99p})
we deduce that
\[
\Lambda_o(z_0)+\Lambda_d(z_0)
\leq\Lambda_o(z_0)+\max_{k}\bigl|G_{kk}(z_0)
- m(z_0)\bigr| + \Lambda\ll\log N^{-1},
\]
which contradicts (\ref{eqnwyz}), and therefore (\ref{eqnsylx2}) is
verified. This completes the proof of the lemma.
\end{pf*}

Now we complete the proof of Theorem~\ref{thmdetailed}.

\begin{pf*}{Proof of Theorem~\ref{thmdetailed}}
From (\ref{eqnZibd}), Lemmas~\ref{lemexpnullset} and~\ref{corDm}, it
follows that for any $\zeta>0$, there exist constants $C_\zeta$,
$D_\zeta$ and $\widetilde C_\zeta$ that
\[
\bigl| \mathcal D(m) (z)\bigr| \leq\varphi^{\widetilde C_\zeta}\Psi+\infty 1_{\mathbf{B}(z)}
\qquad\forall{z\in\mathbf{S}(C_\zeta)} % - O(
%[Z] )+ C(\log N)^{2\xi}\Psi^2 +
\]
holds on the event $A_\zeta$ given by
%
%
%e6.51 #&#
\begin{equation}
\label{245u} A_\zeta= \bigcap_{z\in\mathbf{S}(C_\zeta)}
\Gamma \bigl(z, \varphi^{D_\zeta} \bigr).
\end{equation}
%
%From Lemma~\ref{lemexpnullset} we see that the event $A_\zeta$ holds
Choosing a~larger $C_\zeta$, applying Lemma~\ref{lemfm} with
\[
A = A_\zeta\cap\bigcap_{z\in S(0), \eta=10(1+d)}
\mathbf{B}^c(z)
\]
and $\delta(z)= \varphi^{C_\zeta}(N\eta)^{-1/2}$, we obtain that
%
%
%e6.52 #&#
\begin{equation}
\label{ag240} \Lambda(z)\leq\varphi^{C_\zeta}(N\eta)^{-1/4}
\qquad\forall z\in\mathbf{S}(C_\zeta)
\end{equation}
holds in $A$.
%Using \eqref{245u} and \eqref{gpww4}, we obtain that for any $
%sentence implies something stronger}.
Furthermore, (\ref{eqnsylx2}) implies that%
%
%e6.53 #&#
\begin{equation}
\label{eqnbcompref} A \subset\bigcap_{z\in\mathbf{S}(C_\zeta
)}
\mathbf{B}^c(z).
\end{equation}
This observation gives that $\Lambda(z) \leq\Lambda_d(z) =o(1)$ in $A$
and $\Psi\leq C(N\eta)^{-1/2} $ in~$A$. Now since both $A_\zeta$ and
$\bigcap_{z\in S(0), \eta=10(1+d)} \mathbf{B}^c(z)$ hold with $\zeta
$-high probability [proved, resp., in Lemma~\ref{lemexpnullset} and
(\ref{gpww5})] it follows that the event $A$ holds with $\zeta$-high
probability. Now from the observation (\ref{eqnbcompref}) we see that
$\Omega(z,\varphi^{C_\zeta})^c$ holds with $\zeta$-high probability.
Together with $\Psi\leq C(N\eta )^{-1/2} $ in $A$, we obtain
(\ref{mainlsresult}). This completes the proof of Theorem
\ref{thmdetailed}.
\end{pf*}

%s7 #&#
\section{Strong bound on $[Z]$} \label{seczlemma}
For proving Theorems~\ref{451} and~\ref{452}, the key input is the
following lemma which gives a~much stronger bound on $[Z]$. Throughout
this section, we will assume that $\lim_{N \to\infty}  d_N
\in(0,\infty) \setminus\{1\}$. The following is the main result of this
section:
%
%
%le7.1 #&#
\begin{lemma} \label{lemZlem}
Let $K, L>0$ be such that $\varphi^L\geq K^2(\log N)^4 $. Suppose for
some event
\[
\Xi\subset\bigcap_{z \in{ \mathbf{S}}(L)} \bigl( \Gamma(z, K)\cap
B^{c}(z) \bigr),
\]
we have
\[
\Lambda(z) \leq\widetilde\Lambda(z)\qquad\forall z \in{ \mathbf{S}}(L),
\]
where $\widetilde\Lambda(z)$ is some deterministic number and
$\mathbb{P}(\Xi
^c) \leq e^{-p (\log N)^2 }$ with %$p$, which grows with $N$, satisfying
%
%
%e7.1 #&#
\begin{equation}
\label{kk20} 1\ll p \ll \bigl(\log(N K) \bigr)^{-1}
\varphi^{
L/2}.
\end{equation}
Then there exists $\Xi'$ such that $\mathbb{P}( \Xi')
\geq1-\frac{1}2e^{-p }$, and for any $z \in{ \mathbf{S}}(L)$,
%
%
%e7.2 #&#
\begin{equation}
\label{32you} \bigl| [Z] \bigr| \leq C p ^5K^2 \widetilde
\Psi^2, \qquad\widetilde\Psi:=\sqrt{\frac{\Im m_c+\widetilde
\Lambda}{N\eta}}\qquad\mbox{in
}\Xi'.
\end{equation}
\end{lemma}

%
%
%re7.2 #&#
\begin{remark}
In the application of the above lemma in Section~\ref{secp45}, we will
set $p_N$ and $K=O(\varphi^{O(1)})$. This lemma is analogous to Lemma
5.2 in \cite{EYYgenwig} [with $p=O(1)$], Corollary 4.2 in
\cite{EYYrigid} and Lemma 4.1 in \cite{EKYY12}, which are used in the
contexts of Wigner matrices and sparse matrices. The basic idea is to
utilize the fact that the entries of Green's function are weakly
correlated. But in our work, we give a~simple, general lemma (Lemma
\ref{abstractZlemma}) on the cancellation of weakly coupled random
variables, which may not have the special structure of Green function,
and is thus useful in more general contexts. For instance, our lemma is
used for proving universality in non-Hermitian matrices in
\cite{BYY2012}.
\end{remark}

%s7.1 #&#
\subsection{Abstract decoupling lemma}
First, we are going to introduce the following abstract decoupling
lemma\footnote{This lemma is joint work with Prof. H.~T. Yau, and we
thank him for kindly allowing us to include it here.} which is similar
to Theorem 5.6 of \cite{EKYY11} and Lemma 4.1 of \cite{EYYrigid}.
However, our lemma as stated here is more general and focuses on weakly
coupled random variables and thus is independent of the structure of
the matrix ensemble. Due to this generality, it has been useful in
other contexts; for instance in \cite{BYY2012} where the authors used
it in the context of local circular law.

Let $\mathcal I$ be a~finite set which may depend on $N$ and
\[
\mathcal I_i \subset\mathcal I,\qquad1\leq i\leq N.
\]
Let $\{x_\alpha, \alpha\in\mathcal I \}$ be a~collection of independent
random variables and ${\mathcal Z}_1, \ldots,\break  {\mathcal Z}_N$ be random
variables which are functions of $\{x_\alpha, \alpha\in\mathcal I \} $.
Let $\mathbb{E}_i$ denote the expectation value operator with respect
to $\{x_\alpha, \alpha\in\mathcal I_i \}$. Define the commuting
projection operators%
\begin{eqnarray*}
Q_i &=& 1-\mathbb{E}_i,\qquad P_i =
\mathbb{E}_i, \qquad P_i^2 = P_i,
\\
Q_i^2 &=& Q_i,\qquad
[Q_i,P_j]=[P_i, P_j]=[Q_i,Q_j]=0
\end{eqnarray*}
and, for $A\subset\{1,2,\ldots,N\}$,
\[
Q_A:=\prod_{i\in A}Q_i,
\qquad P_A:=\prod_{i\in A}P_i.
\]
We use the notation
\[
[\mathcal Q {\mathcal Z}] = \frac{1}{N} \sum_{i = 1}^N
Q_i{\mathcal Z}_i.
\]

%
%
%le7.3 #&#
\begin{lemma}[(Abstract decoupling lemma)] \label{abstractZlemma}
Let $\Xi$ be an event and $p$ an even integer, which may depend on $N$.
Suppose the following assumptions hold with some constants $C_0$,
$c_0>0$:
\begin{longlist}[(iii)]
\item[(i)] \textup{(Bound on $Q_A \mathcal Z_i$ in $\Xi$)}. There exist
    deterministic positive numbers $\mathcal X<1$ and $\mathcal Y$
    such that for any set $A\subset\{1,2,\ldots, N\}$ with $i\in A$
    and $| A | \leq p$, $Q_{A}\mathcal Z_i$ in $\Xi$ can be written
    as the sum of two new random variables
%
%
%e7.3 #&#
\begin{equation}
\label{511} {\mathbf1}(\Xi) ( Q_{A}\mathcal Z_i )=
\mathcal Z_{i, A}+ {\mathbf1}(\Xi) Q_{A}{\mathbf1} \bigl(
\Xi^c \bigr) \widetilde{\mathcal Z}_{i, A}
\end{equation}
and
%
%
%e7.4 #&#
\begin{equation}
\label{511a} | \mathcal Z_{i, A} |\leq\mathcal Y\bigl (C_0
\mathcal X|A| \bigr)^{ |A|},\qquad| \widetilde{\mathcal Z}_{i,
A} |
\leq \mathcal Y N^{C_0|A|}.
\end{equation}

\item[(ii)] \textup{(Rough bound on $\mathcal Z_i$)}.
%
%
%e7.5 #&#
\begin{equation}
\label{ZUNC} \max_i | \mathcal Z_i | \leq
\mathcal Y N^{C_0}.
\end{equation}

\item[(iii)] \textup{($\Xi$ is a~high probability event)}. %We require that
%
%
%e7.6 #&#
\begin{equation}
\label{513} \mathbb{P} \bigl[\Xi^c \bigr] \leq e^{-c_0(\log N)^{3/2} p }.
\end{equation}
\end{longlist}

Then, under assumptions \textup{(i)}, \textup{(ii)} and \textup{(iii)}
above, we have
%
%
%e7.7 #&#
\begin{equation}
\label{52} \mathbb{E}[\mathcal Q \mathcal Z] ^{p} \leq
(Cp)^{4p } \bigl[ \mathcal X^{2} + N^{-1}
\bigr]^{p}\mathcal Y^p
\end{equation}
for some $C>0$ and any sufficiently large $N%\ge N_0(A_0, \psi)
$.
%Here $\Gamma$ is the good set defined in \eqref{Gamma}.
\end{lemma}

The intuition behind Lemma~\ref{abstractZlemma} is the following. If
$Z_i$ are totally independent, that is, $Q_A Z_i=0$ if $\exists j\in A$
and $i\neq j$, we see that $\sum Z_i$ is less than $\sum|Z_i|$ by
a~factor $N^{-1/2}$. In this case $Z_i$ depends only on $\{x_\alpha,
\alpha \in\mathcal I_i \}$. For the general case considered in Theorem
\ref{abstractZlemma}, $Z_i$ also weakly depends on sets $\{x_\alpha,
\alpha \in\mathcal I_j \}$ for $i\ne j$. Here $ Q_j Z_i$ can be
considered as the set $\{x_\alpha, \alpha\in\mathcal I_j \}$ ``acting''
on $X_i$, and $ Q_kQ_j Z_i$ the action of $\{x_\alpha,
\alpha\in\mathcal I_k \}$ on the action of $\{x_\alpha,
\alpha\in\mathcal I_j \}$ on $X_i$, so on and so forth. This lemma
shows that if the ``action'' is hierarchical, then indeed $\sum Z_i$ is
much less than $\sum|Z_i|$ in the sense of (\ref{52}).

Before we give a~proof of Lemma~\ref{abstractZlemma}, we introduce a~trivial but useful identity%
%
%e7.8 #&#
\begin{equation}
\label{inxy} \prod_{i=1}^n(x_i+y_i)%-\prod_{i=1}^n(x_i )
=\sum_{s=1}^{n+1 } \Biggl[ \Biggl(\prod
_{i=1}^{s-1} x_i \Biggr)
y_{s} \Biggl(\prod_{i=s+1}^n(x_i+y_i)
\Biggr) \Biggr]
\end{equation}
with the convention that $\prod_{i\in\varnothing}=1$. It implies that
\[
\Biggl|\prod_{i=1}^n(x_i+y_i)-
\prod_{i=1}^n(x_i ) \Biggr|\leq
n \max_i |y_{i}| \Bigl(\max_i|x_i+y_i|+
\max_i|x_i | \Bigr).
\]
For any $1\leq k\leq n$, it follows from
$\prod_{i=1}^n(x_i+y_i)=(x_k+y_k)\prod_{i\neq k} (x_i+y_i)$ and formula
(\ref{inxy}) that
%
%
%e7.9 #&#
\begin{equation}
\label{inxy2} \qquad\prod_{i=1}^n(x_i+y_i)%-\prod_{i=1}^n(x_i )
=\sum_{s\neq k, s=1}^n(x_k+y_k)
\Biggl[ \Biggl(\prod_{ i\neq k, i=1}^{s-1}
x_i \Biggr)y_s \Biggl(\prod
_{i\neq k, i=s+1}^n(x_i+y_i) \Biggr)
\Biggr].
\end{equation}

\begin{pf*}{Proof of Lemma~\ref{abstractZlemma}}
First, by definition, we have
\[
\mathbb{E}[\mathcal Q \mathcal Z] ^{p} =  \frac{1 }{N^p} \sum
_{j_1, \ldots, j_p} \mathbb{E}\prod_{\alpha=1}^pQ_{j_\alpha}
\mathcal Z_{j_\alpha}.
\]
For fixed $j_1, \ldots, j_p$, let $T_\alpha=Q_{j_\alpha}\mathcal
Z_{j_\alpha} $. Now choosing $k=1$, $x_i =P_{j_1}T_i$ and
$y_i=Q_{j_1}T_i$ in (\ref{inxy2}) (noting that $x_i+y_i=T_i$), we have
\[
\prod_{\alpha=1}^p T_\alpha= \sum
_{s=2}^{p+1} T_1 \biggl[
\biggl(\prod_{ \alpha<s,
\alpha\neq1}P_{j_1} T_\alpha
\biggr) (Q_{j_1} T_s) \biggl(\prod
_{\alpha>s,
\alpha\neq1} T_\alpha \biggr) \biggr].
\]
We define $A_{\alpha,s}:={\mathbf1 }_{\{\alpha< s, \alpha\neq1\}
}\{j_1\}$ and $B_{\alpha, s}:={\mathbf1}_{\alpha=s}\{j_1\}$; thus
$B_{\alpha, s}=\{j_1\}$ if $\alpha=s$, otherwise $A_{\alpha,
s}=\varnothing$. It is clear that $A_{1,s}=
B_{1,s}=\varnothing$. Then%
\[
\prod_{\alpha=1}^p T_\alpha=\sum
_{s=2}^{p+1}\prod
_{\alpha}P_{A_{\alpha,s}}Q_{B_{\alpha,s}}T_\alpha.
\]
Generalizing, we replace $s$ with $s_1$ to obtain%
\[
\prod_{\alpha=1}^p T_\alpha=\sum
_{s_1=1}^{p+1}{\mathbf1}(s_1
\neq1) \prod_{\alpha}P_{A_{\alpha,s_1}}Q_{B_{\alpha,s_1 }}T_\alpha
\]
and
%
%
%e7.10 #&#
\begin{equation}
\label{ABals} A_{\alpha,s_1}=\{j_1\dvtx \alpha<s_1,
\alpha \neq1\}, \qquad B_{\alpha,s_1}=\{j_1\dvtx
s_1=\alpha\}.
\end{equation}
Iterating
for $1\leq j_1, j_2,\ldots, j_p\leq N$, we have%
\[
\prod_{\alpha=1}^p T_\alpha=\sum
_{s_1, s_2,\ldots, s_p=1}^{p+1} \prod
_i{\mathbf1}(s_i\neq i) \prod
_{\alpha} P_{A_{\alpha,{\mathbf s}}}Q_{B_{\alpha,{\mathbf
s}}}T_\alpha,
\]
where $\mathbf s$ denotes $s_1, s_2,\ldots, s_p$ and $A _{\alpha,
{\mathbf s}}$, and $B _{\alpha, {\mathbf s}}$ are defined as
\[
A _{\alpha, {\mathbf s}}=\{j_i\dvtx \alpha<s_i, \alpha\neq i\},\qquad B_{\alpha, {\mathbf s}} =\{j_i\dvtx s_i=\alpha\}.
\]
Then it follows that
\[
\Biggl| \mathbb{E}\prod_{\alpha=1}^pQ_{j_\alpha}
\mathcal Z_{j_\alpha} \Biggr| \leq(2p)^p \max
_{\mathbf s} \prod_i{
\mathbf1}(s_i\neq i) \biggl| \mathbb{E}\prod
_{\alpha}P_{A_{\alpha, {\mathbf
s}}}Q_{B_{\alpha, {\mathbf s}}}T_{ \alpha}
\biggr|.
\]
Now to prove (\ref{52}), it remains only to show that for any $\{ j_1,
\ldots, j_p \}$ and ${\mathbf s}=\{s_1, s_2,\ldots, s_p\}$ such that $
s_i\neq i $, we have
%
%
%e7.11 #&#
\begin{equation}
\label{52newf} % {\mathbf1}(j_\al\in A_{\al})
\biggl|\mathbb{E}\prod_{\alpha}P_{A_{\alpha, {\mathbf
s}}}Q_{B_{\alpha, {\mathbf s}}}T_{\alpha}
\biggr|\leq(Cp)^{2p}\mathcal Y^p\mathcal X^{2t},
\qquad t\dvtx = \bigl| \{ j_1, \ldots, j_p \} \bigr|.
\end{equation}
For simplicity, we denote $A_{\alpha, {\mathbf s}}$ and $B_{\alpha,
{\mathbf s}}$ by $A_\alpha$ and $B_\alpha$ and denote the
characteristic function ${\mathbf1}(\Xi)$ by $\Xi$. Thus we need to
show that
%
%
%e7.12 #&#
\begin{equation}
\label{52new} % {\mathbf1}(j_\al\in A_{\al})
\biggl|\mathbb{E}\prod_{\alpha}P_{A_{\alpha}}Q_{B_{\alpha
}}T_{\alpha}
\biggr|\leq(Cp)^{2p}\mathcal Y^p\mathcal X^{2t},
\qquad t\dvtx = \bigl| \{ j_1, \ldots, j_p \} \bigr|.
\end{equation}
Since $T_{1}=Q_{j_1}T_1$ and the operators $P_{A_{\alpha}}$ and
$Q_{B_{\alpha}}$ commute, we have
%
%
%e7.13 #&#
\begin{equation}
\label{3} \mathbb{E}\prod_{\alpha}P_{A_{\alpha}}Q_{B_{\alpha
}}T_{\alpha}
= \mathbb{E}(Q_{j_1} P_{A_1 } Q_{B_1 }
T_{1} ) \Biggl( \prod_{\alpha=2}^p
(P_{A_\alpha} Q_{B_\alpha}) T_{\alpha} \Biggr).
\end{equation}

Hence we can assume that $j_1 \notin\bigcap_{\alpha\neq1}A_\alpha $,
and so $1<s_1\leq p$ [see (\ref{ABals})], $j_1 \in\bigcup_{\alpha\neq
1}B_\alpha$. Similarly for $j_i$, we have $j_i \in\bigcup_{\alpha\neq
i}B_\alpha $ where $ i = 2, \ldots, p$. Recall that $j_\alpha\notin
B_\alpha$. With these two constraints, $B_\alpha$ satisfies the
inequality
%
%
%e7.14 #&#
\begin{equation}
\label{1} p+t \geq\sum_\alpha\bigl|B_\alpha
\cup\{j_\alpha\} \bigr| \ge2 t, \qquad t\dvtx = \bigl| \{ j_1,
\ldots, j_p \} \bigr|.
\end{equation}
Now it remains only to prove (\ref{52new}) under condition
(\ref{1}). First, we write
\[
\mathbb{E}\prod_{\alpha}P_{A_\alpha}Q_{B_\alpha}T_{\alpha}
= \mathbb{E}\prod_{\alpha=1}^p
(P_{A_\alpha
} Q_{\widetilde B_\alpha} \mathcal Z_{j_\alpha}),\qquad\widetilde
B_\alpha:= B_\alpha\cup\{j_\alpha\}.
\]
Using (\ref{inxy}) with $x=P \Xi Q\mathcal Z $ and $y= P\Xi^c Q\mathcal
Z $ ($x+y=P Q\mathcal Z $ ), we have
%
%
%e7.15 #&#
\begin{eqnarray}
\label{yr2}
&& \mathbb{E}\prod_{\alpha=1}^p
(P_{A_\alpha} Q_{\widetilde B_\alpha} \mathcal Z_{j_\alpha})
\nonumber\\[-8pt]\\[-8pt]
&&\qquad  =\sum_{s=1}^{p+1} \Biggl(\mathbb{E}\prod
_{i=1}^{s-1} \bigl(P_{A_i } (\Xi)
Q_{\widetilde B_i } \mathcal Z_{j_i} \bigr) \bigl(P_{A_s }
\bigl(\Xi^c \bigr)Q_{\widetilde B_s }\mathcal Z_{j_s } \bigr)
\prod_{i=s+1}^{p} (P_{A_i }
Q_{\widetilde B_i } \mathcal Z_{j_i }) \Biggr).\hspace*{-30pt}\nonumber
\end{eqnarray}
First for $s\leq p$, we use the following formula. For any bounded
functions $f$ and $h$,%
%
%e7.16 #&#
\begin{equation}
\label{eqnformCSineq} \mathbb{E}\bigl| h \bigl( P \Xi^c Q f \bigr) \bigr| \le\| h
\|_\infty\bigl\| \bigl( \Xi^c Q f \bigr) \bigr\|_2 \le
\sqrt{ \mathbb{P} \bigl( \Xi^c \bigr) } \| f \|_\infty\| h
\|_\infty.
\end{equation}
Let%
\[
h=\prod_{i=1}^{s-1} \bigl(P_{A_i }
(\Xi) Q_{\widetilde B_i } \mathcal Z_{j_i} \bigr) \prod
_{i=s+1}^{p} (P_{A_i } Q_{\widetilde B_i }
\mathcal Z_{j_i }), \qquad f=\mathcal Z_{j_s},
P=P_{A_s}, Q=Q_{\widetilde B_s}.
\]
%
%We have
% \E\Xi\prod_{\al=1}^p (P_{A_\al} Q_{\wt B_\al} \mathcal Z_{j_\al})
By (\ref{ZUNC}) and $p\geq1$, we have
\[
|h|\leq\mathcal Y^{p-1}N^{Cp},\qquad|f|\leq\mathcal Y
N^{C }.
\]
Then with (\ref{513}), we have proved that [see (\ref{yr2})]
\begin{eqnarray*}
&& \sum_{s=1}^p \Biggl(\mathbb{E}\prod
_{i=1}^{s-1} \bigl(P_{A_i } (\Xi)
Q_{\widetilde B_i } \mathcal Z_{j_i} \bigr) \bigl(P_{A_s }
\bigl(\Xi^c \bigr)Q_{\widetilde B_s }\mathcal Z_{j_s } \bigr)
\prod_{i=s+1}^{p} (P_{A_i }
Q_{\widetilde B_i } \mathcal Z_{j_i }) \Biggr)
\\
&&\qquad  \leq\mathcal Y^p N^{Cp}\exp \bigl[ - c(\log N)^{3/2}p
\bigr].
\end{eqnarray*}
Thus the contribution from the above term can be neglected in proving
(\ref{52new}). It remains only to bound the RHS of (\ref{yr2}) in the
case $s=p+1$; that is, we need to show that
%
%
%e7.17 #&#
\begin{equation}
\label{l412n} \quad \Biggl|\mathbb{E}\prod_{\alpha=1}^p
(P_{A_\alpha
} \Xi Q_{\widetilde B_\alpha} \mathcal Z_{j_\alpha}) \Biggr|
\leq(Cp)^{2p}\mathcal Y^p\mathcal X^{2t},
\qquad t\dvtx = \bigl| \{ j_1, \ldots, j_p \} \bigr|
\end{equation}
under assumption (\ref{1}). Using (\ref{511}) and (\ref{inxy}), with
$x=P\Xi\mathcal Z$ and $y=P\Xi Q \Xi^c\widetilde{\mathcal Z}$ we can
write the LHS of (\ref{l412n}) as
%
%
%e7.18 #&#
\begin{eqnarray}
\label{fn418} &&\mathbb{E}\prod_{\alpha=1}^p
(P_{A_\alpha} \Xi Q_{\widetilde B_\alpha} \mathcal Z_{j_\alpha})\nonumber
\\
&&\qquad=\sum_{s=1}^{p+1}
\Biggl(\mathbb{E} \prod_{i=1}^{s-1}
\bigl(P_{A_i } ( \Xi) \mathcal Z_{j_i, \widetilde B_i} \bigr)
\bigl(P_{A_s } (\Xi )Q_{\widetilde B_s } \bigl(\Xi^c \bigr)
\widetilde{\mathcal Z}_{j_s,
\widetilde B_s
} \bigr)
\\
&&\hspace*{152pt}
{}\times \prod_{i=s+1}^{p}(P_{A_i } \Xi Q_{\widetilde B_i } \mathcal Z_{j_i }) \Biggr).\nonumber
\end{eqnarray}
Now we repeat the argument for (\ref{yr2}). For $s\leq p$, one can use
the following formula which is similar to (\ref{eqnformCSineq}). For
any bounded function $f$ and $h$%
\[
\mathbb{E}\bigl| h \bigl( P \Xi Q \Xi^c f \bigr) \bigr| \le\| h
\|_\infty \bigl\| \bigl( \Xi^c f \bigr) \bigr\|_2 \le\sqrt{
\mathbb{P} \bigl( \Xi^c \bigr) } \| f \|_\infty\| h
\|_\infty.
\]
Let%
%
%e7.19 #&#
\begin{eqnarray}
h= \prod_{i=1}^{s-1} \bigl(P_{A_i } (\Xi) \mathcal Z_{j_i,
\widetilde B_i} \bigr) \prod_{i=s+1}^{p} (P_{A_i } \Xi Q_{\widetilde
B_i } \mathcal Z_{j_i } ),\nonumber
\\
\eqntext{f= \widetilde{\mathcal Z}_{ j_s,
{\widetilde B_s}},  P=P_{A_s}, Q=Q_{\widetilde B_s}.}
\end{eqnarray}
With the assumptions in (\ref{511a}) and (\ref{1}), we know the sum
over $1\leq s \leq p$ of RHS of (\ref{fn418}) is bounded above by
\[
Y^p N^{Cp}\exp \bigl[ - c(\log N)^{3/2}p
\bigr],
\]
which can be neglected in proving (\ref{l412n}). For the main term,
with $s=p+1$ on the RHS of (\ref{fn418}), using (\ref{511a}) and
(\ref{1}), we have
\[
\mathbb{E}\prod_{\alpha=1}^p
(P_{A_\alpha} \Xi \mathcal Z_{j_\alpha, \widetilde B_\alpha
})\leq(C\mathcal
Y)^p(C_0 \mathcal Xp)^{2t}\leq \bigl(C
\mathcal Yp^2 \bigr)^p \mathcal X^{2t}
\]
and this completes the proof of Lemma~\ref{abstractZlemma}.
\end{pf*}

%s7.2 #&#
\subsection{A stronger bound on $[Z]$} In this section we are going
to apply Lemma~\ref{abstractZlemma} to prove a~stronger bound on $[Z]$.
We note that using (\ref{eqnGii}) and (\ref{eqnZi}), {$Z$} can be
written as%
%
%e7.20 #&#
\begin{equation}
\label{Z1} Z_i = Q_i \biggl[ \frac{-1}{ G_{ii}}
\biggr], \qquad Q_i:=1-P_i,\qquad P_i:=
\mathbb{E}_{{\mathbf x}_i}.
\end{equation}
%

%le7.4 #&#
\begin{lemma}\label{lemaZ2}
Let $\mathcal Z_i=(G_{ii})^{-1}$, $P_i$ and $Q_i$ defined as in
(\ref{Z1}). We assume that $\eta=\Im z\geq N^{-C}$ for some $C>0$.
Suppose there exists an even integer $p$ and an event $\Xi$, such that
$\mathbb{P}(\Xi^c)\leq e^{-p (\log N)^{3/2}} $, and in $\Xi$,
%
%
%e7.21 #&#
\begin{eqnarray}
\label{yatt}
\max_i| Q_i\mathcal Z_i| &\leq& C \mathcal Y\mathcal X,\qquad\frac{\Lambda_o(z)}{ \min_i|G_{ii}(z)|}\leq C
\mathcal X\ll1,
\nonumber\\[-8pt]\\[-8pt]
\min_i\bigl|G_{ii}(z)\bigr|&\geq& \mathcal Y^{-1},\qquad p \leq\frac{C}{(\log N)\mathcal X},\nonumber
\end{eqnarray}
where $\mathcal X\ll1$ and $\mathcal Y$ are deterministic numbers. Then
there exists $\Xi' $ with $\mathbb{P}((\Xi')^c)\leq e^{ -p } $
and in $\Xi'$,%
%
%e7.22 #&#
\begin{equation}
\label{nnz} \biggl|\frac{1}{N} \sum_iQ_i
\mathcal Z_i \biggr|\leq Cp ^{5 } \bigl(\mathcal
X^2+N^{-1} \bigr) \mathcal Y.
\end{equation}
\end{lemma}
\begin{pf} %{Proof of Lemma~\ref{lemaZ2}}
We are going to apply Lemma~\ref{abstractZlemma}. The claim given in
(\ref{nnz}) will follow from (\ref{52}) and Markov's inequality. Using
the hypothesis, one can easily verify (\ref{ZUNC}) and (\ref{513}) in
the hypotheses of Lemma~\ref{abstractZlemma}. It remains only to show
that for $i\in A\subset\{1,2,\ldots, N\}$ and $|A|\leq p$, there exist
$\mathcal Z_{i,A} $ and $ \widetilde{\mathcal Z}_{i,A} $ such that
%
%
%e7.23 #&#
%e7.24 #&#
\begin{eqnarray}
\label{511new} {\mathbf1}(\Xi) ( Q_{A}\mathcal Z_i )=
\mathcal Z_{i,A}+ {\mathbf1}(\Xi) Q_{A} \bigl(
\Xi^c \bigr)\widetilde{\mathcal Z}_{i,A},
\nonumber\\[-8pt]\\[-8pt]
\eqntext{\mathcal Z_{i,A} \leq\mathcal Y \bigl(C \mathcal X|A| \bigr)^{ |A|}, \widetilde{
\mathcal Z}_{i,A}\leq\mathcal Y N^{C|A|}}
\end{eqnarray}
for some $C>0$. By assumption, formula (\ref{511new}) holds when
$A=\{i\}$. Thus we assume that $|A|\geq2$. As in Lemma 5.1 in
\cite{EKYY11}, let $\mathcal A=\mathcal A(H)=\mathcal A(X^\dagger X )$
be a~function of $X^\dagger X $, and define
\[
(\mathcal A )^{S,U}:= \sum_{S\setminus
U\subset V\subset
S}(-1)^{|V|}
\mathcal A^{(V)}, \qquad A^{(V)}:= A \bigl(
\bigl(X^{(V)} \bigr)^\dagger \bigl(X^{(V)} \bigr)
\bigr)
\]
for any $S,U \subset\{1,2,\ldots,N\}$. Then we have
\[
\mathcal A=\sum_{U\subset S} (\mathcal A
)^{S,U}.
\]
By definition, $ (\mathcal A )^{S,U}$ % is independent of the
%$j$-th column and row of $H$, i.e.,
is independent of the $j$th column of $X$ if $j\in S\setminus U$.
Therefore,%
\[
Q_S\mathcal A=Q_S (\mathcal A )^{S,S}.
\]
In our case,
\[
Q_A\mathcal Z_i=Q_iQ_{A \setminus\{i\}}
\mathcal Z_i =Q_{A} \biggl(\frac{1}{G_{ii}}
\biggr)^{ {A \setminus\{i\}},{A \setminus
\{i\}}}.
\]
Now we choose
\[
\mathcal Z_{i,A}:= {\mathbf1}(\Xi) Q_{A}\Xi \biggl(
\frac
{1}{G_{ii}} \biggr)^{ {A \setminus\{i\}},{A \setminus\{i\}}},\qquad \widetilde{\mathcal
Z}_{i,A}:= \biggl(\frac{1}{G_{ii}} \biggr)^{ {A \setminus\{
i\}},{A \setminus\{i\}}}.
\]
It is easy to prove the bound for $ \widetilde{\mathcal Z}_{i,A}$ in
(\ref{511new}) using its definition. For bounding~$ {\mathcal
Z}_{i,A}$, it remains only to prove that, for $2\leq|A|\leq p_N$,
%
%
%e7.25 #&#
\begin{equation}
\label{g3s} \biggl|{\mathbf1}(\Xi) \biggl(\frac{1}{G_{ii}}
\biggr)^{ {A/\{i\}},{A/\{
i\}}} \biggr|\leq\mathcal Y \bigl(C \mathcal X|A| \bigr)^{ |A|}.
\end{equation}
To prove this, we first show that for $|\mathbb T|\leq p$,
%
%
%e7.26 #&#
\begin{equation}
\label{yttg} \max_{i,j\notin\mathbb T } \bigl|G_{ij}^{(\mathbb T)}\bigr|
\leq C \max_{i,j } |G_{ij} |, \qquad\min
_{i\notin\mathbb T } \bigl|G_{ii}^{(\mathbb T)}\bigr| \geq c \min
_{i } |G_{ii} |
\end{equation}
with the constants $C,c$ independent of $N,i,j$. We start from
$|\mathbb T|=1$, that is, \mbox{$\mathbb T=\{k\}$}. First using
(\ref{eqnGijGijk}) and the hypotheses of this lemma, we have
\begin{eqnarray*}
(G_{ii})^{-1} &=&\frac
{-G_{ij}G_{ji}}{G_{ii}G_{jj}G_{ii}^{(j)}}+ \bigl(G^{(j)}_{ii}
\bigr)^{-1}= \bigl(1+O \bigl(\mathcal X^2 \bigr) \bigr)
\bigl(G^{(j)}_{ii} \bigr)^{-1},
\\
\bigl|G^{(k)}_{ij}\bigr| &=& \biggl| G_{ij}-
\frac{G_{ik}G_{kj}}{G_{kk}} \biggr|\leq\Lambda_o \bigl(1+O(\mathcal X) \bigr).
\end{eqnarray*}
It follows that
\[
\max_{i,j\neq k } \bigl|G_{ij}^{(k)}\bigr| \leq \bigl(1+O(
\mathcal X) \bigr) \max_{i,j } |G_{ij} |, \qquad\min
_{i\neq k } \bigl|G_{ii}^{(k)}\bigr| \geq \bigl(1-O(
\mathcal X) \bigr) \min_{i } |G_{ii} |.
\]
Then using induction on $| \mathbb T|$ and the assumption $\mathcal
Xp\ll1$, we obtain the desired result (\ref{yttg}).

Now we return to prove (\ref{g3s}) for the case $|A|=2$. If $i\neq j$,
using (\ref{eqnGijGijk}), (\ref{yttg}) and (\ref{yatt}), we have
\[
\biggl(\frac{1}{G_{ii}} \biggr)^{
{j},j}=(G_{ii})^{-1}-
\bigl(G^{(j)}_{ii} \bigr)^{-1}=\frac
{-G_{ij}G_{ji}}{G_{ii}G_{jj}G_{ii}^{(j)}}
\leq O \bigl(\mathcal Y\mathcal X^2 \bigr).
\]
The general case has been proved in Lemma 5.11 of \cite{EKYY11} (also
see below), which gives that
\[
\biggl(\frac{1}{G_{ii}} \biggr)^{ {A/\{i\}},{A/\{i\}}}\leq\bigl(C|A|\bigr)^{|A|}
\frac{
(\max_{i,j\notin\mathbb T, \mathbb T\subset{A/\{i\}}}
|G_{ij}^{(\mathbb T)}| )^{|A|}
} {
(\min_{j\notin\mathbb T, \mathbb T\subset{A/\{i\}}
}|G_{jj}^{(\mathbb T)}| )^{|A|+1}}.
\]
Together with (\ref{yttg}) and (\ref{yatt}), we obtain (\ref{g3s}) for
$|A| = 2$.

Finally we need to point out that the definition of $G_{ij}^{(V)}$
($ij\notin V$) in \cite{EKYY11} is different from the definition in our
paper, although they are equivalent. We have
\[
G ^{(V)}= \bigl( \bigl(X^{(V)} \bigr)^\dagger
\bigl(X^{(V)} \bigr)-z \bigr)^{-1}
\]
and \cite{EKYY11} has
\[
G ^{(V)}= \bigl( H^{(V)}-z \bigr)^{-1},
\]
where $H^{(V)}$ is the minor of $H$ obtained by removing all $i$th rows
and columns of $H$ indexed by $i \in V$. But one can see that if $H=X
^\dagger X$, then $H^{(V)}=(X^{(V)})^\dagger(X^{(V)})$. Thus we finish
the proof of Lemma~\ref{lemaZ2}.
\end{pf}
%
%Note: the size of $zZ_i$ is proved in \eqref{eqnZibd}.
Finally we give the proof of the main result of this section.

\begin{pf*}{Proof of Lemma~\ref{lemZlem}}
It is a~special case of Lemma~\ref{lemaZ2} with $\mathcal
X=K\widetilde\Psi$ and $\mathcal Y=C$ for a~constant $C$ (possibly
large, but independent of $N$). First, the bound $\max_i| Q_i\mathcal
Z_i|\leq C\mathcal Y\mathcal X$ is proved in (\ref{eqnZibd}). By
assumption, if $\Xi\subset\bigcap_{z \in{ \mathbf{S}}(L)} ( \Gamma(z,
K)\cap \mathbf{B}^{c}(z) )$, then
\[
\Lambda_o, \Lambda_d \leq K\Psi\leq K\widetilde
\Psi=X\leq CK(N\eta)^{-1/2}\ll1
\]
in $\Xi$. Thus we obtain%
\[
\frac{\Lambda_o(z)}{ \min_i|G_{ii}(z)|}\leq C\mathcal X\ll1, \qquad \min_i\bigl|G_{ii}(z)\bigr|
\geq\mathcal Y^{-1}.
\]
Furthermore formula (\ref{kk20}) and $\eta\geq N^{-1}\varphi^L$ [since
$z\in S(L)$] imply that $p \leq C((\log N)\mathcal X)^{-1}$, and the
proof of Theorem~\ref{lemZlem} is finished.
\end{pf*}

%s8 #&#
\section{Strong Marcenko--Pastur law and rigidity of
eigenvalues}\label{secp45} In this section, our goal is to prove
Theorems~\ref{451} and~\ref{452}. Throughout this section, we will
assume that $\lim_{N \to\infty}  d_N \in(0,\infty) \setminus\{1\}$.

Let us first give a~brief sketch of the proof strategy for the main
technical estimate (\ref{Lambdafinal}). We will prove, by an induction
on the exponent $\tau$, that $\Lambda(z) \le(N\eta)^{-\tau}$ holds
modulo logarithmic factors with high probability. Notice that we have
already proved this statement for $\tau=1/4$ in Theorem
\ref{thmdetailed}. Lemma~\ref{lemfm} asserts that if this statement is
true for some $\tau$, then it also holds for $\frac{1+\tau}{2}$
assuming a~bound on~$[Z]$. Now, an application of Lemma~\ref{lemZlem}
will yield that the required bound for~$[Z]$ holds with high
probability. Repeating the induction step for $O(\log\log N)$ times, we
will obtain that $\tau$ is essentially one, implying Theorem~\ref{451}.
However, we must keep track of the increasing logarithmic factors and
the deteriorating probability estimates of the exceptional sets.
%s8.1 #&#
\subsection{Proof of Theorem \protect\texorpdfstring{\ref{451}}{3.1}} We start by establishing
(\ref{Lambdafinal}) and (\ref{Lambdaofinal}).
\begin{pf*}{Proof of (\ref{Lambdafinal}) and (\ref{Lambdaofinal})}
Without loss of generality, we assume $ \zeta\geq1$. Using Lemma
\ref{lemexpnullset} and Theorem~\ref{thmdetailed}, for any $ \zeta>0$,
there exists $C_\zeta$ such that
%
%
%e8.1 #&#
\begin{equation}
\label{xy33} \Xi_1\subset\bigcap_{z\in\mathbf{S}(C_\zeta
)}
\mathbf{B}^c(z)\cap\Gamma(z, C_\zeta)
\end{equation}
holds with $(\zeta+4)$-high probability. Then from Lemma~\ref{corDm},
we see that for $z\in\mathbf{S}(3C_\zeta)$,
%
%
%e8.2 #&#
\begin{equation}
\label{mll} \bigl|\mathcal D(m) (z)\bigr|\leq\varphi^{2C_\zeta}
\Psi^2+\bigl|[Z]\bigr|\qquad\mbox{in }\Xi_1.
\end{equation}
Let $\Lambda_1=1$, so that $\Lambda\leq\Lambda_1$ in $\Xi_1$.
Therefore, we can apply Lemma~\ref{lemZlem} with
\[
p =p_{1}=-\log \bigl[1-\mathbb{P}(\Xi_1) \bigr]/ (\log
N)^2.
\]
Without loss of generality, we can assume that $\mathbb{P}(\Xi_1)$ is
not too close to 1; otherwise, we can replace $\Xi_1$ by a~subset of
itself. It follows that
\[
p_{ 1}=C\varphi^{ \zeta+4}/(\log N)^2.
\]
We assume that $C_\zeta\geq6\zeta$ and therefore (\ref{kk20}) holds.
Then (\ref{32you}) gives that, for $z\in\mathbf{S}(3C_\zeta)$, there
exists $\Xi_2$ such that
\[
\Xi_2\subset\Xi_1, \qquad\mathbb{P}(\Xi_2
)=1-e^{-p_1}
\]
and
\[
\bigl| [Z] \bigr| \leq\varphi^{2C_\zeta+11\zeta} \Psi_1^2, \qquad
\Psi_1:=\sqrt{\frac{\Im m_W+\Lambda_1}{N\eta}}\qquad\mbox{in }\Xi_2.
\]
Since in $\Xi_2\subset\Xi_1$, by (\ref{mll}), $\Lambda\leq\Lambda_1$
and thus $\Psi\leq\Psi_1 $ in $\Xi_2$, and consequently
%
%
%e8.3 #&#
\begin{equation}
\label{cl33} \bigl|\mathcal D(m) (z)\bigr|\leq\varphi^{2C_\zeta+11}
\frac{\Im m_W+\Lambda_1}{N\eta}\qquad\mbox{in }\Xi_2.
\end{equation}
Then applying Lemma~\ref{lemfm}, (\ref{eqnsylx2}) shows that, for
$z\in\mathbf{S}(3C_\zeta)$,
\[
\Lambda(z)\leq\Lambda_2(z):= \varphi^{C_\zeta+6\zeta}
\Lambda_1^{1/2}(N\eta)^{-1/2}\qquad\mbox{in }
\Xi_2.
\]
Now the proof proceeds via iterating the above process. Indeed, by
choosing
\[
p_{ 2}=-\log \bigl[1-\mathbb{P}(\Xi_2) \bigr]/ (\log
N)^2=C\varphi^{
\zeta+4}/(\log N)^4
\]
we deduce that there exists $\Xi_3$ such that
\[
\Xi_3\subset\Xi_2, \qquad\mathbb{P}(
\Xi_3)=1-e^{-p_2}
\]
and, for $z\in\mathbf{S}(3C_\zeta)$,
\[
\Lambda(z)\leq\Lambda_3(z):= \varphi^{C_\zeta+6\zeta}
\Lambda_2^{1/2}(N\eta)^{-1/2} \leq
\varphi^{2C_\zeta+12\zeta}(N\eta)^{-3/4}\qquad\mbox{in }\Xi_3.
\]
We iterate this process $K$ times, $K:=\log\log N/(\log1.9)$. For
$k\leq K$, we infer that for some
\[
\Xi_{k}\subset\Xi_{k-1}, \qquad\mathbb{P}(
\Xi_k)=1-e^{-p_{k-1}},
\]
where
\[
p_{ k}=-\log \bigl[1-\mathbb{P}(\Xi_{k-1}) \bigr]/ (\log
N)^2=C\varphi^{ \zeta+4}/(\log N)^{2k}\geq
\varphi^{ \zeta}
\]
and, for $z\in\mathbf{S}(3C_\zeta)$,
%
%
%e8.4 #&#
\begin{eqnarray}
\Lambda(z)\leq\Lambda_{k+1}(z):= \varphi^{C_\zeta+6\zeta}
\Lambda_k^{1/2}(N\eta)^{-1/2}\leq \varphi^{2C_\zeta+12\zeta}
(N\eta)^{-1+(1/2)^{k }}\nonumber
\\[-4pt]
\eqntext{\mbox{in } \Xi_{k+1}.}
\end{eqnarray}
Note that%
\[
N^{(1/2)^{K }}\leq\varphi.
\]
Thus for $k=K$ and $z\in\mathbf{S}(3C_\zeta)$, the bound
%
%
%e8.5 #&#
\begin{equation}
\label{ycbb} \Lambda(z)\leq\Lambda_{k+1}(z) \leq
\varphi^{2C_\zeta+12\zeta} (N\eta)^{-1+(1/2)^{K }} \leq\varphi ^{2C_\zeta+12\zeta+1} (N
\eta)^{-1}
\end{equation}
holds with $\zeta$-high probability, and this completes the proof of
(\ref{Lambdafinal}). Furthermore, since $\Xi_{K+1}\subset\Xi_1$ with
(\ref{xy33}), we obtain (\ref{Lambdaofinal}).
\end{pf*}
Next we assume (\ref{443}) holds and prove (\ref{4444}) first.
\begin{pf*}{Proof of (\ref{4444})}
Using (\ref{Lambdaofinal}), we have for any $i$,
%
%
%e8.6 #&#
\begin{equation}
\label{tbca} \max_{\lambda_-/5\leq E\leq5\lambda_+} \Im G_{ii} \bigl(E+i
\varphi^{C_\zeta}N^{-1} \bigr)\leq C.
\end{equation}
By definition,
\[
\Im G_{ii}=\sum_\alpha
\frac{|\mathbf{v}_\alpha(i)|^2\eta
}{(\lambda_\alpha-E)^2+\eta^2}.
\]
Then choosing $E=\lambda_\alpha$ and $\eta=\varphi^{C_\zeta }N^{-1}$,
using (\ref{tbca}), we deduce that for any index~$\alpha$
\[
\bigl|\mathbf{v}_\alpha(i)\bigr|^2\leq\eta=\varphi^{C_\zeta}N^{-1},
\]
which implies (\ref{4444}). Here formula (\ref{443}) guarantees that
$\lambda_-/5\leq E\leq5\lambda_+ $.
\end{pf*}

Now to establish Theorem~\ref{451}, all that remains is the proof of
(\ref{443}) which we give below.

\begin{pf*}{Proof of (\ref{443})}
The proof proceeds via taking the following four steps:
\begin{itemize}
\item \textit{Step} 1. For any $\zeta>0$, there exists some
    $D_\zeta>0$
such that%
\[
\max\{\lambda_j\dvtx \lambda_j\leq5\lambda_+ \} \leq
\lambda_++N^{-2/3}\varphi^{4D_\zeta}
\]
and
\[
\min\{\lambda_j\dvtx \lambda_j \ge{ {
\mathbf1}_{d>1}} \lambda_-/5 \} \ge\lambda_- -N^{-2/3}
\varphi^{D_\zeta}
\]
hold with $\zeta$-high probability.
\item \textit{Step} 2. Recall $ {\mathfrak n}(E)$ and $n_c(E)$ from
    (\ref{deffn}) and (\ref{nsc}).
We will show that%
%
%e8.7 #&#
%e8.8 #&#
\begin{eqnarray}
\label{fixE1E2} \bigl| \bigl({\mathfrak n}(E_1)-{\mathfrak
n}(E_2) \bigr)- \bigl( n_c(E_1)-n_c(E_2)
\bigr) \bigr| \le\frac{C (\log N)\varphi
^{C_\zeta}}{N},
\nonumber\\[-8pt]\\[-8pt]
\eqntext{E_1, E_2\in[ {
\mathbf1_{d>1} } \lambda_-/4, 4\lambda_+ ],}
\end{eqnarray}
which implies that
%
%
%e8.9 #&#
\begin{equation}
\label{39ll} \# \bigl\{j\dvtx \lambda_j\notin[ {\mathbf
1_{d>1} }\lambda_-/5, 5\lambda_+ ] \bigr\}\leq\varphi^{C_\zeta}.
\end{equation}
We note that though we need only (\ref{39ll}) for (\ref{443}), but
(\ref{fixE1E2}) will be used later to prove Theorem~\ref{452}.
\item \textit{Step} 3. Next, using the above two steps we will show
    that $ \max_j\lambda_j \leq5\lambda_+ $, with $\zeta$-high
    probability. This step will imply (\ref{443}) in the case
    $d<1$.
\item \textit{Step} 4. Finally, we show that, for $d>1$, that is,
    $N>M$, we have $ \lambda_M\geq\lambda_-/5$,
with $\zeta$-high probability. %The case $N < M$ follows trivially
%from above by
%interchanging the roles of $M$ and $N$.
\end{itemize}

Step 1 of proof of (\ref{443}). By repeating the iteration in the proof
of (\ref{443}) one more time, that is, replacing $\Lambda_1$ in
(\ref{cl33}) with $\Lambda_{k+1}$ in (\ref{ycbb}), we obtain
\[
\bigl|\mathcal D(m) (z)\bigr|\leq\varphi^{ C _\zeta}\frac{\Im
m_c+(1/N\eta) }{N\eta}
\]
for some large $C_\zeta$. From (\ref{eqnsylx}) again, we obtain that
for some $D_\zeta\geq1$
%
%
%e8.10 #&#
\begin{equation}
\label{3yc} \Lambda(z)\leq\varphi^{D_\zeta}\frac{ \delta}{\sqrt {\kappa+\eta+\delta
}},\qquad
\delta:= \biggl( \frac{\Im m_c}{N\eta} +\frac{1}{(N\eta
)^2} \biggr).
\end{equation}
For any $E$ such that $E\geq\lambda_++N^{-2/3}\varphi^{4D_\zeta}$, and
\[
\eta:=\varphi^{-D_\zeta}N^{- 1/2}\kappa^{ 1/4},\qquad
\kappa=E-\lambda_+
\]
(thus $\kappa\geq N^{-2/3}\varphi^{4D_\zeta} $), it is easy to check
that
%
%
%e8.11 #&#
\begin{equation}
\label{TT3} \kappa\gg\varphi^{D_\zeta}\eta,\qquad N\eta\sqrt \kappa\gg
\varphi^{D_\zeta},\qquad\frac{\sqrt\kappa}{N\eta^2}\gg1.
\end{equation}
Using (\ref{esmallfake}) and (\ref{TT3}), we have
%
%
%e8.12 #&#
\begin{equation}
\label{zmm} \Im m_c(z)=C\frac{\eta}{ \sqrt\kappa},
\end{equation}
which implies
\[
\delta\leq\frac{C}{N\sqrt\kappa}+(N\eta)^{-2}.
\]
Therefore, $\kappa\geq\delta$. Together with (\ref{3yc}) and
(\ref{TT3}), we have
\[
\Lambda(z)\leq C\varphi^{D\zeta} \biggl(\frac{\eta}{\kappa}+
\frac
{1}{N\eta\sqrt\kappa} \biggr) \frac{1}{N\eta}\ll\frac{1}{N\eta}.
\]
Combining (\ref{zmm}) and the last inequality of (\ref{TT3}) yields
\[
\Im m_c(z)\ll\frac{1}{N\eta}
\]
and therefore we can conclude that
\[
\Im m (z)\ll\frac{1}{N\eta}.
\]
Note that if $\Im m (z)< (2N\eta)^{-1}$ (recall $z=E+i\eta$), then the
number of the eigenvalues in the interval $[E-\eta, E+\eta]$ is zero,
which is implied by the
following observation:%
%
%e8.13 #&#
\begin{equation}
\label{qpn} \Im m(z) = \frac{1}N\sum_\alpha
\frac{\eta}{(\lambda_\alpha-E)^2+\eta^2} \geq\sum_{\alpha\dvtx
|\lambda_\alpha-E|\leq\eta}
\frac{1}{2N\eta}.
\end{equation}
Since $\Im m (z)\ll\frac{1}{N\eta}$ holds for any $E\geq\lambda
_++N^{-2/3}\varphi^{4D_\zeta}$, we have proved that for any $\zeta>0$,
there exists some $D_\zeta>0$ such that
\[
\max\{\lambda_j\dvtx \lambda_j\leq5\lambda_+ \} \leq
\lambda_++N^{-2/3}\varphi^{4D_\zeta}
\]
holds with $\zeta$-high probability. An analogous bound for the
smallest eigenvalue can be proved similarly.

Step 2 of proof of (\ref{443}). The proof is similar to that of Theorem
2.2 in \cite{EYYrigid}. The strategy is to translate the information on
the Stieltjes transform obtained in Theorem~\ref{451} to prove
(\ref{fixE1E2}) on the location of the eigenvalues.
%The following lemma is a~special case of Lemma 6.1 proved in

In the following lemma, $A_1, A_2$ represent two numbers with $|A_{1} +
A_2|\leq O(1)$. For any $E_1, E_2 \in[A_1, A_2 ]$, and $\eta=N^{-1}$ we
define
\[
f(\lambda):= f_{E_1,E_2,\eta}(\lambda)
\]
to be the characteristic function of $[E_1, E_2]$ smoothed on scale
$\eta$, that is, $f\equiv1$ on $[E_1+\eta, E_2-\eta]$, $f\equiv0$ on
$\mathbb{R}\setminus[E_1, E_2 ]$ and $|f'|\le C\eta^{-1}$, $|f''|\le
C\eta^{-2}$.

%
%
%le8.1 #&#
\begin{lemma}\label{lmHS1new}
Let $\varrho^\Delta$ be a~signed measure on the real line and
$m^\Delta$ be the Stieltjes transform of $\varrho^\Delta$. Suppose for
some positive number $U$ (which may depend on $N$) we have
%
%
%e8.14 #&#
\begin{equation}
\bigl|m^\Delta(x+iy)\bigr|\le\frac{CU}{Ny} \qquad\mbox{for } y<1,  x
\in[A_1, A_2 ]. \label{trivv1new}
\end{equation}
Then
%
%
%e8.15 #&#
\begin{equation}
\biggl|\int_\mathbb{R}f_{E_1,E_2,\eta}(\lambda)
\varrho^\Delta(\lambda)\,\mathrm{d}\lambda \biggr|\le\frac
{CU|\log\eta|}{ N }.
\label{genHS1new}
\end{equation}
\end{lemma}

\begin{pf} %{Proof of Lemma~\ref{lmHS1new}}
For notational simplicity, we drop the $\Delta$ superscript in the
proof. Let $\chi(y)$ be a~smooth cutoff function with support in
$[-1,1]$, with $\chi(y)= 1$ for $ |y|\leq1/2$ and with bounded
derivatives. Using Helffer--Sjostrand functional calculus, we obtain
\[
f(\lambda)=\frac{1}{2\pi}\int_{\mathbb{R}^2}\frac{iyf''(x)\chi
(y)+i(f(x)+iyf'(x))\chi'(y)}{\lambda-x-iy}
\,\mathrm{d}x \,\mathrm{d}y.
\]
Since $f$ and $\chi$ are real,
\begin{eqnarray}
\nonumber
\qquad \biggl|\int f(\lambda)\varrho(\lambda)\,\mathrm{d}\lambda \biggr| &\leq& C
\int_{\mathbb{R}^2} \bigl( \bigl|f(x)\bigr| +|y| \bigl|f'(x)\bigr| \bigr) \bigl|
\chi'(y)\bigr| \bigl| m(x+iy)\bigr| \,\mathrm{d}x \,\mathrm{d}y
\nonumber
\\
&&{} +C \biggl|\int_{|y|\leq\eta}\int y f''(x)
\chi(y) \Im m(x+iy)\,\mathrm{d}x \,\mathrm{d}y \biggr|\label{intr2fe1}
\\
&&{}+C \biggl|\int_{|y|\geq\eta}\int_\mathbb{R}y
f''(x)\chi(y) \Im m(x+iy)\,\mathrm{d}x \,\mathrm{d}y\biggr|.
\nonumber
\end{eqnarray}
Using (\ref{trivv1new}), the first term can be estimated as
%
%
%e8.16 #&#
\begin{equation}
\label{zll} \int_{\mathbb{R}^2} \bigl( \bigl|f(x)\bigr| +|y|
\bigl|f'(x)\bigr| \bigr) \bigl|\chi'(y)\bigr| \bigl| m(x+iy)\bigr| \,\mathrm{d}x
\,\mathrm{d}y \le CU.
\end{equation}
For the second term on the RHS of (\ref{intr2fe1}), notice that from
(\ref{trivv1new}) it follows that, for any $0 < y \leq1$,
%
%
%e8.17 #&#
\begin{equation}
y \bigl|\Im m(x+iy)\bigr|\le CU. \label{ym}
\end{equation}
With $|f''|\leq C\eta^{-2}$ and
%
%
%e8.18 #&#
\begin{equation}
\operatorname{supp} f'(x)\subset\bigl\{|x-E_1|\leq\eta\bigr\}
\cup\bigl\{ |x-E_2|\leq\eta\bigr\}, \label{fpr}
\end{equation}
we get
\[
\biggl|\int_{|y|\leq\eta}\int y f''(x)
\chi(y) \Im m(x+iy)\,\mathrm{d}x \,\mathrm{d}y \biggr|\le CU.
\]
Now we integrate the third term in (\ref{intr2fe1}) by parts first in
$x$, then in $y$. Then we bound it in absolute value by
%
%
%e8.19 #&#
\begin{eqnarray}
\label{temp7501} \qquad\quad &&C\int_{\mathbb{R}}\eta\bigl|f'(x)\bigr| \bigl|\Re
m(x+i\eta)\bigr|\,\mathrm{d}x+ C\int_{\mathbb{R}^2}y \bigl|f'(x)
\chi'(y)\Re m(x+iy)\bigr| \,\mathrm{d}x \,\mathrm{d}y
\nonumber\\[-8pt]\\[-8pt]
&&\qquad{} +\frac{C}\eta\int_{\eta\le y\leq1}\int
_{\operatorname{supp} f'} \bigl|\Re m(x+iy)\bigr|\,\mathrm{d}x \,\mathrm{d}y.\nonumber
\end{eqnarray}
By using (\ref{trivv1new}) and (\ref{fpr}) in the first term,
(\ref{zll}) in the second and (\ref{trivv1new}) in the third, we have
\[
(\ref{temp7501})\leq CU+ CU\eta^{-1} \int
_{\operatorname{supp} f'}\,\mathrm{d}x \int_{\eta\le y\le1}
\frac{1}{yN} \,\mathrm{d}y \leq CU|\log\eta|.
\]
This completes the proof of Lemma~\ref{lmHS1new}.
\end{pf}

We will apply Lemma~\ref{lmHS1new} with $[A_1, A_2]\subset[
{\mathbf1_{d>1} }\lambda_-/4, 4\lambda_+ ] $ and the signed measure
$\varrho^{\Delta}$ equal to the difference of the empirical density and
the MP law,%
\[
\varrho^\Delta(\mathrm{d}\lambda)=\varrho(\mathrm{d}\lambda) -
\varrho_c(\lambda)\,\mathrm{d} \lambda,\qquad\varrho(\mathrm{d}
\lambda):=\frac{1}{N}\sum_i \delta(
\lambda_i-\lambda).
\]

Now we prove that (\ref{fixE1E2}) holds. By Theorem~\ref{451}, if $y\ge
y_0:= \varphi^{C_\zeta}/N $, the assumptions of Lemma~\ref{lmHS1new}
hold for the difference $m^\Delta=m-m_c$ and $U=\varphi^{C_\zeta}$. For
$y\le y_0 $, set $z=x+iy$, $z_0=x+iy_0$ and estimate
%
%
%e8.20 #&#
\begin{eqnarray}
\label{mmmsc}
&& \bigl| m(z)-m_c(z)\bigr|
\nonumber\\[-8pt]\\[-8pt]
&&\qquad \le\bigl| m(z_0)-m_c(z_0)\bigr|
+ \int_y^{y_0} \bigl| \partial_\eta
\bigl( m(x+i\eta)-m_c(x+i\eta) \bigr) \bigr|\,\mathrm{d}\eta.\nonumber
\end{eqnarray}
Note that%
\begin{eqnarray*}
\bigl|\partial_\eta m(x +i\eta)\bigr| &= & \biggl|\frac{1}{N}\sum
_j \partial_\eta G_{jj}(x+i\eta) \biggr|
\\
\nonumber
&\le& \frac{1}{N}\sum_{jk}
\bigl|G_{jk}(x+i\eta)\bigr|^2 = \frac{1}{N\eta}\sum
_j \Im G_{jj}(x+i\eta) = \frac{1}{\eta}\Im
m(x+i\eta)
\end{eqnarray*}
and similarly
\[
\bigl|\partial_\eta m_c(x+i\eta)\bigr| = \biggl|\int\frac{\varrho_c(s)}{(s-x-i\eta
)^2}
\,\mathrm{d}s \biggr| \le\int \frac{\varrho_c(s)}{|s-x-i\eta|^2}\,\mathrm{d}s = \frac{1}{\eta} \Im
m_c(x+i\eta).
\]
Now we use the fact that the functions $y\to y\Im m(x+iy)$ and $y\to
y\Im m_{W}(x+iy)$ are monotone increasing for any $y>0$ since both are
Stieltjes transforms of a~positive measure. Therefore the integral in
(\ref{mmmsc}) can be bounded by%
%
%e8.21 #&#
\begin{eqnarray}
\label{intbb}
&& \int_y^{y_0}
\frac{\mathrm{d}\eta}{\eta} \bigl[ \Im m(x+i\eta) + \Im m_{W}(x+i\eta) \bigr]
\nonumber\\[-8pt]\\[-8pt]
&&\qquad \le y_0 \bigl[ \Im m(z_0) + \Im m_{W}(z_0)
\bigr] \int_y^{y_0} \frac{\mathrm{d}\eta}{\eta^2}.\nonumber
\end{eqnarray}

By definition, $\Im m_c(x+iy_0) \leq|m_c(x+iy_0)| \le C$. By the choice
of $y_0$ and Theorem~\ref{451}, we have
%
%
%e8.22 #&#
\begin{equation}
\Im m(x+iy_0) \le\Im m_c(x+ i y_0) +
\frac{\varphi^{C_\zeta}} {N y_0} \le C \label{imest}
\end{equation}
with $\zeta$-high probability for any $\zeta>0$. Together with (\ref
{intbb}) and (\ref{mmmsc}), this proves that (\ref{trivv1new}) holds
for $y\le y_0 $ as well if $U$ is increased to $U=C\varphi^{C_\zeta}$.

The application of Lemma~\ref{lmHS1new} shows that, for any $\eta\ge
1/N$,%
%
%e8.23 #&#
\begin{equation}
\qquad \biggl|\int_\mathbb{R}f_{E_1,E_2,\eta}(\lambda)\varrho(
\lambda) \,\mathrm{d}\lambda-\int_\mathbb{R}f_{E_1,E_2,\eta
}(
\lambda) \varrho_c(\lambda)\,\mathrm{d}\lambda \biggr|\le
\frac{C (\log
N)\varphi^{C_\zeta}}{N}. \label{genHS2new}
\end{equation}
Using the fact $y\to y\Im m(x+iy)$ is monotone increasing for any
$y>0$, we now use (\ref{imest}) to deduce a~crude upper bound on the
empirical density. Indeed, for any interval $I:=[x-\eta, x+\eta]$, with
$\eta=1/N$, we have
%
%
%e8.24 #&#
\begin{equation}
\qquad {\mathfrak n}(x+\eta)- {\mathfrak n}(x-\eta) \le C\eta\Im m ( x+ i\eta)\le
Cy_0 \Im m ( x+ iy _0 ) \le\frac{C\varphi^{C_\zeta}}{N}.
\label{lemdensity}
\end{equation}
Formulas (\ref{genHS2new}) and (\ref{lemdensity}) yield (\ref{fixE1E2})
and we have achieved Step 2.

Step 3 of proof of (\ref{443}): now we prove $ \lambda_1\leq5\lambda_+
$ holds with $\zeta$-high probability. Note that there is nothing
special about the number $5$ and below we show that some large
$K$,%
\[
\lambda_1\leq K\lambda_+
\]
with $\zeta$-high probability. Let
%
%
%e8.25 #&#
\begin{equation}
\label{191tt} z=E+i\eta, \qquad E\geq K\lambda_+, \qquad\eta=
EN^{-2/3}.
\end{equation}
With (\ref{fixE1E2}) and choosing $E_1=\lambda_-$ and $E_2=K\lambda_+$,
we have proved that there are at least $\varphi^{O(1)}$ eigenvalues
larger than $K\lambda_+$. Then by
definition,%
%
%e8.26 #&#
\begin{equation}
\label{190} \Im m^{(\mathbb T)} \leq\frac{C\eta}{E^2}+
\frac{\varphi^{C_\zeta}}{N\eta}, \qquad\bigl|\Re m^{(\mathbb
T)}\bigr|\leq CE^{-1}+
\frac{\varphi^{C_\zeta}}{N\eta}\leq O \bigl(E^{-1} \bigr)
\end{equation}
for any index set $\mathbb T$ with $|\mathbb T|=O(1)$. Now using the
large deviation lemma, as in (\ref{227pp}) and (\ref{eqnZibd}), we have
%
%
%e8.27 #&#
\begin{equation}
\label{190cg}  |Z_i|\leq|E| \biggl(E^{-1}N^{-1/2}+
\frac{\varphi^{C_\zeta}}{N\eta} \biggr), \qquad \bigl\langle \mathbf{x}_i,
\mathcal{G}^{(i,j)} \mathbf{x}_j \bigr\rangle\leq
E^{-1}N^{-1/2}+ \frac{\varphi^{C_\zeta}}{N\eta}.\hspace*{-40pt}
\end{equation}
First we estimate $G_{ii}$, with (\ref{eqnGii}), (\ref{eqnZi}) and
(\ref{eqndtrGmG}),
\[
| G_{ii} |= \bigl|1 - z-d-z d m^{(i)}(z)-Z_i
\bigr|^{-1}
\]
and
%
%
%e8.28 #&#
\begin{equation}
\label{192tt} \tfrac{1}2 E^{-1}\leq| G_{ii}
| \leq2 E^{-1},
\end{equation}
where we used (\ref{190}), (\ref{190cg}), $\eta=EN^{-2/3} $ and the
fact $K$ is large enough. Similarly for $G_{ij}$, from (\ref{eqnGij})
and (\ref{227pp}) it follows that%
%
%e8.29 #&#
\begin{equation}
\label{190gj} |G_{ij}|\leq E^{-1} \biggl(
\frac{\varphi^{C_\zeta}}{N\eta
}+E^{-1}N^{-1/2} \biggr).
\end{equation}
%
%and
% \begin{eqnarray}
% \end{eqnarray}
Furthermore with (\ref{eqnGijGijk}) and (\ref{eqnGij}),
\[
\bigl|m^{(i)}-m\bigr|=\frac{1}{N} \biggl|\sum_j
\frac{G_{ji}G_{ij}}{G_{ii}} \biggr|\leq E^{-1} \biggl|\frac{\varphi^{C_\zeta}}{N\eta}+E^{-1}N^{-1/2}
\biggr| ^2.
\]
Using these bounds,
\[
G_{ii}= \frac{1}{ 1 - z-d-z d m }+ O \bigl(m^{(i)}-m \bigr)+
\frac{Z_i}{ (1 - z-d-z d
m )^2 }+ E^{-3}O \bigl(Z_i^2 \bigr)
\]
and
%
%
%e8.30 #&#
\begin{eqnarray}
\label{192}
m&=&\frac{1}N\sum_i
G_{ii}
\nonumber\\[-10pt]\\[-10pt]
&=& \frac{1}{ 1 - z-d-z d m }+O \bigl(E^{-1} \bigr) \biggl(
\frac
{\varphi^{C_\zeta}}{N\eta}+ E^{-1}N^{-1/2} \biggr)^2+O
\bigl(E^{-2} [Z] \bigr).\hspace*{-40pt}\nonumber
\end{eqnarray}
Since $ |\Re(1 - z-d-z d m)| \geq|\Im(1 - z-d-z d m)| $,
%
%
%e8.31 #&#
\begin{equation}
\label{193} \Im\frac{1}{1 - z-d-z d m }\leq CE^{-2} \eta+
\frac{1}2 \Im m(z).
\end{equation}
Together with (\ref{192}) and (\ref{190cg}), with $\zeta$-high
probability,
%
%
%e8.32 #&#
\begin{eqnarray}
\label{194}
\Im m(z) &\leq& CE^{-2}\eta+E^{-1} \biggl(
\frac{\varphi^{C_\zeta}}{N\eta}+E^{-1}N^{-1/2} \biggr)
\nonumber\\[-8pt]\\[-8pt]
&=& \biggl(
\frac{N\eta^2}{E^2 } +\frac{\eta N^{1/2}}{E^2} +\frac{\varphi
^{C_\zeta}}{E } \biggr)
\frac{1}{N\eta}.\nonumber
\end{eqnarray}
If $E\geq N^{\varepsilon}$ for some $ \varepsilon>0$, with $\zeta
$-high probability, we have
%
%
%e8.33 #&#
\begin{equation}
\label{195} \Im m\ll\frac{1}{N\eta}.
\end{equation}
From the observation made in (\ref{qpn}), it follows that there are no
eigenvalues in the interval $[E-\eta, E+\eta]$ with $\zeta$-high
probability, or equivalently there are no eigenvalues larger than
$N^{\varepsilon}$ with $\zeta$-high probability.

Now, it only remains to prove (\ref{195}) for $K\lambda_+\leq E\leq
N^\varepsilon$. Using the above result, $\max_j\lambda_j\leq
N^\varepsilon$, with $\zeta$-high probability we have
\[
|G_{ii}| \geq N^{-2\varepsilon}.
\]
Therefore, applying (\ref{nnz}) and (\ref{Z1}) with $\mathcal
X=N^\varepsilon(N^{-1/2}+ \frac{\varphi^{C_\zeta}}{N\eta} )$, $\mathcal
Y=N^{2\varepsilon}$ and $p=N^\varepsilon$ and by using (\ref{191tt}),
(\ref{190gj}), (\ref{190cg}),
(\ref{192tt}), we have%
\[
\bigl|[Z]\bigr|\leq N^{C\varepsilon} \biggl(N^{-1/2}+ \frac{\varphi^{C_\zeta
}}{N\eta}
\biggr)^2.
\]
Inserting this in (\ref{192}), with (\ref{190}), (\ref{193}), we obtain
that the conclusion (\ref{195}) with $\zeta$-high probability for
$K\lambda_+\leq E\leq N^\varepsilon$. Again using (\ref{qpn}), we
deduce that there are no eigenvalues located in the interval
$[K\lambda_+, N^\varepsilon] $ with $\zeta$-high probability. Thus we
have achieved Step 3.

Step 4 of proof of (\ref{443}). Now we prove the last component of the
proof for (\ref{443}), that is, in the case of $d>1$ and thus $N>M$, we
have $ \lambda_M\geq\lambda_-/5$. As remarked earlier, it remains only
to prove that for some large $K$, the following bound holds with $\zeta
$-high probability,
%
%
%e8.34 #&#
\begin{equation}
\label{ttt} \lambda_M\geq\lambda_-/K.
\end{equation}
%
%{\mathbf Here is an explanation why the case $d>1$ and small $E$ is
%different from the others. In this case, there are many zero
%eigenvalues, therefore, $G$ or $m$ is not bounded by $O(1)$ if $E\ll
%1$. So here we need to discuss this case separately. }
Recall $\mathcal{G}=(XX^\dagger-z)^{-1}$. Let
%
%
%e8.35 #&#
\begin{equation}
\label{191ttn} z=E+i\eta, \qquad0\leq E\leq\lambda_-/K, \qquad \eta=
N^{-1/2-\varepsilon}
\end{equation}
for some small enough $\varepsilon>0$. Recall we have proved that among
$\lambda_i$, $i\leq M$, there are at least $\varphi^{O(1)}$ eigenvalues
less than $ \lambda_- $. Then for some
$C, c\geq0$%
%
%e8.36 #&#
\begin{equation}
\label{190n} \Im\frac{1}N \operatorname{Tr}\mathcal{G}(z)\leq C
\eta+\frac{\varphi^{C_\zeta}}{N\eta}, \qquad% \frac{1}{2
c%- \frac{\varphi^{C_\zeta}}{N\eta}
\leq\Re
\frac{1}N \operatorname{Tr}\mathcal{G}(z) \leq C.% + \frac{
\end{equation}
In the above, the term $\frac{\varphi^{C_\zeta}}{N\eta}$ is contributed
by these $\varphi^{O(1)}$ eigenvalues. Using Cauchy's interlacing
theorem of eigenvalues, it is easy to see that (\ref{190n}) also holds
for $\mathcal{G}^{(\mathbb T)}$ for $|\mathbb T|=O(1)$. Using the large
deviation lemma, with $\zeta$-high probability,%
%
%e8.37 #&#
\begin{eqnarray}
\label{ooo}
|Z_i|&\leq& |z| \biggl( N^{-1/2}+
\frac{\varphi^{C_\zeta}}{N\eta} \biggr)\leq|z| N^{-1/2+2\varepsilon},
\nonumber\\[-8pt]\\[-8pt]
\bigl\langle \mathbf{x}_i, \mathcal{G}^{(i,j)}\mathbf{x}_j
\bigr\rangle &\leq&  N^{-1/2}+ \frac{\varphi
^{C_\zeta}}{N\eta}\leq N^{-1/2+2\varepsilon}.\nonumber
\end{eqnarray}
First using (\ref{eqnGii}), we obtain,%
%
%e8.38 #&#
\begin{equation}
\label{Gtb} G_{ii} = \biggl(- z -z d\frac{1}N
\operatorname{Tr} \mathcal{G}^{(i)}(z)-Z_i
\biggr)^{-1}.
\end{equation}
Then using (\ref{190n}) we deduce that with $\zeta$-high probability,%
%
%e8.39 #&#
\begin{equation}
\label{wgl} c|z|^{-1}\leq| G_{ii} |\leq
C|z|^{-1}.
\end{equation}
Similarly from (\ref{eqnGij}), it follows that with $\zeta$-high
probability,
%
%
%e8.40 #&#
\begin{equation}
\label{odo} |G_{ij}|\leq|z|^{-1}N^{-1/2+C\varepsilon}.
\end{equation}
We have%
\[
\operatorname{Tr}G^{(i)}(z)-\operatorname{Tr}\mathcal{G}^{(i)}(z)=
\frac{M-N+1 }{z} = \operatorname{Tr}G (z)-\operatorname{Tr}\mathcal{G}(z)+
\frac{1}z.
\]
Together with (\ref{Gtb}),
\[
G_{ii} = \biggl(- z -z d\frac{1}N\operatorname{Tr}
\mathcal{G}(z)-z d \biggl(m^{(i)}-m-\frac
{1}{Nz}
\biggr)-Z_i \biggr)^{-1}.
\]
Using the bound [see (\ref{190n})],
\[
c|z|\leq\biggl|- z -z d\frac{1}N\operatorname{Tr}\mathcal{G}(z)\biggr|\leq C|z|,
\]
equation (\ref{ooo}) and $|m^{(i)}-m|\leq(N\eta)^{-1}$, we take the
average of $G_{ii}$ and use Taylor expansion to obtain [similar to
(\ref{192})]
%
%
%e8.41 #&#
\begin{eqnarray}
\label{gbmm} m&=& \frac{1}{ 1 - z-d-z d m (z) }+\delta,\nonumber
\\
\delta&:=&|z|^{-1}O \biggl(\frac{1}N\sum
_i \bigl(m^{(i)}-m \bigr)-(Nz)^{-1}
\biggr)
\\
&&{} +|z|^{-2}O \bigl([Z] \bigr)+|z|^{-1}O
\bigl(N^{-1+C\varepsilon} \bigr)\nonumber
\end{eqnarray}
%
%we have
% |G_{ii}-G_{jj}|\leq|z|^{-1}\varphi^{C_\zeta}N^{-1/2+\varepsilon},
% \end{eqnarray}
% For $G_{ij}$, with \eqref{eqnGij}
% G_{ij}\leq|z|^{-1}\varphi^{C_\zeta}N^{-1/2+\varepsilon}
% \end{eqnarray}
with $\zeta$-high probability. Similarly, by estimating the difference
$G_{ii}-G_{jj}$, we
have%
%
%e8.42 #&#
\begin{equation}
\label{gwjs} |G_{ii}-m|\leq|z|^{-1}N^{-1/2+C\varepsilon}
\end{equation}
with $\zeta$-high probability. First for the term $ m^{(i)}-m$ in
(\ref{gbmm}), using (\ref{eqnGijGijk}), (\ref{wgl}) and (\ref{gwjs}),
we have
\[
m^{(i)}-m=\frac{-1}{N} \sum_j
\frac{ G_{ji}G_{ij}}{G_{ii}} =\frac
{-1}{N}\frac{G^2_{ii}}{G_{ii}} =\frac{-1}{N}
\frac{G^2_{ii}}{m} +O \bigl(|z| N^{-3/2+C\varepsilon} \bigr) \bigl|
\bigl(G^2 \bigr)_{ii} \bigr|. % =\frac{-1}{N}
\]
%
%Furthermore, with $\tr G=\tr\mG-z^{-1}(N-M)$ and \eqref{190n}, small
%enough $z\leq O(1)$, we have
% \begin{eqnarray}
% m^{(i)}-m
%=\frac{-1}{N}\frac{[G^2]_{ii}}{(-z^{-1})(1-d^{-1})}+O(|z|^{-1})
%E^{-1}\left(\frac{\varphi^{C_\zeta}}{N\eta}+E^{-1}N^{-1/2}\right)^2
% \end{eqnarray}
Averaging $ m^{(i)}-m$, we obtain that%
%
%e8.43 #&#
\begin{equation}
\label{dhdh} \frac{1}{N}\sum_i \bigl(
m^{(i)}-m \bigr) =\frac{-1 }{N^2}\frac{\operatorname{Tr}
[G^2] }{m}+ O \bigl(|z|
N^{-5/2+C\varepsilon} \bigr)\sum_i \bigl|
\bigl(G^2 \bigr)_{ii} \bigr|. % +O(N^{-1+C\varepsilon})%\leq
%E^{-1}\left(\frac{\varphi^{C_\zeta}}{N
\end{equation}
Since we have proved that there are at least $\varphi^{O(1)}$ nonzero
eigenvalues less than $ 0.9\lambda_- $, then under (\ref{191ttn}), with
$\zeta$-high probability
%
%
%e8.44 #&#
\begin{equation}
\label{jxf} \operatorname{Tr} \bigl[G^2 \bigr]=\sum
_{\alpha}\frac{1}{(\lambda_\alpha-z)^2}=\frac{N-M}{z^2}+O \bigl(\varphi
^{C_\zeta} \bigr)\eta^{-2}+O(N).
\end{equation}
These three terms come from zero eigenvalues, small eigenvalues (which
are less than $0.9\lambda_-$) and the eigenvalues in the interval
$[\lambda_-, \lambda_+]$, respectively. We denote the three terms
appearing on the RHS of (\ref{jxf}) as $T_0$, $T_s$ and $T_n$,
respectively. Similarly, we have [note that here $z \leq O(1)$
is small enough]
%
%e8.45 #&#
\begin{equation}
\label{jxf2}\qquad Nm=\operatorname{Tr}[G ] =\frac{N-M}{-z }+O \bigl(
\varphi^{C_\zeta} \bigr)\eta^{-1}+O(N) =\frac{N-M}{-z
}
\bigl(1+O(z) \bigr)
\end{equation}
with $\zeta$-high probability and%
\[
\bigl| \bigl(G^2 \bigr)_{ii} \bigr|\leq \biggl|\sum
_{\alpha}\frac{|u_\alpha
(i)|^2}{(\lambda_\alpha-z)^2} \biggr| \leq C\sum
_{\alpha\in T_0}\frac{|u_\alpha(i)|^2}{|z^2|} +C\sum_{\alpha\in T_s}
\frac{|u_\alpha(i)|^2}{\eta^2}+ C\sum_{\alpha\in T_n} \bigl|u_\alpha(i)\bigr|^2.
\]
The last bound implies that
\[
\sum_i \bigl| \bigl(G^2
\bigr)_{ii} \bigr|\leq C \frac{N
}{|z|^2}+O \bigl(
\varphi^{C_\zeta} \bigr)\eta^{-2}+O(N).
\]
Together
with (\ref{dhdh}), we have%
%
%e8.46 #&#
\begin{equation}
\label{tnf} \frac{1}{N}\sum_i \bigl(
m^{(i)}-m \bigr) =\frac{-1 }{N^2}\frac{\operatorname{Tr}
(G^2) }{m}+ O
\bigl(|z|^{-1} N^{-3/2+C\varepsilon} \bigr).
\end{equation}
Dividing (\ref{jxf}) by $Nm$ [see (\ref{jxf2})], for $|z|$ small
enough, we have
%
%
%e8.47 #&#
\begin{equation}
\label{tnf2} \frac{\operatorname{Tr}(G^2) }{Nm}= \frac
{-1}{z}+O \bigl( z
N^{ 2\varepsilon} \bigr) +O(1).
\end{equation}
Recall $\delta$ from (\ref{gbmm}). Now combining (\ref{tnf}) and
(\ref{tnf2}) with (\ref{gbmm}), we obtain
%
%
%e8.48 #&#
\begin{equation}
\label{nede} \delta\leq O \bigl( \bigl|z^{-2}\bigr|N^{-3/2+C\varepsilon
}+\bigl|z^{-1}\bigr|N^{-1+C\varepsilon}
\bigr)+|z|^{-2}O \bigl([Z] \bigr).
\end{equation}
Now we apply Lemma~\ref{lemaZ2} (with $\mathcal X=N^{-1/2+C\varepsilon
}$, $\mathcal Y=C|z |$ and $p=N^\varepsilon$) to estimate $[Z]$. Using
Lemma~\ref{lemaZ2}, (\ref{ooo}), (\ref{wgl}) and (\ref{odo}), we get
\[
|z|^{-2} \bigl|[Z]\bigr|\leq|z|^{-1} N^{-1+C\varepsilon}.
\]
Combining the above with (\ref{nede}) gives%
%
%e8.49 #&#
\begin{equation}
\label{llz} \delta\leq O \bigl( \bigl|z^{-2}\bigr|N^{-3/2+C\varepsilon
}+\bigl|z^{-1}\bigr|N^{-1+C\varepsilon}
\bigr).
\end{equation}
Using (\ref{gbmm}) and the definition of $m_c$,
\[
m-m_c= \frac{1}{ 1 - z-d-z d m (z) }-\frac{1}{ 1 - z-d-z d m_c(z)
}+\delta,
\]
which implies that
\[
\biggl( \frac{zd }{ (1 - z-d-z d m (z) )(1 - z-d-z d m_c
(z))}-1 \biggr) (m-m_c)=\delta.
\]
As above, we have $c|z|\leq|1 - z-d-z d m (z)|$, $|1 - z-d-z d
m_c(z)|\leq C|z|$ for all $|z| \leq\varepsilon_0$ for a constant
$\varepsilon_0$ independent of $N$. Therefore, we have
\[
|m-m_c| \leq|z \delta|.
\]
Using (\ref{llz}), we have%
\[
|m-m_c| \leq O \bigl( \bigl|z^{-1}\bigr|N^{-3/2+C\varepsilon
}+N^{-1+C\varepsilon
}
\bigr) \ll(N\eta)^{-1}.
\]
Furthermore, it is easy to prove
that%
\[
\Im \biggl(m_c-\frac{1-d^{-1}}{-z } \biggr)=O( \eta)\ll(N
\eta)^{-1}.
\]
Together with $\operatorname{Tr}G=\operatorname{Tr}
\mathcal{G}-z^{-1}(N-M)$, we obtain%
\[
\Im\operatorname{Tr}\mathcal{G}(z)\ll\frac{1}{\eta}
\]
with $\zeta$-high probability. As in (\ref{qpn}), we have $\lambda
_\alpha\notin[E-\eta, E+\eta]$ for $E\in[0, \lambda_-/K]$ with large
enough $K=O(1)$ obtaining (\ref{ttt}). This completes step 4 and we
have thus proved (\ref{443}).
\end{pf*}

Thus we have verified (\ref{Lambdafinal}), (\ref{Lambdaofinal}),
(\ref{443}) and (\ref{4444}) and have finished the proof of Theorem
\ref{451}.
%s8.2 #&#
\subsection{Proof of Theorem \texorpdfstring{\protect\ref{452}}{3.3}}
We confirm formulas (\ref{resrig2}) and (\ref{resrig}) separately.
\begin{pf*}{Proof of (\ref{resrig2})}
Recall (\ref{fixE1E2}) and the fact that there is no eigenvalue in $(
0, \lambda_-/4]\cup[4\lambda_+, +\infty]$. We deduce that
%
%
%e8.50 #&#
\begin{equation}
\max_{E\in\mathbb{R}} \bigl| {\mathfrak n}(E) - n_c(E) \bigr| \le
\frac{C (\log N)\varphi^{C_\zeta}}{N} \label{fixE}
\end{equation}
holds with $\zeta$-high probability. The supremum over $E $ is
a~standard argument for extremely small events and we omit the details.
\end{pf*}

Now we give the proof of (\ref{resrig}).

\begin{pf*}{Proof of (\ref{resrig})}
The proof is very similar to the one for generalized Wigner matrix
obtained in formula (2.25) of \cite{EYYrigid}. For the reader's sake,
we reproduce that argument below. By symmetry, we assume that $ 1\le j
\le N/2 $ and set $E = \gamma_j$, $E'= \lambda_j$. Also $t_N= (\log
N)\varphi^{C_\zeta}$ for compactness of notation. From
(\ref{fixE}) we have%
%
%e8.51 #&#
\begin{equation}
n_c(E) = \frak{n} \bigl(E' \bigr) = n_c
\bigl(E' \bigr) + O(t_N/N). \label{nnn}
\end{equation}
Clearly $E\geq\lambda_C:=( \lambda_++3\lambda_-)/4$, and using
(\ref{fixE}) we see that $E'\ge\lambda_C$ also holds with $\zeta $-high
probability. First, using (\ref{443}) and
%
%
%e8.52 #&#
\begin{equation}
\label{nscxs} n_c(x) \sim(\lambda_+-x)^{3/2}\qquad
\mbox{for } \lambda_C\leq x\leq\lambda_+,
\end{equation}
or equivalently,
\[
n_c(E)=n_c(\gamma_j) = \frac{j}{N}
\sim( \lambda_+-E)^{3/2},
\]
we know that (\ref{resrig}) holds (possibly with a~larger constant) if
\[
E,E'\geq\lambda_+-t_N N^{-2/3}.
\]
%
%The correct power $(\log N)^L$ can be restored by increasing $L$
%(hence $A_0$)
%and decreasing $\phi$, as before.

Hence, we can assume that one of $E$ and $E'$ is in the interval
$[\lambda_C,\lambda_+-t_N N^{-2/3}]$. With (\ref{nscxs}), this
assumption implies that at least one of $n_c(E)$ and $n_c(E')$ is
larger than $t_N^{3/2}/ N$. Inserting this information into
(\ref{nnn}), we obtain that both $n_c(E)$ and $n_c(E')$ are positive
and
\[
n_c(E) = n_c \bigl(E' \bigr) \bigl[ 1 + O
\bigl(t_N^{-1/2} \bigr) \bigr]
\]
and in particular, $\lambda_+-E\sim\lambda_+-E'$. Using the fact that
$n_c'(x)\sim(\lambda_+-x)^{1/2}$ for $\lambda_C\le x\le\lambda_+$, we
obtain that $n_c'(E) \sim n_c'(E')$, and in fact $n_c'(E)$ is
comparable with $n_c'(E'')$ for any $E''$ between $E$ and $E'$. Then
with Taylor's expansion, we have
%
%
%e8.53 #&#
\begin{equation}
\label{88} \bigl|n_c \bigl(E' \bigr)-n_c(E)\bigr|
\le C \bigl| n_c'(E)\bigr| \bigl|E'-E\bigr|.
\end{equation}
Since $n_c'(E) = \varrho_c(E) \sim\sqrt\kappa$ and $n_c(E) \sim
\kappa^{3/2}$, moreover, by $E=\gamma_j$ we also have $n_c(E) =j/N$, we
obtain from (\ref{nnn}) and (\ref{88}) that
\[
\bigl|E'-E\bigr| \le\frac{C| n_c(E')- n_c(E)|}{n_c'(E)}\le\frac{Ct_N} { N
n_c'(E) } \le
\frac{C t_N} { N (n_c(E))^{1/3} } \le\frac{C t_N} { N^{2/3} j^{1/3} },
\]
which proves (\ref{resrig}), again with a~larger constant.
\end{pf*}
We have proved (\ref{resrig}) and (\ref{resrig2}) and the proof of
Theorem~\ref{452} is complete.

%s9 #&#
\section*{Acknowledgments}
The authors would like to thank Horng-Tzer Yau, Antti Knowles, Lazlo
Erd\"os, Paul Bourgade and Alain Pajor for very useful discussions and
help. The authors also thank the editorial panel for useful feedback
which significantly improved the presentation. Special thanks to Steve
Finch for a~careful reading.

% zodis "Acknowledgments" paliekamas pagal autoriu

%suskaldyti doi

% imsref loaded by linak, 2013-11-21 10:58:46
%

\printaddresses

\end{document}